\documentclass[review,onefignum,onetabnum]{siamart190516}

\usepackage{lipsum}
\usepackage{amsmath,amssymb,amsfonts,latexsym,stmaryrd}
\usepackage{graphicx,float}
\usepackage{mathrsfs} 
\usepackage{color}
\usepackage{epstopdf}
\usepackage{algorithmic}
\usepackage{multirow}
\ifpdf
  \DeclareGraphicsExtensions{.eps,.pdf,.png,.jpg}
\else
  \DeclareGraphicsExtensions{.eps}
\fi

\def\O{\Omega}

\def\l{\lambda}

\def\E{E}
\def\VK{V^{\E}}

\renewcommand\sp{\mathop{\mathrm{Sp}}\nolimits}

\usepackage{booktabs}
\usepackage{float}

\newcommand\bv{\boldsymbol{v}}

\def\Vh{V_h}

\def\PiK{\Pi^{\E}}

\def\CE{\mathcal{E}}
\renewcommand\sp{\mathop{\mathrm{sp}}\nolimits}
\def\CB{\mathscr{B}}

\def\CT{\mathcal{T}}

\def\O{\Omega}
\def\P{\mathbb{P}}
\def\PiK{\Pi^{\nabla,E}}

\def\Pio{\Pi^{E}}
\def\Vh{V_h}
\def\VK{V^{E}_h}
\def\WK{\widetilde{V}_h^E}
\def\l{\lambda}

\def\CT{{\mathcal T}}

\newcommand\R{\mathbb{R}}

\renewcommand\L{\mathrm{L}}

\renewcommand\O{\Omega}

\renewcommand\sp{\mathop{\mathrm{sp}}\nolimits}

\newcommand{\vertiii}[1]{{\left\vert\kern-0.25ex\left\vert\kern-0.25ex\left\vert #1 
    \right\vert\kern-0.25ex\right\vert\kern-0.25ex\right\vert}}

\newsiamremark{remark}{Remark}
\newsiamremark{hypothesis}{Hypothesis}
\crefname{hypothesis}{Hypothesis}{Hypotheses}
\newsiamthm{claim}{Claim}

\headers{ VEM for the convection-diffusion eigenvalue problem}{D. Amigo, F. Lepe and  G. Rivera}

\title{A priori and a posteriori error analysis for a VEM discretization of the convection-diffusion eigenvalue problem\thanks{Submitted to the editors DATE.
\funding{DA and FL were partially supported by DIUBB through project 2120173 GI/C Universidad del B\'io-B\'io. FL was partially supported by 
ANID-Chile through FONDECYT project 11200529 (Chile). GR was supported by Universidad de Los Lagos Regular R02/21 and ANID-Chile through FONDECYT project 1231619 (Chile).
}}}

\author{Danilo Amigo\thanks{GIMNAP-Departamento de Matem\'atica, Universidad del B\'io-B\'io, Casilla 5-C, Concepci\'on, Chile. 
\texttt{danilo.amigo2101@alumnos.ubiobio.cl}.}
\and
Felipe Lepe\thanks{GIMNAP-Departamento de Matem\'atica, Universidad del B\'io-B\'io, Casilla 5-C, Concepci\'on, Chile. 
\texttt{flepe@ubiobio.cl}.}
\and
Gonzalo Rivera\thanks{Departamento de Ciencias Exactas, Universidad de Los Lagos, Osorno, Chile.
\texttt{gonzalo.rivera@ulagos.cl}}}

\usepackage{amsopn}

\ifpdf
\hypersetup{
  pdftitle={A priori and a posteriori VEM discretization of the convection-diffusion eigenvalue problem},
  pdfauthor={Danilo Amigo, Felipe Lepe, Gonzalo Rivera}
}
\fi

\begin{document}

\nolinenumbers
\maketitle

\begin{abstract}
In this paper we propose and analyze a  virtual element method for the two dimensional non-symmetric diffusion-convection eigenvalue problem in order to derive a priori and  a posteriori error estimates. Under the classic assumptions of the meshes, and with the aid of the classic theory of compact operators, we prove error estimates for the eigenvalues and eigenfunctions. Also, we develop an a posteriori error estimator  which, in one hand, results to be  reliable and on the other, with standard bubble functions arguments, also results to be efficient. We test our method on domains where the complex eigenfunctions are not sufficiently regular, in order to assess the performance of the estimator that we compare with the uniform refinement given by the a priori analysis. \end{abstract}



\section{Introduction}
\label{sec:introduccion}
The virtual element method (VEM) has proved through the years the efficiency and versatility on the approximation of the solutions of partial differential equations (PDEs). This has led to a number of works where  this popularity of the VEM is confirmed. We mention the book \cite{ABM2022} where recent advances on the development of the VEM are documented. In particular, the VEM has resulted to be a
good numerical strategy to approximate the eigenfunctions and eigenvalues of spectral problems arising on the continuum mechanics and related problems, where conforming and no-conforming virtual methods have emerged  such as \cite{MR3867390, MR3340705,MR3895875} and the references therein. These references show the accuracy of the VEM when the spectrum is computed and the easy way to handle with the spurious eigenvalues, where the methods result to be precisely spurious free. In this sense, the literature 
related to the a priori analysis for eigenvalue problems is very rich and the research on this  is in constant progress. However, the literature available on the analysis of a posteriori error estimators for eigenvalue problems using VEM is, for the best of the author's knowledge,  a topic that has not been sufficiently studied. On this subject, we recall the works  \cite{adak2022vem,MR4550402,MR4050542,MR3715326,MR4497827} 
 where for different eigenvalue problems, adaptive strategies have been considered. Clearly the literature for a posteriori error analysis of eigenvalue problems must be taken into consideration.

It is well known that there exist different manners for which the eigenfunctions of the eigenvalue problems may result to be non sufficiently smooth. This can be due to the presence of certain parameters that may affect the regularity of the eigenfunctions (we can think for instance on the elasticity eigenvalue problems and the influence of the Poisson ratio when it is close to $1/2$) or the classical geometric singularities on the domains, particularly the non-convexity of some domains or the presence of fractures.  This is precisely the motivation for the development of reliable and efficient a posteriori error indicators, in order to recover the optimal order of convergence through adaptive refinements in the regions where the singularities arise. In particular, our contribution is focused on a priori and a posteriori error analysis for non-symmetric eigenvalue problems based on the VEM. This nature of eigenvalue problems demand a more sharp analysis, since not only the primal eigenvalue problem deserves attention, but also the dual eigenvalue problem. In \cite{MR3212379} there is a complete treatment, from the finite element point of view, to analyze a posteriori error estimators for a non-symmetric eigenvalue problem. Regarding non-symmetric spectral problems using the VEM we refer to the following contributions \cite{MR3895875,Mora2021A2425,Mora2021,MR4497827} where once again the VEM confirms its accuracy on the approximation of the eigenvalues and eigenfunctions.

In our paper we consider the diffusion-convection eigenvalue problem which reads as follows: For an open and  bounded  domain $\Omega \subset \mathbb{R}^2$ with polygonal boundary $\partial\Omega$, find $u \neq 0$ such that 
\begin{equation}
\label{eq:state_equation}
\nabla \cdot (-\kappa(\mathbf{x}) \nabla u)  + \mathbf{\vartheta}(\mathbf{x}) \cdot \nabla u = \lambda u \quad \textrm{in}~\Omega, \qquad u = 0 \quad \textrm{on}~\partial\Omega,
\end{equation}
where  $\kappa$ is a  smooth function $\Omega\rightarrow\mathbb{R}$ with $\kappa(\mathbf{x})\geq\kappa_0>0$ for all $\mathbf{x}\in\Omega$ and  $\mathbf{\vartheta}$ is a sufficiently smooth vector-valued function  $\Omega\rightarrow\mathbb{R}^2$. Clearly this problem is non-symmetric and leads to complex eigenvalues and eigenfunctions. This problem has been already analyzed in \cite{MR3133493} with an a posteriori error estimator constructed by means of the finite element method. Precisely the estimator, which is of the residual type,  results to be reliable and efficient as is expected. The difference on our contributions  lie on the fact that an inf-sup condition is a key point to perform the analysis, which in \cite{MR3133493}  is no needed, since under suitable norms (an energy norm considering only the real part of the complex sesquilinear forms) all the estimates are derived. We directly use the $H^1$ and $L^2$ norms defined on complex fields as in \cite{lepe2023vem} for a virtual element method allowing for small edges, where a priori error estimates are obtained. However, the inf-sup condition that we need is not the one proved in \cite{lepe2023vem} (since this reference uses other type of VEM), but is the one already well established in \cite{beiraosec} for a classic VEM.  Clearly these highlights are important to take into consideration for the VEM that we analyze, where now the difficulties arise on the fact that the virtual element solution of the spectral problem is not  computable and the standard projections, virtual interpolators, and approximation properties 
now play a role on the estimates, leading to a different algebraic and numerical difficulties on the analysis. An example of this is available on the recent paper \cite{MR4497827} where a priori and a posteriori error estimators for a non-symmetric eigenvalue problem are derived. It is also necessary to differentiate the a priori analysis presented in this paper with the one presented in \cite{lepe2023vem}. In order to obtain a computable primal and dual residual estimator, it is necessary to write the sesquilinear form of the right-hand side depending on the classical stability bilinear form for the inner product in $L^2$ (see \cite{AABMR13}). However, this modification to the right-hand side bilinear form automatically implies that one cannot define a solution operator $T_h$ from the continuous space into the discrete space, since this operator is not computable. Therefore, the natural alternative is to define the operator $T_h$ from the virtual space with itself, which would imply using the theory of non-compact operators (see \cite{DNR1,DNR2}). However, inspired by \cite{MR4050542} we will define a suitable operator to be able to use the theory of and just prove the convergence in norm of operators.

\subsection{Organization of the paper}
The outline of the paper is the following: In Section \ref{sec:model} we present the model problem under consideration, presenting the functional space in which our work is supported, sesquilinear forms, well posedness of the problem, the continuous solution operators and the regularity of the primal and dual solutions. Section \ref{sec:virtual} is dedicated to introduce  the virtual element method, where the assumptions on the meshes are  presented, the virtual element spaces with the corresponding degrees of freedom, projections and discrete sesquilinear forms which are needed to present the discrete counterpart of the continuous eigenvalue problem introduced in the previous section.  The core of our paper begins in Section \ref{sec:a_priori} where we study the discrete discrete eigenvalue problem. More precisely, we derive the necessary error estimates for the the eigenvalues and eigenfunctions for the primal and dual eigenvalue problems. In Section \ref{sec:a_post} we introduce and analyze the a posteriori error estimators of our interest which are considered for both, the primal and dual eigenvalue problems. We prove on this section the reliability and efficiency of the proposed estimators. Finally in Section \ref{sec:numerics} we report a series of numerical tests where we are able to assess the performance of the proposed estimator, implementing different polygonal meshes for non-convex domains.
\noindent\section{The variational problem}
\label{sec:model}
long our paper the relation $\texttt{a} \lesssim \texttt{b}$ indicates that $\texttt{a} \leq C \texttt{b}$, with a positive constant $C$ which is independent of $\texttt{a}$, $\texttt{b}$. Now, since the eigenvalue problem that we are considering is non-symmetric, it is necessary to introduce complex Hilbert spaces in order to perform the analysis. Let us consider the space $H_0^1(\O,\mathbb{C})$ which corresponds to the classic $H^1$ space but defined on a complex field. We endow this space with the following inner product
\begin{equation*}
(v,w)_{1,\O}=\int_{\O}v\overline{w}+\int_{\O}\nabla v\cdot\nabla\overline{w}\quad\forall v,w\in H_0^1(\O,\mathbb{C}),
\end{equation*}
where $\overline{w}$ denotes the conjugated of $w$. Let us remark that the norm induced by the complex inner product defined above is the standard for the space $H^1$.

In order to simplify the presentation of the material, let us define $V:=H_0^{1}(\O,\mathbb{C})$. Let us begin with the variational formulation of problem  \eqref{eq:state_equation}: Find $\lambda\in\mathbb{C}$ and $0\neq u\in V$
such that  
\begin{equation}
\label{eq:spectral1}
\CB(u,v)=\lambda  c(u,v)  \quad \forall v \in V,
\end{equation}
where $\CB(\cdot,\cdot)$ is the sesquilinear form defined by $\CB(w,v):=a(w,v) + b(w,v)$ for all $w,v\in V$
and $c(\cdot,\cdot)$ is the sesquilinear form defined by 
$c(w,v):=(w, v)_{0,\O}$, whereas $a(\cdot,\cdot)$ and $b(\cdot,\cdot)$ are the sesquilinear forms defined as follows
\begin{equation}\label{eq:ab}
\begin{split}
a: V\times V\longrightarrow \mathbb{C}; \quad a(w,v) := \int_{\Omega}\kappa(\mathbf{x}) \nabla w\cdot\nabla \bar{v}, \quad \forall w,v \in V, \\
b: V\times V \longrightarrow \mathbb{C}; \quad b(w,v) := \int_{\Omega}( \mathbf{\vartheta}(\mathbf{x})\cdot \nabla w)\bar{v}, \quad \forall w,v \in V.
\end{split}
\end{equation}

The assumptions on the coefficients $\kappa$ and $\vartheta$ lead us to the correct definition of the sesquilinear forms $a(\cdot,\cdot)$ and $b(\cdot,\cdot)$ and hence, the continuity of $\CB(\cdot,\cdot)$, i.e, there exists a constant $M_1>0$ such that $\mathscr{B}(w,v)\leq M_1\|w\|_{1,\O}\|v\|_{1,\O}$ for all $v,w\in V$ with $M_{1} := \max\{\|\kappa\|_{\infty,\O},\|\vartheta\|_{\infty,\O}\}$.
Also, under the assumption that $\vartheta(\mathbf{x})$ is divergence-free, the following condition holds (see \cite[Equation (3.5)]{beiraosec})
\begin{equation}
\label{eq:inf-supB}
\displaystyle\sup_{v\in H_0^1(\Omega)}\frac{\mathscr{B}(w,v)}{\|v\|_{1,\O}}\geq \beta\|w\|_{1,\O}\quad \forall w\in V,
\end{equation}
where $\beta>0$ is a constant independent of $v$. This allows us to introduce the solution operator $T: V \longrightarrow V$, which is defined for a source $f\in V$ by $Tf := \widetilde{u}$ where $\widetilde{u} \in V$ is the solution of the following source problem 
\begin{equation*}
\CB(\widetilde{u},v)=c(f,v)  \quad \forall v \in V.
\end{equation*}
We remark that $T$ is well defined. On the other hand, we observe that $(\lambda,u) \in \mathbb{C} \times V$ is solution of \eqref{eq:spectral1} if and only if $(\mu,u) \in \mathbb{C} \times V$ is eigenpair of $T$, with $\mu = 1/\lambda$. Finally, since the bilinear form $\CB(\cdot,\cdot)$ is non-symmetric, the operator $T$ is not selfadjoint. This requires to consider the dual eigenvalue problem. Let us denote by $V^*$ the dual space of $V$. The dual eigenvalue problem reads as follows: Find $\lambda^{*} \in \mathbb{C}$ and $0 \neq u^{*} \in V^*$ such that 
\begin{equation}\label{eq:dual_problem}
\CB(v,u^{*}) = \overline{\lambda^{*}} c(v,u^{*}) \quad \forall v \in V^*,
\end{equation}
where, using integration by parts and the fact that $\vartheta(\mathbf{x})$ is divergence free, we have
\begin{equation*}
\CB(v,u^{*}) = a(u^{*},v) - b(u^{*},v), \quad \forall v \in V^*. 
\end{equation*}
Then, we define the dual solution operator $T^{*}: V^{*} \longrightarrow V^{*}$, which is defined for a source $f^{*}\in V^{*}$ by $T^{*}f^{*} := \widetilde{u}^{*}$ where $\widetilde{u}^{*} \in V^{*}$ is the solution of the following source problem 
\begin{equation*}
\CB(v,\widetilde{u}^{*})=c(v,f^{*})  \quad \forall v \in V.
\end{equation*}

Now, for the implementation of the virtual element method of our interest, the regularity of  \eqref{eq:spectral1} is a key ingredient in order to
obtain approximation properties. In fact, there exists $s > 0$, depending on $\O$, such that the solution  of \eqref{eq:spectral1} satisfies $u \in H^{1+s}(\O,\mathbb{C})$ and
\begin{equation}
\label{eq:regularity}
\|u\|_{1+s,\Omega}\lesssim \|u\|_{0,\Omega}.
\end{equation}

Finally, due to the compact inclusion $H^{1+s}(\O,\mathbb{C})$ onto $V$, we deduce that $T$ is compact. Therefore, the following spectral characterization of $T$ holds.
\begin{lemma}[Spectral characterization of $T$]
The spectrum of $T$ is such that $\sp(T)=\{0\}\cup\{\mu_k\}_{k\in\mathbb{N}}$, where $\{\mu_k\}_{k\in\mathbb{N}}$ is a sequence of complex eigenvalues that converge to zero, according to their respective multiplicities.
\end{lemma}

For the dual eigenvalue problem, an additional regularity for the eigenfunctions is also needed. In particular, there exists $s^*>0$ such that for $u^*\in V^*$ solution of \eqref{eq:dual_problem}, the following estimate holds
\begin{equation}
\label{eq:regularity_dual}
\|u^*\|_{1+s^*,\Omega}\lesssim \|u^*\|_{0,\Omega}.
\end{equation}
Finally the spectral characterization of $T^*$ is given as follows.
\begin{corollary}[Spectral characterization of $T^*$]
The spectrum of $T^*$ is such that $\sp(T^*)=\{0\}\cup\{\mu_k^*\}_{k\in\mathbb{N}}$, where $\{\mu_k^*\}_{k\in\mathbb{N}}$ is a sequence of complex eigenvalues that converge to zero, according to their respective multiplicities.
\end{corollary}

It is important to take into account the following:  if  $\mu$ is an eigenvalue of $T$ with multiplicity $m$ and  $\mu^{*}$ is an eigenvalue of $T^{*}$ with the same multiplicity, then $\mu=\overline{\mu^{*}}$. Moreover,  for the discrete solution operators, if $\mu_h$ is an eigenvalue of $T_h$ and $\mu_h^{*}$ is an eigenvalue of $T_h^{*}$, then $\mu_h=\overline{\mu_h^{*}}$.

\section{The virtual element method}
\label{sec:virtual}

In this section we briefly review the  virtual element method that for the system \eqref{eq:spectral1}. First we recall the mesh construction and the assumptions considered in \cite{BBCMMR2013} for the virtual element
method. 
Let $\left\{\CT_h\right\}_h$ be a sequence of decompositions of $\Omega$ into polygons, $E$. Let us denote by $h_E$  the diameter of the element $E$ and $h$ the maximum of the diameters of all the elements of the mesh, i.e., $h:=\max_{E\in\Omega}h_E$.  Moreover, for simplicity in what follows we assume that $\kappa $ is  piecewise constant with  respect to the decomposition $\mathcal{T}_h$, i.e., it is  piecewise constant for all $E\in \mathcal{T}_h$.

 For the analysis of the VEM, we will make as in \cite{BBCMMR2013} the following
assumptions:
\begin{itemize}
\item \textbf{A1.} There exists $\rho > 0$ such that, for all meshes
$\CT_h$, each polygon $E\in\CT_h$ is star-shaped with respect to a ball
of radius greater than or equal to $\rho h_{E}$.
\item \textbf{A2.} The distance between any two vertexes of $E$ is $\geq Ch_{E}$, where $C$ is a positive constant.
\end{itemize}

For any simple polygon $ E$ we define 
\begin{align*}
	\widetilde{V}_h^E:=\{v_h\in H^{1}(E,\mathbb{C}):\Delta v_h \in \mathbb{P}_1(E),
	v_h|_{\partial E}\in C^0(\partial E), v_h|_{e}\in \mathbb{P}_1(e) \  \forall e \in \partial E  \}.
\end{align*}

Now, in order to choose the degrees of freedom for $\widetilde{V}_h^E$ we define
\begin{itemize}
	\item $\mathcal{V}_E^h$: the value of $w_{h}$ at each vertex of $E$,
\end{itemize}
as a set of linear operators from $\widetilde{V}_h^E$ into $\R$. In \cite{AABMR13} it was established that $\mathcal{V}_E^h$ constitutes a set of degrees of freedom for the space $\widetilde{V}_h^E$.

 On the other hand, we define  the projector $\PiK:\ \WK\longrightarrow\P_1(E)\subseteq\WK$ for
each $v_{h}\in\WK$ as the solution of 
\begin{equation*}
\int_E (\nabla\PiK v_{h}-\nabla v_{h})\cdot\nabla q=0
\quad\forall q\in\P_1(E),\qquad
\overline{\PiK v_{h}}=\overline{v_{h}},
\end{equation*}
where  for any sufficiently regular
function $v$, we set $\overline{v}:=\vert\partial E\vert^{-1}(v,1)_{0,\partial E}$.
%
We observe that the term $\PiK v_{h}$ is well defined and computable from the degrees of freedom  of $v$ given by $\mathcal{V}_E^h$, and in addition the projector $\PiK$ satisfies the identity  $\PiK(\P_{1}(E))=\P_{1}(E)$ (see for instance \cite{AABMR13}).

We are now in position  to introduce our local virtual space
\begin{equation*}\label{Vk}
\VK
:=\left\{v_{h}\in 
\WK: \displaystyle \int_E \PiK v_{h}p=\displaystyle \int_E v_{h}p,\quad \forall p\in 
\mathbb{P}_1(E)\right\}.
\end{equation*}
Now, since $\VK\subset \WK$ the operator $\PiK$ is well defined on $\VK$ and computable  only on the basis of the output values of the operators in $\mathcal{V}_E^h$.
In addition, due to the particular property appearing in definition of the space $\VK$, it can be seen that for every $p \in \mathbb{P}_1(E)$ and every $v_h\in \VK$ the term $(v_h,p)_{0,E}$
is computable from $\PiK v_h$, and hence  the  ${\mathrm L}^2(E,\mathbb{C})$-projector operator $\Pio: \VK\to \P_1(E)$ defined  by
$$\int_E \Pi^{E}v_h p=\int_E v_h p\qquad \forall p\in \mathbb{P}_1(E),$$
depends only on the values of the degrees of freedom of $v_h$.  Actually, it is easy to check that the projectors $\PiK$ and $\Pi^{E}$ are the same operators  on the space $\VK$ (see \cite{AABMR13} for further details).

Finally, for every decomposition $\CT_h$ of $\Omega$ into simple polygons $ E$ we define the global virtual space
\begin{equation}
\label{eq:globa_space}
\Vh:=\left\{v\in H^{1}(E,\mathbb{C}):\ v|_{ E}\in\VK\quad\forall E\in\CT_h\right\},
\end{equation}
and the global degrees of freedom are obtained by collecting the local ones, with the nodal and interface degrees of freedom corresponding to internal entities counted only once
those on the boundary are fixed to be equal to zero in accordance with the ambient space $V$. Finally, we denote by $\Pi^{\nabla}$ the global virtual projection and by $\Pi$ the global $L^{2}$-orthogonal projection.

\subsection{Discrete formulation}
 In order to construct the discrete scheme, we need some preliminary definitions. First, we split the sesquilinear form $\CB(\cdot, \cdot)$ as follows:
 \begin{equation*}
\label{eq:bilineal_form_B_split}
\CB(w,v):=\sum_{ E\in\CT_h} \CB^{E}(w,v):=\sum_{ E\in\CT_h} a^{E}(w,v)+ b^{E}( w,v)\,\, \forall w,v\in  V,
\end{equation*}
\noindent where
\begin{equation*}
	 a^{E}(w,v):=\int_E \kappa(\mathbf{x}) \nabla w\cdot\nabla \bar{v}, \quad
	  b^{E}(w,v):=\int_E(\mathbf{\vartheta}(\mathbf{x})\cdot \nabla w)\bar{v}.\quad
\end{equation*}

Now, in order to propose the discrete counterpart of $a(\cdot,\cdot)$ (cf. \eqref{eq:ab}), we consider any symmetric and semi-positive definite sesquilinear form  $S^{E}:\VK\times\VK\to \mathbb{R}$ satisfying
\begin{equation}
\label{eq:staba}
a_{0}a^{E}(w_{h},w_{h}) \leq S^{E}(w_{h},w_{h}) \leq a_{1}a^{E}(w_{h},w_{h}), \quad  \forall w_{h} \in \VK \cap \ker(\PiK),
\end{equation}
where $a_{0}$ and $a_{1}$ are positive constants depending on the mesh assumptions and on $\kappa$. Next, we introduce another symmetric and positive definite sesquilinear form $S_{0}^{E}:\VK\times\VK\to \mathbb{R}$ satisfying
\begin{equation}
\label{eq:stabc}
c_{0}c^{E}(w_{h},w_{h}) \leq S_{0}^{E}(w_{h},w_{h}) \leq c_{1}c^{E}(w_{h},w_{h}), \quad  \forall w_{h} \in \VK \cap \ker(\Pi^{E}),
\end{equation}
where $c_{0}, c_{1}$ are two uniform positive constants. Then, we introduce on each element $E$ the local (and computable) sesquilinear forms

\begin{itemize}
\item $a_h^E(w_h, v_h)  :=a^{E}(\PiK w_h,\PiK v_h) +S^E( w_h-\PiK w_h, v_h-\PiK  v_h)$,
\vspace{0.3cm}
\item $b_h^E( w_h, v_h)  :=b^E(\PiK w_h,\Pi^{E}v_h)$,
\vspace{0.3cm}
\item $c_{h}^{E}(w_h, v_h) := (\Pi^{E}w_h, \Pi^{E}v_h)_{0,\O} + S_{0}^E( w_h-\Pi^{E} w_h, v_h-\Pi^{E}  v_h)$,\\
\end{itemize}
for all $w_h,v_h\in \VK.$ Moreover, $a_{h}^{E}(\cdot,\cdot)$ satisfies the classical properties of consistency and stability presented in \cite[Section 3.1]{BBCMMR2013}.
We also introduce the broken $H^1$-seminorm $\left|v\right|_{1,h}^2
:=\sum_{E\in\CT_h}\left\|\nabla v\right\|_{0,E}^2.$

As is customary, the discrete sesquilinear forms $\CB_h(\cdot,\cdot)$, $c_{h}(\cdot,\cdot)$ can be expressed componentwise as follows
 \begin{equation*}
\begin{split}
\label{eq:bilineal_form_B_split_{h}}
&\CB_{h}(w_{h},v_{h}):= \sum_{ E\in\CT_h}\CB_{h}^{E}(w_{h},v_{h})
= \sum_{ E\in\CT_h} a_{h}^{E}( w_{h}, v_{h})+ b_{h}^{E}( w_{h},v_{h}), \\
&c_{h}(w_h,v_h) := \sum _{E\in\CT_h} c_{h}^{E}(w_h,v_h).
\end{split}
\end{equation*}
\subsection{Spectral discrete problem}
Now we introduce the VEM discretization of problem \eqref{eq:spectral1}. To do this task, we requiere the global space $V_h$
defined in \eqref{eq:globa_space} together with the assumptions introduced in Section \ref{sec:virtual}.

The spectral problem reads as follows: Find $\lambda_h\in\mathbb{C}$ and $0\neq u_h\in V_h$ such that
\begin{equation}
\label{eq:spectral_disc}
\CB_{h}(u_h,v_h)=\lambda_h c_h(u_h,v_h) \quad \forall v_h \in V_h.
\end{equation}

From \cite[Lemma 5.7]{beiraosec}, we invoke the existence  of  a constant $\underline{\widehat{\beta}}>0$ such that, for all $h<h_0$ there holds
\begin{equation}\label{eq:inf-supBhab}
\displaystyle\sup_{w_h\in V_h}\frac{\mathscr{B}_{h}(v_h,w_h)}{\|w_h\|_{1,\O}}\geq\underline{\widehat{\beta}}\|v_h\|_{1,\Omega}\quad\forall v_h\in V_h.
\end{equation}

This allows us to introduce the discrete solution operator $T_h:V_h\rightarrow V_h$, which is defined by $T_hf_h:=\widetilde{u}_h$ where $\widetilde{u}_h\in V_h$ is the solution of the following source problem
\begin{equation}
\label{eq:source_discrete}
\CB_{h}(\widetilde{u}_h,v_h)=c_h(f_h,v_h)\quad \forall v_h \in V_h.
\end{equation}
As in the continuous case, is important to consider the discrete dual eigenvalue problem: Find $\lambda_{h}^{*} \in \mathbb{C}$ and $0 \neq u_{h}^{*} \in V_{h}$ such that 
\begin{equation*}
\label{eq:spectral_disc_dual}
\CB_{h}(v_h,u_h^{*})=\overline{\lambda_h^{*}} c_h(v_h,u_h^*) \quad \forall v_h \in V_h.
\end{equation*}
For this dual eigenvalue problem, we introduce the discrete adjoint operator of $T_{h}$, $T_h^*:V_h \rightarrow V_h$, which is defined by $T_h^*f_h:=\widetilde{u}_h^*$ where $\widetilde{u}_h^*\in V_h$ is the solution of the following dual source problem
\begin{equation}
\label{eq:source_discrete_dual}
\CB_{h}(v_h,\widetilde{u}_h^*)=c_h(v_h,f_h)\quad \forall v_h \in V_h.
\end{equation}

On the other hand, the following best approximation property for the $L^{2}$-projector $\Pi^{E}$ holds
\begin{equation}\label{eq:bestapproximation}
\|v - \Pi^{E}v\|_{0,E} = \underset{q \in \mathbb{P}_{k}(E)}{\inf} \|v - q\|_{0,E}, \quad \forall v \in L^{2}(E,\mathbb{C}).
\end{equation}
Moreover, the following estimate holds
\begin{equation}\label{eq:errorprojec}
\|v - \Pi^{E}v\|_{0,E} \lesssim h_{E}|v|_{1,E}, \quad \forall v \in H^{1}(E,\mathbb{C}).
\end{equation}

On the other hand, we also have the following well known approximation result for polynomials in star-shaped domain (see for instance \cite{BS-2008}).

\begin{lemma}
\label{eq:polyapprox}
Under the assumptions \textbf{A1} and \textbf{A2}, let $v \in H^{1+s}(E,\mathbb{C})$, with $0 \leq s \leq 1$. Then, there exists $v_{\pi} \in \mathbb{P}_{1}(E)$ such that 
\begin{equation*}
\|v - v_{\pi}\|_{0,E} + h_{E}|v - v_{\pi}|_{1,E} \lesssim h_{E}^{1+s}|v|_{1+s,E}.
\end{equation*}
\end{lemma}

Finally, we have the following result, that provides the existence of an interpolant operator on the virtual space (see \cite[Proposition 4.2]{MR3340705}).

\begin{lemma}[Existence of the virtual interpolation]
\label{eq:interpolant}
Under the assumptions \textbf{A1} and \textbf{A2}, let $v \in H^{1+s}(E,\mathbb{C})$, with $0 \leq s \leq 1$. Then, there exists $v_{I} \in V_{h}^{E}$ such that 
\begin{equation*}
\|v - v_{I}\|_{0,E} + h_{E}|v - v_{I}|_{1,E} \lesssim h_{E}^{1+s}|v|_{1+s,E}.
\end{equation*}
\end{lemma}

\section{A priori error estimates}
\label{sec:a_priori}
In the present section, we present the a priori error estimates for the eigenfunctions and eigenvalues, as well as error estimates for the dual eigenfunctions. With the aim of taking advantage of the compactness of $T$ to prove convergence in norm, and inspired by \cite{MR4050542}, we introduce the operator $P_h:\L^2(\O,\mathbb{C})\rightarrow V_h\hookrightarrow V$ with range $V_h$, which is defined according to the following property:
$
c(P_h u-u,v_h)=0\quad\forall v_h\in V_h
$, where $c(\cdot,\cdot)$ is nothing else but the complex $L^2$ product. Moreover, we have that  $\|P_h u\|_{0,\O}\leq \|u\|_{0,\O}$. With the aid of this operator $P_h$, let us introduce the following operator               $\widehat{T}_h:T_hP_h: V\rightarrow V_h$ which is correctly defined for any source $f\in V$. Moreover, it is easy to check that $\sp(\widehat{T}_h)=\sp(T_h)\cup\{0\}$ and that the eigenfunctions of $\widehat{T}_h$ and $T_h$ coincide. To define the operator $\widehat{T}_h^*$ we procede in the analogous way, but considering the space $V^*$ instead $V$ to define the composition. 

All the previous calculations allow us to obtain the desire convergence in norm of $\widehat{T}_h$ to $T$ as $h$ goes to zero. This is established in the following result. 
\begin{lemma}
\label{col:conv_in_norm}
The following estimates hold:
\begin{enumerate}
\item For the primal operators, there exists $s>0$ such that 
\begin{equation*}
\|(T - \widehat{T}_h) f\|_{1,\O}\lesssim h^{s}\|f\|_{1,\O},
\end{equation*}
\item For the dual operators,  there exists $s^*>0$ such that 
\begin{equation*}
\|(T^* - \widehat{T}^*_h) f^*\|_{1,\O}\lesssim h^{s^*}\|f^*\|_{1,\O},
\end{equation*}
\end{enumerate}
where the hidden constants on each estimate are independent of $h$.
\end{lemma}

\begin{proof}
Let us define $u := Tf$ and $u_{h} := \widehat{T}_{h}f$, and let $u_{I} \in V_{h}$ be the interpolant of $u$. Using triangle inequality, we obtain
\begin{equation*}
\|(T - \widehat{T}_{h})f\|_{1,\O} = \|u - u_{h}\|_{1,\O} \leq \|u - u_{I}\|_{1,\O} + \|u_{h} - u_{I}\|_{0,\O}. 
\end{equation*}
Observe that the first term in the above inequality  is estimated by using Lemma \ref{eq:interpolant} and the boundedness of $T$, so we only need to estimate the second term. To do this task, let us define $v_{h} := u_{h} - u_{I} \in V_{h}$. Then, using \eqref{eq:inf-supBhab} with $w_h=v_h$ we obtain
\begin{multline*}
\widehat{\beta}\|v_{h}\|_{1,\O}^{2} \leq \CB_{h}(v_{h},v_{h}) = \CB_{h}(u_{h},v_{h}) - \CB_{h}(u_{I},v_{h}) \\
= c_{h}(P_{h}f,v_{h}) - \displaystyle{\sum_{E \in \mathcal{T}_{h}} \left[a_{h}^{E}(u_{I}, v_{h}) + b_{h}^{E}(u_{I},v_{h})\right]} \\
= \underbrace{c_{h}(P_{h}f,v_{h}) - c(f,v_{h})}_{\textbf{(I)}} + \underbrace{\displaystyle{\sum_{E \in \mathcal{T}_{h}} \left[a^{E}(u,v_{h}) - a_{h}^{E}(u_{I}, v_{h})\right]}}_{\textbf{(II)}} \\
+ \underbrace{\displaystyle{\sum_{E \in \mathcal{T}_{h}} \left[b^{E}(u,v_{h}) - b_{h}^{E}(u_{I},v_{h})\right]}}_{\textbf{(III)}}.
\end{multline*}
Now, we need to estimate the contributions on the right hand side of the above estimate. Observe that \textbf{(I)} can be estimated using the same arguments  on the  proof of \cite[Lemma 4.2]{MR4050542}, in order to obtain
\begin{equation}\label{eqI}
\textbf{(I)} \lesssim h^{s}\|v_{h}\|_{1,\O}\left(\|f - f_{I}\|_{0,\O} + \|f - f_{\pi}\|_{0,\O}\right),
\end{equation}
where $f_{\pi} \in L^{2}(\O,\mathbb{C})$ with $f_{\pi}\lvert_{E} \in \mathbb{P}_{1}(E)$ is chosen such that Lemma \ref{eq:polyapprox} is satisfied. On the other hand, adding and substracting $\Pi^{E}v_{h}$ in \textbf{(III)}, using the definition of $\PiK$ in \textbf{(II)} and the Cauchy-Schwarz inequality, is easy to see that 
\begin{equation}\label{eqII}
\textbf{(II)} + \textbf{(III)} \lesssim h^{s}\|f\|_{1,\O}\|v_{h}\|_{1,\O} + \|v_{h}\|_{1,\O}\left(|u_{I} - \Pi^{\nabla} u_{I}|_{1,h} + |u - u_{I}|_{1,\O}\right),
\end{equation}
where $u_{\pi} \in L^{2}(\O,\mathbb{C})$ with $u_{\pi}\lvert_{E} \in \mathbb{P}_{1}(E)$ is chosen such that Lemma \ref{eq:polyapprox} is satisfied, and the hidden constant in the previous estimate depends on $\kappa$ and $\vartheta$. Then, applying Lemmas \ref{eq:polyapprox} and \ref{eq:interpolant} on \eqref{eqI} and \eqref{eqII}, together with the boundedness of $T$, we can deduce that $\|v_{h}\|_{1,\O} \lesssim h^{s}\|f\|_{1,\O}$.
Thus, we obtain the desired estimate. Analogous arguments can be used for the term $\|T^{*} - \widehat{T}_{h}^{*}\|_{1,\O}$.
\end{proof}

As a consequence of Lemma \ref{col:conv_in_norm}, we have the following estimates for primal and dual eigenfunctions
\begin{equation}\label{eq:aprioriestimates}
\|u - u_{h}\|_{1,\O} \lesssim h^{s}, \quad \|u^{*} - u_{h}^{*}\|_{1,\O} \lesssim h^{s^{*}},
\end{equation}
where the hidden constants depend on $\kappa$ and $\vartheta$, but not on $h$.

Let us denote by  $\mathfrak{E}$ the eigenspace associated to the eigenvalue $\mu$ of $T$, $\mu_{h}^{(i)}$, $i = 1,\ldots, m$ be discrete eigenvalues of $\widehat{T}_{h}$ (repeated according their respective multiplicities) such that $\mu_{h}^{(i)}$ converges to $\mu$, $1\leq i \leq m$ and let $\mathfrak{E}_{h}$ be the direct sum of their associated eigenspaces. By the relation between the eigenvalues of $T$ and $T^{*}$, let us denote by $\mathfrak{E}^{*}$ the dual of $\mathfrak{E}$.
The following result provides an error estimate for the eigenvalues of our spectral problem. Let us remark that this result depends strongly on the properties of the primal and dual problems.
\begin{theorem}\label{eq:quadord}
The following estimate holds
\begin{equation*}
|\mu - \widehat{\mu}_{h}| \lesssim h^{s+s^{*}}, 
\end{equation*}
where  $\widehat{\mu}_{h}:=\dfrac{1}{m}\sum_{k=1}^m\mu_h^(k)$ and the hidden constant depends on $\kappa$ and $\vartheta$, but not on $h$ and the exponents $s, s^*>0$ are the ones provided by \eqref{eq:regularity} and \eqref{eq:regularity_dual}.
\end{theorem}
\begin{proof}
Let $\{u_{k}\}_{k=1}^{m}$ be such that $Tu_{k} = \mu u_{k}$ for $k=1,\ldots,m$, and let $\{u_{k}^{*}\}_{k=1}^{m}$  be the dual basis of $\{u_{k}\}_{k=1}^{m}$, such that $\CB(u_{k},u_{l}^{*}) = \delta_{k,l}$, where $\delta_{k,l}$ represents the Kronecker delta. From \cite[Theorem 7.2]{BO}, the following identity holds true 
\begin{equation*}
|\mu - \mu_{h}| \lesssim \dfrac{1}{m}\displaystyle{\sum_{k=1}^{m} \left\lvert\left\langle (T-\widehat{T}_{h})u_{k},u_{k}^{*}\right\rangle\right\lvert} + \|T-\widehat{T}_{h}\|_{\mathcal{L}(V)}\|T^{*}-\widehat{T}_{h}^{*}\|_{\mathcal{L}(V^{*})}.
\end{equation*}
Observe that the last two terms can be estimated directly from \eqref{eq:aprioriestimates}. Hence, we need to bound the first term in the above estimate. To do this task, we observe that the following identity can be obtained
\begin{multline*}
\left\langle (T-\widehat{T}_{h})u_{k},u_{k}^{*}\right\rangle  = \CB((T-\widehat{T}_{h})u_{k},u_{k}^{*}) \\
= \CB((T-\widehat{T}_{h})u_{k},u_{k}^{*} - (u_{k}^{*})_{I}) + \CB(Tu_{k},(u_{k}^{*})_{I}) - \CB(\widehat{T}_{h}u_{k},(u_{k}^{*})_{I}) \\
= \underbrace{\CB((T-\widehat{T}_{h})u_{k},u_{k}^{*} - (u_{k}^{*})_{I})}_{(\textrm{I})} + \underbrace{[c(u_{k},(u_{k}^{*})_{I}) - c_{h}(P_{h}u_{k}, (u_{k}^{*})_{I})]}_{(\textrm{II})} \\
+ \underbrace{[\CB_{h}(\widehat{T}_{h}u_{k},(u_{k}^{*})_{I}) - \CB(\widehat{T}_{h}u_{k},(u_{k}^{*})_{I})]}_{(\textrm{III})}.
\end{multline*}
We now estimate the contributions on the right hand side in the above estimate. For $(\textrm{I})$, using Cauchy-Schwarz inequality and Lemmas \ref{col:conv_in_norm} and \ref{eq:interpolant}, we have
\begin{multline}\label{eq:termC1}
(\textrm{I}) \leq C_{\kappa,\vartheta}\|(T-\widehat{T}_{h})u_{k}\|_{1,\O}\|u_{k}^{*} - (u_{k}^{*})_{I}\|_{1,\O} \lesssim h^{s+s^{*}}|u_{k}|_{1+s,\O}|u_{k}^{*}|_{1,\O} \\
\lesssim h^{s+s^{*}}|u_{k}|_{1,\O}|u_{k}^{*}|_{1,\O}.
\end{multline}
For $(\textrm{II})$, using \eqref{eq:errorprojec} and the stability of $P_{h}$ and $(u_{k}^{*})_{I}$, we have
\begin{multline}\label{eq:termC2}
(\textrm{II}) = c(P_{h}u_{k},(u_{k}^{*})_{I}) - c_{h}(P_{h}u_{k}, (u_{k}^{*})_{I}) \\ 
\leq \displaystyle{\sum_{E \in \mathcal{T}_{h}} \|P_{h}u_{k} - \Pi^{E}P_{h}u_{k}\|_{0,E}\|(u_{k}^{*})_{I} - \Pi^{E}(u_{k}^{*})_{I}\|_{0,E}} \lesssim h^{s+s^{*}}|u_{k}|_{1,\O}|u_{k}^{*}|_{1,\O}.
\end{multline}
For $(\textrm{III})$, using triangle inequality we have
\begin{multline*}
(\textrm{III})\leq \displaystyle{\sum_{E \in \mathcal{T}_{h}} |\CB^{E}_{h}(\widehat{T}_{h}u_{k},(u_{k}^{*})_{I}) - \CB^{E}(\widehat{T}_{h}u_{k},(u_{k}^{*})_{I})|} \\
\leq \underbrace{\displaystyle{\sum_{E \in \mathcal{T}_{h}} |a^{E}_{h}(\widehat{T}_{h}u_{k},(u_{k}^{*})_{I}) - a^{E}(\widehat{T}_{h}u_{k},(u_{k}^{*})_{I})|}}_{(\textrm{IV})} \\ 
+  \underbrace{\displaystyle{\sum_{E \in \mathcal{T}_{h}} |b^{E}_{h}(\widehat{T}_{h}u_{k},(u_{k}^{*})_{I}) - b^{E}(\widehat{T}_{h}u_{k},(u_{k}^{*})_{I})|}}_{(\textrm{V})}.
\end{multline*}
We now need to estimate $(\textrm{IV})$ and $(\textrm{V})$. For $(\textrm{IV})$, using Cauchy-Schwarz inequality  and \eqref{eq:staba}  we have
\begin{multline*}
(\textrm{IV}) \leq C_{\kappa}\left(\displaystyle{\sum_{E\in\mathcal{T}_{h}} a^{E}(\PiK \widehat{T}_{h}u_{k} - \widehat{T}_{h}u_{k}, \PiK (u_{k}^{*})_{I} - (u_{k}^{*})_{I})}\right. \\
\left. + \displaystyle{\sum_{E\in\mathcal{T}_{h}} S^{E}(\widehat{T}_{h}u_{k} - \PiK \widehat{T}_{h}u_{k}, (u_{k}^{*})_{I} - \PiK (u_{k}^{*})_{I})}\right) \\
\lesssim \left(\displaystyle{\sum_{E\in\mathcal{T}_{h}} |\widehat{T}_{h}u_{k} - \PiK \widehat{T}_{h}u_{k}|_{1,E}^{2}}\right)^{1/2}\left(\displaystyle{\sum_{E\in\mathcal{T}_{h}} |(u_{k}^{*})_{I} - \PiK (u_{k}^{*})_{I}|_{1,E}^{2}}\right)^{1/2} \\
\lesssim |(\widehat{T}_{h}u_{k} - \Pi^{\nabla}\widehat{T}_{h}u_{k}|_{1,h}|(u_{k}^*)_{I} - \Pi^{\nabla} (u_{k}^{*})_{I}|_{1,h}.
\end{multline*}
Now using triangle inequality and the stability of $\PiK$, together with \eqref{eq:aprioriestimates} and Lemma \ref{eq:polyapprox}, we obtain
\begin{equation*}
|\widehat{T}_{h}u_{k} - \Pi^{\nabla}\widehat{T}_{h}u_{k}|_{1,h} \lesssim \|(T-\widehat{T}_{h})u_{k}\|_{1,\O} + |Tu_{k} - \Pi^{\nabla} Tu_{k}|_{1,h} \lesssim h^{s}|u_{k}|_{1+s,\O}.
\end{equation*} 
On the other hand, using triangle inequality, the stability of $\PiK$ and Lemmas \ref{eq:interpolant} and \ref{eq:polyapprox}, we obtain
\begin{equation*}
|(u_{k}^{*})_{I} - \Pi^{\nabla}(u_{k}^{*})_{I}|_{1,h} \lesssim |u_{k}^{*} - (u_{k}^{*})_{I}|_{1,\O} + |u_{k}^{*} - \Pi^{\nabla} u_{k}^{*}|_{1,h} \lesssim h^{s^{*}}|u_{k}^{*}|_{1+s^{*},\O}.
\end{equation*}
Thus, we obtain $(\textrm{III}) \lesssim h^{s+s^{*}}|u_{k}|_{1,\O}|u_{k}^{*}|_{1,\O}$. Finally, for $(\textrm{V})$, we have
\begin{multline*}
b^{E}(\widehat{T}_{h}u_{k},(u_{k}^{*})_{I}) - b_{h}^{E}(\widehat{T}_{h}u_{k},(u_{k}^{*})_{I}) = \left(\vartheta(\mathbf{x})\cdot\nabla \widehat{T}_{h}u_{k},(u_{k}^{*})_{I}\right)_{0,E} \\ 
- \left(\vartheta(\mathbf{x})\cdot\nabla\PiK \widehat{T}_{h}u_{k},\Pi^{E}(u_{k}^{*})_{I}\right)_{0,E} \\
= \left(\vartheta(\mathbf{x})\cdot\nabla \widehat{T}_{h}u_{k}, (u_{k}^{*})_{I} - \Pi^{E}(u_{k}^{*})_{I}\right)_{0,E} \\ 
+ \left(\vartheta(\mathbf{x})\cdot(\nabla \widehat{T}_{h}u_{k} - \nabla\PiK \widehat{T}u_{k}), \Pi^{E}(u_{k}^{*})_{I}\right)_{0,E} \\
= \left(\vartheta(\mathbf{x})\cdot\nabla \widehat{T}_{h}u_{k} - \Pi^{E}(\vartheta(\mathbf{x})\cdot\nabla \widehat{T}_{h}u_{k}), (u_{k}^{*})_{I} - \Pi^{E}(u_{k}^{*})_{I}\right)_{0,E} \\ 
+ \left(\nabla \widehat{T}_{h}u_{k} - \nabla\PiK \widehat{T}_{h}u_{k}, \vartheta(\mathbf{x})\Pi^{E}(u_{k}^{*})_{I} - \vartheta(\mathbf{x})(u_{k}^{*})_{I}\right)_{0,E} \\
+ \left(\nabla \widehat{T}_{h}u_{k} - \nabla\PiK \widehat{T}_{h}u_{k}, \vartheta(\mathbf{x})(u_{k}^{*})_{I} - \Pi^{E}(\vartheta(\mathbf{x})(u_{k}^{*})_{I})\right)_{0,E}.
\end{multline*}
Similar to the estimate in \textrm(IV), using the Cauchy-Schwarz and triangle inequalities on the above estimate, together with Lemmas \ref{eq:polyapprox} and \ref{col:conv_in_norm} and the stability of the interpolant, we obtain
\begin{multline*}
\textrm{(V)} \lesssim \|\vartheta(\mathbf{x})\cdot\nabla\widehat{T}_{h}u_{k} - \Pi(\vartheta(\mathbf{x})\cdot\nabla\widehat{T}_{h}u_{k})\|_{0,\O}\|(u_{k}^{*})_{I} - \Pi(u_{k}^{*})_{I}\|_{0,\O} \\ 
 |\widehat{T}_{h}u_{k} - \Pi^{\nabla}\widehat{T}_{h}u_{k}|_{1,h}\left(\|(u_{k}^{*})_{I} - \Pi (u_{k}^{*})_{I}\|_{0,\O}\right. \\ 
\left. + \|\vartheta(\mathbf{x})(u_{k}^{*})_{I} - \Pi (\vartheta(\mathbf{x})(u_{k}^{*})_{I})\|_{0,\O}\right) \leq C_{\vartheta}h^{s+s^{*}}|u_{k}|_{1,\O}|u_{k}^{*}|_{1,\O},
\end{multline*}

 allowing us to  conclude that 
\begin{equation}\label{eq:termC3}
(\textrm{III})\lesssim h^{s+s^{*}}|u_{k}|_{1,\O}|u_{k}^{*}|_{1,\O},
\end{equation}
where the hidden constant depends on $\kappa$ and $\vartheta$. Hence, combining \eqref{eq:termC1}, \eqref{eq:termC2} and \eqref{eq:termC3}, we conclude the proof.
\end{proof}
\begin{remark}
The estimation for the eigenvalue error  $|\lambda-\widehat{\lambda}_h|$ by means of $\widehat{\lambda}:=\dfrac{1}{m}\sum_{k=1}^{m} \lambda_h^{(k)}$,  where $\lambda_h^{(k)}=1/\mu_h^{(k)}$ is analogous to the estimation shown above.
\end{remark}
Now we need an error estimates in $L^{2}$ norm.

The following result is based on a classical duality argument.
\begin{lemma}\label{eq:duality1}
Let $f \in \mathfrak{E}$ be such that $\widehat{u} := Tf$ and $\widehat{u}_{h} := \widehat{T}_{h}f$, and let $f^{*} \in \mathfrak{E}^{*}$ be such that $\widehat{u}^{*} := T^{*}f$ and $\widehat{u}_{h}^{*} := \widehat{T}_{h}^{*}f$. Then, the following results hold:
\begin{enumerate}
\item There exists $s >0$ such that 
\begin{multline*}
\|\widehat{u} - \widehat{u}_{h}\|_{0,\O}\lesssim h^{s}\left(\|\widehat{u} - \widehat{u}_{h}\|_{1,\O} + |\widehat{u} - \Pi^{\nabla}\widehat{u}_{h}|_{1,h}\right. \\ 
\left. + \|\vartheta(\mathbf{x})\cdot\nabla \widehat{u}_{h} - \Pi(\vartheta(\mathbf{x})\cdot\nabla \widehat{u}_{h})\|_{0,\O}\right),
\end{multline*}
\item There exists $s^{*}>0$ such that
\begin{multline*}
\|\widehat{u}^* - \widehat{u}_{h}^{*}\|_{0,\O}\lesssim h^{s^*}\left(\|\widehat{u}^* - \widehat{u}_{h}^*\|_{1,\O} + |\widehat{u}^* - \Pi^{\nabla}\widehat{u}_{h}^*|_{1,h}\right. \\
\left. + \|\vartheta(\mathbf{x})\cdot\nabla \widehat{u}_{h}^{*} - \Pi(\vartheta(\mathbf{x})\cdot\nabla \widehat{u}_{h}^{*})\|_{0,\O}\right),
\end{multline*}
\end{enumerate}
where the hidden constants on each estimate depends on $\kappa$ and $\vartheta$, but not on $h$.
\end{lemma}

\begin{proof}
Since the proof for the primal and dual estimates are based in the same arguments, we only present the  proof for the primal estimate. Let us consider the following auxilliary problem: find $w \in V$ such that
\begin{equation}\label{eq:auxiliaryproblem}
\CB(v,w) = c(v,\widehat{u}-\widehat{u}_{h}), \quad \forall v \in V.
\end{equation}
Observe that the solution $w$ of \eqref{eq:auxiliaryproblem} satisfies $\|w\|_{1+s,\O} \lesssim \|\widehat{u}-\widehat{u}_{h}\|_{0,\O}$. 
Then, replacing $v = \widehat{u}-\widehat{u}_{h}$ in \eqref{eq:auxiliaryproblem} and adding and subtracting $w_{I}$, we obtain
\begin{multline}\label{errorenL2}
\|\widehat{u}-\widehat{u}_{h}\|_{0,\O}^{2} = c(\widehat{u}-\widehat{u}_{h},\widehat{u}-\widehat{u}_{h}) = \underbrace{\CB(\widehat{u}-\widehat{u}_{h},w-w_{I})}_{\textbf{(I)}} + \underbrace{\CB(\widehat{u}-\widehat{u}_{h},w_{I})}_{\textbf{(II)}},
\end{multline}
being $w_{I} \in V_{h}$ the interpolant of $w$. Now our task is to estimate the contributions on the right hand side in the above equation. For \textbf{(I)}, using the Cauchy-Schwarz inequality and the additional regularity for $w$, we obtain
\begin{equation}\label{(I)}
\textbf{(I)} \leq C_{\kappa,\vartheta}|w-w_{I}|_{1,\O}|\widehat{u}-\widehat{u}_{h}|_{1,\O} \lesssim h^{s}\|\widehat{u}-\widehat{u}_{h}\|_{0,\O}\|\widehat{u}-\widehat{u}_{h}\|_{1,\O}.
\end{equation}
On the other hand, for \textbf{(II)} we have
\begin{multline}\label{(II)}
\textbf{(II)} = \CB(\widehat{u}-\widehat{u}_{h},w_{I}) = c(f,w_{I}) - B(\widehat{u}_{h},w_{I}) \\ 
= \underbrace{c(f,w_{I}) - c_{h}(P_{h}f,w_{I})}_{\textbf{(III)}} + \underbrace{\CB_{h}(\widehat{u}_{h},w_{I}) - \CB(\widehat{u}_{h},w_{I})}_{\textbf{(IV)}}.
\end{multline}
Now, for \textbf{(III)} we use the same arguments in the proof of \cite[Lemma 4.10]{MR4050542} and the additional regularity of $w$, in order to obtain
\begin{equation}
\textbf{(III)} \lesssim h^{s}\left(\|\widehat{u} - \widehat{u}_{h}\|_{1,\O} + |\widehat{u} - \Pi \widehat{u}_{h}|_{1,h} \right)\|\widehat{u}-\widehat{u}_{h}\|_{0,\O}.
\end{equation}\label{(III)}
On the other hand, to estimate \textbf{(IV)}, we proceeding as in Theorem \ref{eq:quadord}. Then, applying triangle inequality, Lemmas \ref{eq:polyapprox} and \ref{eq:interpolant} and the stability of the interpolant, we obtain
\begin{multline}\label{eq:aa}
|a_{h}^{E}(\widehat{u}_{h},w_{I}) - a^{E}(\widehat{u}_{h},w_{I})| \leq C_{\kappa}|\widehat{u}_{h}-\Pi^{\nabla}\widehat{u}_{h}|_{1,E}|w_{I} - \Pi^{\nabla}w_{I}|_{1,E} \\
\lesssim \left(\|\widehat{u}-\widehat{u}_{h}\|_{1,E} + |\widehat{u} - \Pi^{\nabla} \widehat{u}_{h}|_{1,E}\right)\left(|w - w_{I}|_{1,E} + |w - \Pi^{\nabla}w|_{1,E}\right) \\
\lesssim h^{s}\left(\|\widehat{u}-\widehat{u}_{h}\|_{1,E} + |\widehat{u} - \Pi^{\nabla} \widehat{u}_{h}|_{1,E}\right)|w|_{1+s,E},
\end{multline}
and
\begin{multline}\label{eq:bb}
|b_{h}^{E}(\widehat{u}_{h},w_{I}) - b^{E}(\widehat{u}_{h},w_{I})| \leq C_{\vartheta}\|\vartheta(\mathbf{x})\cdot\nabla \widehat{u}_{h} - \Pi^{E}(\vartheta(\mathbf{x})\cdot\nabla \widehat{u}_{h})\|_{0,E}\|w_{I} - \Pi^{E}w_{I}\|_{0,E} \\ 
+ |\widehat{u}_{h} - \PiK \widehat{u}_{h}|_{1,E}\|\vartheta(\mathbf{x})w_{I} - \vartheta(\mathbf{x})\Pi^{E}w_{I}\|_{0,E} \\
+ |\widehat{u}_{h} - \PiK \widehat{u}_{h}|_{1,E}\|\vartheta(\mathbf{x})w_{I} - \Pi^{E}(\vartheta(\mathbf{x})w_{I})\|_{0,E} \\
\lesssim h^{s}\left(\|\widehat{u}-\widehat{u}_{h}\|_{1,E} + |\widehat{u}-\PiK\widehat{u}_{h}|_{1,E}\right. \\ 
\left.+ \|\vartheta(\mathbf{x})\cdot\nabla \widehat{u}_{h} - \Pi^{E}(\vartheta(\mathbf{x})\cdot\nabla \widehat{u}_{h})\|_{0,E}\right)|w|_{1+s,E}.
\end{multline}
Therefore, gathering \eqref{eq:aa} and \eqref{eq:bb}, using the additional regularity of $w$ and summing over all polygons, we conclude that
\begin{multline}\label{(IV)}
\textbf{(IV)} \lesssim h^{s}\left(\|\widehat{u}-\widehat{u}_{h}\|_{1,\O} + |\widehat{u}-\Pi^{\nabla}\widehat{u}_{h}|_{1,h}\right. \\
\left. + \|\vartheta(\mathbf{x})\cdot\nabla \widehat{u}_{h} - \Pi(\vartheta(\mathbf{x})\cdot\nabla \widehat{u}_{h})\|_{0,\O}\right)\|\widehat{u}-\widehat{u}_{h}\|_{0,\O}.
\end{multline}
Finally, combining the estimates \eqref{(I)}-\eqref{(IV)} with \eqref{errorenL2}, we conclude the proof. Similar arguments can be used for the second estimate.
\end{proof}

Now, we are in position to define a solution operator on the space $L^{2}(\O,\mathbb{C})$ as $\widetilde{T} : L^{2}(\O,\mathbb{C}) \rightarrow L^{2}(\O,\mathbb{C})$, which is defined by $\widetilde{T}\widetilde{f} := \widetilde{u},$ where $\widetilde{u}$ is solution of \eqref{eq:source_discrete}. Observe that $\widetilde{T}$ is compact but not self-adjoint. Hence, we need to introduce the dual solution operator $\widetilde{T}^{*} : L^{2}(\O,\mathbb{C}) \rightarrow L^{2}(\O,\mathbb{C})$, which is defined by $\widetilde{T}^{*}\widetilde{f}^{*} := \widetilde{u}^{*}$, where $\widetilde{u}^{*}$ is solution of \eqref{eq:source_discrete_dual}.  
Moreover, the spectra of $\widetilde{T}$ and $T$ coincide. 
Now, we will establish the convergence of $\widehat{T}_{h}$ to $\widetilde{T}$, and $\widehat{T}_{h}^{*}$ to $\widetilde{T}^{*}$.

\begin{lemma}\label{eq:duality2}
The following results hold:
\begin{enumerate}
\item There exists $s>0$ such that 
\begin{equation*}
\|(\widetilde{T}-\widehat{T}_{h})f\|_{0,\O} \lesssim h^{s}\|f\|_{0,\O} \quad \forall f \in L^{2}(\O,\mathbb{C}),
\end{equation*}
\item There exists $s^{*}>0$ such that
\begin{equation*}
\|(\widetilde{T}^{*}-\widehat{T}_{h}^{*} )f\|_{0,\O} \lesssim h^{s^{*}}\|f\|_{0,\O} \quad \forall f \in L^{2}(\O,\mathbb{C}),
\end{equation*}
where the hidden constants in each estimate depends on $\kappa$ and $\vartheta$, but not on $h$.
\end{enumerate}
\end{lemma}
\begin{proof}
The result for the first estimate follows by repeating the arguments on the proof of Lemma \ref{col:conv_in_norm}, but the term \textbf{(I)} is estimated by using the arguments in the proof of \cite[Lemma 4.11]{MR4050542}. Analogous arguments can be used for the second estimate.
\end{proof}
Now, we establish the following estimates for the error of the primal and dual eigenfunctions.
\begin{lemma}\label{eq:duality3}
Let $u_{h}$ be an eigenfunction of $\widehat{T}_{h}$ associated with the eigenvalue $\mu_{h}^{(i)}$, $1\leq i \leq m$, with $\|u_{h}\|_{0,\O} = 1$, and let $u_{h}^{*}$ be an eigenfunction of $\widehat{T}_{h}^{*}$ associated to the eigenvalue $\mu_{h}^{(i)*}$. The following results hold:
\begin{enumerate}
\item There exists an eigenfunction $u \in L^{2}(\O,\mathbb{C})$ of $T$ associated to the eigenvalue $\mu$ and $s>0$ such that 
\begin{multline*}
\|u-u_{h}\|_{0,\O} \lesssim h^{s}\left(\|\widehat{u} - \widehat{u}_{h}\|_{1,\O} + |\widehat{u} - \Pi^{\nabla}\widehat{u}_{h}|_{1,h}\right. \\ 
\left. + \|\vartheta(\mathbf{x})\cdot\nabla \widehat{u}_{h} - \Pi(\vartheta(\mathbf{x})\cdot\nabla \widehat{u}_{h})\|_{0,\O}\right),
\end{multline*}
\item There exists an eigenfunction $u^{*} \in L^{2}(\O,\mathbb{C})$ of $T^{*}$ associated to the eigenvalue $\mu^{*}$ and $s^{*}>0$ such that 
\begin{multline*}
\|u^{*}-u_{h}^{*}\|_{0,\O} \lesssim h^{s}\left(\|\widehat{u}^{*} - \widehat{u}_{h}^{*}\|_{1,\O} + |\widehat{u}^{*} - \Pi^{\nabla}\widehat{u}_{h}^{*}|_{1,h}\right. \\ 
\left. + \|\vartheta(\mathbf{x})\cdot\nabla \widehat{u}_{h}^{*} - \Pi(\vartheta(\mathbf{x})\cdot\nabla \widehat{u}_{h}^{*})\|_{0,\O}\right),
\end{multline*}
\end{enumerate}
where the hidden constants in each estimate depends on $\kappa$ and $\vartheta$, but not on $h$.
\end{lemma}
\begin{proof}
Applying Lemma \ref{eq:duality2} and \cite[Theorem 7.1]{BO}, we have spectral convergence of $\widetilde{T}_{h}$ to $\widetilde{T}$. Now, due to the relation between the eigenfunctions of $T$ and $T_{h}$ with those of $\widetilde{T}$ and $\widehat{T}_{h}$, we have $u_{h} \in \mathfrak{E}_{h}$ and there exists $u \in \mathfrak{E}$ such that 
\begin{equation*}
\|u-u_{h}\|_{0,\O} \lesssim \underset{\widetilde{f} \in \widetilde{\mathfrak{E}} : \|\widetilde{f}\|_{0,\O}=1}{\sup} \|(\widetilde{T}-\widehat{T}_{h})\widetilde{f}\|_{0,\O},
\end{equation*}
where $\widetilde{\mathfrak{E}}$ is an eigenspace of $\widetilde{T}$. On the other hand, using Lemma \ref{eq:duality1}, for all $\widetilde{f} \in \widetilde{\mathfrak{E}}$, if $f \in \mathfrak{E}$ is such that $f = \widetilde{f}$, then 
\begin{multline*}
\|(\widetilde{T}-\widehat{T}_{h})\widetilde{f}\|_{0,\O} = \|(T-\widehat{T}_{h})f\|_{0,\O} \lesssim h^{s}\left(\|\widehat{u} - \widehat{u}_{h}\|_{1,\O} + |\widehat{u} - \Pi^{\nabla}\widehat{u}_{h}|_{1,h}\right. \\ 
\left.  + \|\vartheta(\mathbf{x})\cdot\nabla \widehat{u}_{h} - \Pi(\vartheta(\mathbf{x})\cdot\nabla \widehat{u}_{h})\|_{0,\O}\right).
\end{multline*}
This concludes the proof for the first estimate. Similar arguments can be used in order to obtain the second estimate.
\end{proof}
We now present the following results, which establish error estimates for the eigenfunctions in $L^{2}$ norm for the primal and dual eigenfunctions.
\begin{theorem}\label{teoremamejororden}
The following results hold:
\begin{enumerate}
\item There exists $s>0$ such that 
\begin{multline*}
\|u-u_{h}\|_{0,\O} \lesssim h^{s}\left(\|u-u_{h}\|_{1,\O} + |u - \Pi^{\nabla}u_{h}|_{1,h} + |\lambda-\lambda_{h}|\|u_{h}\|_{1,\O}\right. \\
\left. + \|\vartheta(\mathbf{x})\cdot\nabla u_{h} - \Pi(\vartheta(\mathbf{x})\cdot\nabla u_{h})\|_{0,\O}\right),
\end{multline*}
\item There exists $s^{*}>0$ such that 
\begin{multline*}
\|u^{*}-u_{h}^{*}\|_{0,\O} \lesssim h^{s^{*}}\left(\|u^{*}-u_{h}^{*}\|_{1,\O} + |u^{*} - \Pi^{\nabla}u_{h}^{*}|_{1,h} + |\lambda-\lambda_{h}|\|u_{h}^{*}\|_{1,\O}\right. \\
\left. + \|\vartheta(\mathbf{x})\cdot\nabla u_{h}^{*} - \Pi(\vartheta(\mathbf{x})\cdot\nabla u_{h}^{*})\|_{0,\O}\right),
\end{multline*}
\end{enumerate}
where the hidden constants in each estimate depends on $\kappa$ and $\vartheta$, but not on $h$.
\end{theorem}
\begin{proof}
We focus our proof on the error for the primal eigenfunctions, since similar arguments can be used for the error of dual eigenfunctions. Invoking Lemma \ref{eq:duality3}, we have 
\begin{multline*}
\|u-u_{h}\|_{0,\O} \lesssim h^{s}\left(\|\widehat{u} - \widehat{u}_{h}\|_{1,\O} + |\widehat{u} - \Pi^{\nabla}\widehat{u}_{h}|_{1,h}\right. \\ 
\left. + \|\vartheta(\mathbf{x})\cdot\nabla \widehat{u}_{h} - \Pi(\vartheta(\mathbf{x})\cdot\nabla \widehat{u}_{h})\|_{0,\O}\right).
\end{multline*}
Our task is to estimate each contribution of the right hand side. Let $\widehat{u} \in V$ be the solution of the problem $\CB(\widehat{u},v) = c(u,v)$ for all $v \in V$.

Since $u$ is solution of \eqref{eq:spectral1}, we obtain $\widehat{u} = u/\lambda$. On the other hand, let $\widehat{u}_{h} \in V_{h}$ be the solution of the discrete problem $\CB_{h}(\widehat{u}_{h},v_{h}) = c_{h}(P_{h}u,v_{h})$, for all $v_{h} \in V_{h}$. Recalling that $u_{h}$ solves problem \eqref{eq:spectral_disc}, observe that the problem $\CB_{h}(\widehat{u}_{h} - u_{h}/\lambda_{h},v_{h}) = c_{h}(P_{h}u - u_{h},v_{h})$ is well-posed for all $\bv_h\in V_h$ and its solution satisfies
\begin{equation*}
\|\widehat{u}_{h} - u_{h}/\lambda_{h}\|_{1,\O} \lesssim \|P_{h}\|\|u-u_{h}\|_{0,\O} \lesssim \|u-u_{h}\|_{0,\O}.
\end{equation*}
On the other hand, adding and substracting $u_{h}/\lambda$ and $u_{h}/\lambda_{h}$, and using triangle inequality we obtain
\begin{multline}\label{111}
\|\widehat{u}-\widehat{u}_{h}\|_{1,\O} \lesssim \dfrac{\|u-u_{h}\|_{1,\O}}{\lambda} + \dfrac{|\lambda - \lambda_{h}|}{|\lambda\lambda_{h}|}\|u_{h}\|_{1,\O} + \|u-u_{h}\|_{0,\O} \\
\lesssim \|u-u_{h}\|_{1,\O} + |\lambda-\lambda_{h}|\|u_{h}\|_{1,\O}.
\end{multline}
Now, we need to control $|\widehat{u} - \Pi^{\nabla}\widehat{u}_{h}|_{1,h}$. To do this task, adding and substracting $u_{h}/\lambda_{h}$ and $\widehat{u}_{h}$, and using triangle inequality we have
\begin{multline}\label{222}
|\widehat{u}-\Pi^{\nabla}\widehat{u}_{h}|_{1,h} \leq  \|\widehat{u}-\widehat{u}_{h}\|_{1,\O} + |\widehat{u}_{h} - u_{h}/\lambda_{h}|_{1,\O} + \dfrac{|u_{h} - \Pi^{\nabla}u_{h}|_{1,h}}{\lambda_{h}} \\
\lesssim \|u-u_{h}\|_{1,\O} + |\lambda - \lambda_{h}|\|u_{h}\|_{1,\O} + |u - \Pi^{\nabla}u_{h}|_{1,h}.
\end{multline}
Finally, proceeding analogously to the previous estimate, we obtain
\begin{multline}\label{333}
\|\vartheta(\mathbf{x})\cdot\nabla \widehat{u}_{h} - \Pi(\vartheta(\mathbf{x})\cdot\nabla \widehat{u}_{h})\|_{0,\O} \lesssim \|\widehat{u}_{h} - u_{h}/\lambda_{h}\|_{1,\O} \\ 
+ \dfrac{\|\vartheta(\mathbf{x})\cdot\nabla u_{h} - \Pi(\vartheta(\mathbf{x})\cdot\nabla u_{h})\|_{0,\O}}{\lambda_{h}} + \|\Pi(\vartheta(\mathbf{x})\cdot\nabla u_{h}/\lambda_{h} - \vartheta(\mathbf{x})\cdot\nabla \widehat{u}_{h})\|_{0,\O} \\
\lesssim \|u-u_{h}\|_{1,\O} + \|\vartheta(\mathbf{x})\cdot\nabla u_{h} - \Pi(\vartheta(\mathbf{x})\cdot\nabla u_{h})\|_{0,\O}.
\end{multline}
Hence, combining \eqref{111}, \eqref{222}, \eqref{333}, and invoking Lemma \ref{eq:duality3} for the primal operators, we conclude the proof.
\end{proof}
\section{A posteriori error analysis}
\label{sec:a_post}
The aim of this section is to introduce a suitable
residual-based error estimator for the convection-diffusion
eigenvalue problem, which results to be  fully computable,
in the sense that  it depends only on quantities available from the VEM solution. 
Let us introduce the following definitions and notations. For any polygon $E\in \CT_h$, we denote by $\CE_{E}$ the set of edges of $E$ and  $\CE_h:=\cup_{E\in\CT_h}\CE_{E}$.
We decompose $\CE_h=\CE_{\O}\cup\CE_{\partial\O}$,
where  $\CE_{\partial\O}:=\{\ell\in \CE_h:\ell\subset \partial\O\}$
and $\CE_{\O}:=\CE\backslash\CE_{\partial \O}$.
For each inner edge $\ell\in \CE_{\O}$ and for any  sufficiently smooth  function
$v$, we define the jump of its normal derivative on $\ell$ by $\left[\!\!\left[ \partial v/\partial \boldsymbol{n}\right]\!\!\right]_\ell:=\nabla (v|_{E})  \cdot \boldsymbol{n}_{E}+\nabla ( v|_{E'}) \cdot \boldsymbol{n}_{E'}$,
where $E$ and $E'$ are  the two elements in $\CT_{h}$  sharing the
edge $\ell$, and $\boldsymbol{n}_{E}$ and $\boldsymbol{n}_{E'}$ are the respective outer unit normal vectors. As a consequence of the mesh regularity assumptions,
we have that each polygon $E\in\CT_h$ admits a sub-triangulation $\CT_h^{E}$
obtained by joining each vertex of $E$ with the midpoint of the ball with respect
to which $\E$ is starred. Let $\CT_h:=\bigcup_{E\in\CT_h}\CT_h^{E}$.
Since we are also assuming \textbf{A2}, $\big\{\CT_h\big\}_{0 < h \leq 1}$
is a shape-regular family of triangulations of $\O$. 
Finally, let us remark that through all this analysis, we will focus only on eigenvalues with simple multiplicity and hence, on their respective associated eigenfunctions.

In order to obtain the error equations for the primal and dual problems, let us define the errors $\texttt{e}_{h} := u - u_{h}$, $\texttt{e}_{h}^{*} := u^{*} - u_{h}^{*}$. We have the following result.
\begin{lemma}\label{eq:errors}
Let $v \in V$. For the primal eigenvalue problem, the following identity holds true
\begin{multline*}
\CB(\texttt{e}_{h},v) = \lambda c(u,v) - \lambda_{h}c(u_h,v) \\
- \displaystyle{\sum_{E \in \mathcal{T}_{h}} \left[a^{E}(u_{h} - \PiK u_h,v) + b^{E}(u_{h} - \PiK u_h,v) - \lambda_{h}c^{E}(u_{h} - \Pi^{E}u_h,v)\right]} \\
+ \displaystyle{\sum_{E \in \mathcal{T}_{h}} \int_{E} \left[\nabla\cdot(\kappa(\mathbf{x})\nabla\PiK u_h) - (\vartheta(\mathbf{x})\cdot \nabla\PiK u_h) + \lambda_{h}\Pi^{E}u_{h}\right]v} \\
+ \displaystyle{\sum_{E \in \mathcal{T}_{h}} \sum_{\ell \in \mathcal{E}_{E}} \int_{\ell} \left(\dfrac{1}{2}\left[\!\!\left[ \dfrac{\partial (\kappa(\mathbf{x})\nabla\PiK u_h)}{\partial{ \boldsymbol{n}}}\right]\!\!\right]_\ell\right) v}.
\end{multline*}
Moreover, for the dual eigenvalue problem, for $v\in V^*$ we have
\begin{multline*}
\CB(\texttt{e}_{h}^{*},v) = \lambda c(\overline{u^{*}},v) - \lambda_{h}c(\overline{u_{h}^{*}},v) \\
 - \displaystyle{\sum_{E \in \mathcal{T}_{h}} \left[a^{E}(\overline{u_{h}^{*}} - \overline{\PiK u_{h}^{*}},v) - b^{E}(\overline{u_{h}^{*}} - \overline{\PiK u_{h}^{*}},v) - \lambda_{h}c^{E}(\overline{u_{h}^{*}} - \overline{\Pi^{E}u_{h}^{*}},v)\right]} \\
+ \displaystyle{\sum_{E \in \mathcal{T}_{h}} \int_{E} \left[\nabla\cdot(\kappa(\mathbf{x})\nabla\overline{\PiK u_{h}^{*}}) + (\vartheta(\mathbf{x})\cdot \nabla\overline{\PiK u_{h}^{*}}) + \lambda_{h}\overline{\Pi^{E}u_{h}^{*}}\right]v} \\
+ \displaystyle{\sum_{E \in \mathcal{T}_{h}} \sum_{\ell \in \mathcal{E}_{E}} \int_{\ell} \left(\dfrac{1}{2}\left[\!\!\left[ \dfrac{\partial (\kappa(\mathbf{x})\nabla\overline{\PiK u_{h}^{*}})}{\partial{ \boldsymbol{n}}}\right]\!\!\right]_\ell\right)v}.
\end{multline*}
\end{lemma}
\begin{proof}
We only prove the estimate for the primal problem since similar arguments are used for the dual error equation. Let  $(\lambda, u)\in\mathbb{C}\times V$ be an eigenpair that solves \eqref{eq:spectral1}. Therefore, using integration by parts and adding and subtracting the projection $\PiK$, we have
\begin{multline*}
\CB(\texttt{e}_{h},v) = \CB(u,v) - \CB(u_h,v) = \lambda c(u,v) - \displaystyle{\sum_{E \in \mathcal{T}_{h}} \left[a^{E}(u_h,v) + b^{E}(u_h,v)\right]} \\
= \lambda c(u,v) - \lambda_{h}c(u_{h},v) - \displaystyle{\sum_{E \in \mathcal{T}_{h}} \left[a^{E}(\PiK u_h,v) + b^{E}(\PiK u_h,v) - \lambda_{h}c^{E}(\Pi^{E}u_h,v)\right]} \\
- \displaystyle{\sum_{E \in \mathcal{T}_{h}} \left[a^{E}(u_h - \PiK u_h,v) + b^{E}(u_h - \PiK u_h,v) - \lambda_{h}c^{E}(u_h - \Pi^{E}u_h,v)\right]} \\
= \lambda c(u,v) - \lambda_{h}c(u_{h},v) - \displaystyle{\sum_{E \in \mathcal{T}_{h}} \left[a^{E}(u_h - \PiK u_h,v) + b^{E}(u_h - \PiK u_h,v)\right.} \\
\left. - \lambda_{h}c^{E}(u_h - \Pi^{E}u_h,v)\right] + \displaystyle{\sum_{E \in \mathcal{T}_{h}} \int_{E}  \left\lbrace\left[\nabla\cdot(\kappa(\mathbf{x})\nabla\PiK u_h) - (\vartheta(\mathbf{x})\cdot\nabla\PiK u_h)\right.\right.} \\
\left.\left.+ \lambda_{h}\Pi^{E}u_h\right]v\right\rbrace + \displaystyle{\sum_{E\in\mathcal{T}_{h}} \sum_{\ell \in \mathcal{E}_{E}}\int_{\ell} \left(\dfrac{1}{2}\left[\!\!\left[ \dfrac{\partial (\kappa(\mathbf{x})\nabla\PiK u_h)}{\partial{ \boldsymbol{n}}}\right]\!\!\right]_\ell\right)v}.
\end{multline*}
This concludes the proof.
\end{proof}
\begin{remark}
Observe that $\nabla\cdot(\kappa(\mathbf{x})\nabla\PiK u_h) = 0$ on each polygon $E \in \mathcal{T}_{E}$, so we will omit this term from now on.
\end{remark}

In order to define our a posteriori error estimator, we introduce the local terms $\Theta_{E}$, $\Theta_{E}^{*}$, $R_{E}$, $R_{E}^{*}$ and the local error indicators $\eta_{E}$ and $\eta_{E}^{*}$ as follows:
\begin{align*}
\Theta_{E}^{2} &:= a_{h}^{E}(u_h - \PiK u_h, u_h - \PiK u_h), & \quad \Theta_{E}^{*2} &:= a_{h}^{E}(\overline{u_{h}^{*}} - \overline{\PiK u_{h}^{*}}, \overline{u_{h}^{*}} - \overline{\PiK u_{h}^{*}}), \\
R_{E}^{2} &:= h_{E}^{2}\|\Upsilon_{E}\|_{0,E}^{2}, & \quad R_{E}^{*2} &:= h_{E}^{2}\|\Upsilon_{E}^{*}\|_{0,E}^{2}, \\
\eta_{E}^{2} &:= \Theta_{E}^{2} + R_{E}^{2} + \displaystyle{\sum_{\ell \in \mathcal{E}_{E}} h_{E}\|J_{\ell}\|_{0,\ell}^{2}}, & \quad \eta_{E}^{*2} &:= \Theta_{E}^{*2} + R_{E}^{*2} + \displaystyle{\sum_{\ell \in \mathcal{E}_{E}} h_{E}\|J_{\ell}^{*}\|_{0,\ell}^{2}},
\end{align*}
where $\Upsilon_{E} := (-(\vartheta(\mathbf{x})\cdot\nabla\PiK u_h) + \lambda_{h}\Pi^{E}u_h)\lvert_{E}$ and $\Upsilon_{E}^{*} := ((\vartheta(\mathbf{x})\cdot\nabla\overline{\PiK u_{h}^{*}}) + \lambda_{h}\overline{\Pi^{E}u_{h}^{*}})\lvert_{E}$ and  the edge residuals are defined  by
\begin{equation*}
J_{\ell} := \left\{\begin{array}{lc}
\dfrac{1}{2}\left[\!\!\left[ \dfrac{\partial (\kappa(\mathbf{x})\nabla\PiK u_h)}{\partial{ \boldsymbol{n}}}\right]\!\!\right]_\ell \quad &\forall \ell \in \mathcal{E}_{\O}, \\
0 \quad &\forall \ell \in \mathcal{E}_{\partial\O},
\end{array}\right.
\end{equation*}
\begin{equation*}
J_{\ell}^{*} := \left\{\begin{array}{lc}
\dfrac{1}{2}\left[\!\!\left[ \dfrac{\partial (\kappa(\mathbf{x})\nabla\overline{\PiK u_{h}^{*}})}{\partial{ \boldsymbol{n}}}\right]\!\!\right]_\ell \quad &\forall \ell \in \mathcal{E}_{\O}, \\
0 \quad &\forall \ell \in \mathcal{E}_{\partial\O},
\end{array}\right.
\end{equation*}
which are clearly computable. With these ingredients at hand, we define the global error estimators $\eta$ and $\eta^{*}$ by
\begin{equation}\label{eq:errorestimator}
\eta := \left(\displaystyle{\sum_{E \in \mathcal{T}_{h}} \eta_{E}^{2}}\right)^{1/2} \quad\text{and}\quad  \eta^{*} := \left(\displaystyle{\sum_{E \in \mathcal{T}_{h}} \eta_{E}^{*2}}\right)^{1/2}.
\end{equation}

\subsection{Reliability}
Our task is to prove that the global error estimators are reliable and efficient. We begin with the reliability analysis proving the following estimate for the error of the eigenfunctions.
\begin{lemma}\label{lema45}
For the primal and dual problems, the following error estimates hold
\begin{align*}
\|u - u_h\|_{1,\O} &\lesssim \eta + |\lambda - \lambda_{h}| + \|u - u_{h}\|_{0,\O}, \\
\|u^{*} - u_{h}^{*}\|_{1,\O} &\lesssim \eta^{*} + |\lambda - \lambda_{h}| + \|u^{*} - u_{h}^{*}\|_{0,\O},
\end{align*}
where the hidden constants depends on $\kappa, \vartheta$ but not on $h$.
\end{lemma}

\begin{proof}
Given $v \in V$ and denoting by $v_{I} \in V_{h}$ its interpolant, we have the following identity
\begin{multline*}
\CB(\texttt{e}_{h},v) = \underbrace{\lambda c(u,v) - \lambda_{h}c(u_h,v) + \displaystyle{\sum_{E \in \mathcal{T}_{h}} \lambda_{h}c^{E}(u_{h} - \Pi^{E}u_{h}, v)}}_{T_{1}} \\ 
+ \underbrace{\displaystyle{\sum_{E \in \mathcal{T}_{h}} \CB(u_{h} - \PiK u_{h}, v)}}_{T_{2}} + \underbrace{\displaystyle{\sum_{E \in \mathcal{T}_{h}} \left\lbrace -\CB(\PiK u_{h}, v - v_{I}) + \lambda_{h}c^{E}(\Pi^{E}u_{h},v - v_{I}) \right\rbrace}}_{T_{3}} \\
+ \underbrace{\displaystyle{\sum_{E \in \mathcal{T}_{h}} \left\lbrace \lambda_{h}c^{E}(\Pi^{E}u_{h},v_{I}) - \CB(\PiK u_h, v_{I}) \right\rbrace}}_{T_{4}},
\end{multline*}
where our task is to estimate each of the contributions on the right hand side. For $T_{1}$, using the triangle  and Cauchy-Schwarz inequalities we obtain
\begin{multline*}\label{eq:T1}
T_{1} = \lambda c(u,v) - \lambda_{h}c(u_h,v) + \displaystyle{\sum_{E \in \mathcal{T}_{h}} \lambda_{h}c^{E}(u_{h} - \Pi^{E}u_{h}, v)} \\
= (\lambda - \lambda_{h})c(u,v) + \lambda_{h}c(u - u_{h},v) + \displaystyle{\sum_{E \in \mathcal{T}_{h}} \lambda_{h}c^{E}(u_{h} - \Pi^{E}u_{h}, v)} \\
\leq \|v\|_{1,\O}\left(|\lambda - \lambda_{h}|\|u\|_{0,\O} + |\lambda_{h}|\|u - u_{h}\|_{0,\O}\right. \\ 
\left.+ |\lambda_{h}|\left(\displaystyle{\sum_{E \in \mathcal{T}_{h}} c^{E}(u_{h} - \Pi^{E}u_{h}, u_{h} - \Pi^{E}u_{h})}\right)^{1/2}\right),
\end{multline*}
where in the last inequality we  have used the obvious inequality $\|v\|_{0,\O} \leq \|v\|_{1,\O}$. For $T_{2}$, applying once again the Cauchy-Schwarz inequality we have
\begin{multline}\label{eq:T2}
T_{2} = \displaystyle{\sum_{E \in \mathcal{T}_{h}} \CB(u_{h} - \PiK u_{h}, v)} = \displaystyle{\sum_{E \in \mathcal{T}_{h}} a^{E}(u_{h} - \PiK u_{h}, v) + b^{E}(u_{h} - \PiK u_{h}, v)} \\
\leq C_{\kappa, \vartheta}\displaystyle{\sum_{E \in \mathcal{T}_{h}} |u_{h} - \PiK u_{h}|_{1,E}(|v|_{1,E} + \|v\|_{0,E})} \\
\leq C_{\kappa,\vartheta}\left(\displaystyle{\sum_{E \in \mathcal{T}_{h}} a_{h}^{E}(u_h - \PiK u_h, u_h - \PiK u_h)}\right)^{1/2}\|v\|_{1,\O},
\end{multline}
where in the last inequality we have used the stability of $a^{E}(\cdot,\cdot)$. For $T_{3}$, using integration by parts and the Cauchy-Schwarz inequality, we have
\begin{multline}\label{eq:T3}
T_{3} = \displaystyle{\sum_{E \in \mathcal{T}_{h}} \left\lbrace -\CB(\PiK u_{h}, v - v_{I}) + \lambda_{h}c^{E}(\Pi^{E}u_{h},v - v_{I}) \right\rbrace} \\
= \displaystyle{\sum_{E \in \mathcal{T}_{h}} \int_{E} \left[-(\vartheta(\mathbf{x})\cdot\nabla\PiK u_h) + \lambda_{h}\Pi^{E}u_h\right](v - v_{I})} + \displaystyle{\sum_{E\in\mathcal{T}_{h}} \sum_{\ell \in \mathcal{E}_{E}} \int_{\ell} J_{\ell}(v - v_{I})} \\
\leq \displaystyle{\sum_{E \in \mathcal{T}_{h}} \|\Upsilon_{E}\|_{0,E}\|v - v_{I}\|_{0,E}} + \displaystyle{\sum_{E \in\mathcal{T}_{h}} \sum_{\ell \in \mathcal{E}_{E}} \|J_{\ell}\|_{0,\ell}\|v-v_{I}\|_{0,\ell}} \\
\lesssim \left(\left(\displaystyle{\sum_{E \in \mathcal{T}_{h}} h_{E}^{2}\|\Upsilon_{E}\|_{0,E}^{2}}\right)^{1/2} + \left(\displaystyle{\sum_{E \in \mathcal{T}_{h}} \sum_{\ell \in \mathcal{E}_{E}} h_{E}\|J_{\ell}\|_{0,\ell}^{2}}\right)^{1/2}\right)\|v\|_{1,\O},
\end{multline}
where in the last inequality we have used a trace inequality and Lemma \ref{eq:interpolant}. Finally, for $T_{4}$ we have
\begin{multline}\label{eq:T4}
T_{4} = \displaystyle{\sum_{E \in \mathcal{T}_{h}} \left\lbrace \lambda_{h}\left(c^{E}(\Pi^{E}u_{h},v_{I}) - c_{h}^{E}(u_{h},v_{I})\right) + \CB_{h}(u_h, v_{I}) - \CB(\PiK u_h, v_{I}) \right\rbrace} \\
= T_{5} + T_{6},
\end{multline}
where the terms $T_5$ and $T_6$ are defined by 
\begin{equation*}
\begin{split}
T_{5} &:= \lambda_{h}\displaystyle{\sum_{E \in \mathcal{T}_{h}} S_{0}^{E}(u_{h} - \Pi^{E}u_{h}, v_{I} - \Pi^{E}v_{I})}, \\
T_{6} &:= \displaystyle{\sum_{E \in \mathcal{T}_{h}} \lambda_{h}c^{E}(\Pi^{E}u_{h}, v_{I}-\Pi^{E}v_{I}) + \CB_{h}(u_{h},v_{I}) - \CB(\PiK u_h, v_{I})},
\end{split}
\end{equation*}
which must be correctly estimated. To accomplish this task, we begin with  $T_{5}$, where using the  Cauchy-Schwarz inequality and \eqref{eq:stabc} we obtain
\begin{multline}\label{eq:T5}
T_{5} \leq |\lambda_{h}|\displaystyle{\sum_{E \in \mathcal{T}_{h}} S_{0}^{E}(u_{h} - \Pi^{E}u_{h}, u_{h} - \Pi^{E}u_{h})^{1/2}S_{0}^{E}(v_{I} - \Pi^{E}v_{I}, v_{I} - \Pi^{E}v_{I})^{1/2}} \\
\leq |\lambda_{h}|\left(\displaystyle{\sum_{E \in \mathcal{T}_{h}} c^{E}(u_{h} - \Pi^{E}u_{h}, u_{h} - \Pi^{E}u_{h})}\right)^{1/2}\|v\|_{0,\O},
\end{multline}
where  last inequality have been obtained as consequence of the  stability of the interpolant $v_I$. Now for the term $T_{6}$, using integration by parts, Cauchy-Schwarz inequality, \eqref{eq:staba}, \eqref{eq:errorprojec}, and Lemma \ref{eq:interpolant}, we obtain
\begin{multline}\label{eq:T6}
T_{6} = \displaystyle{\sum_{E \in \mathcal{T}_{h}} \lambda_{h}c^{E}(\Pi^{E}u_{h}, v_{I}-\Pi^{E}v_{I}) + \CB_{h}(u_{h},v_{I}) - \CB(\PiK u_h, v_{I})} \\
= \displaystyle{\sum_{E \in \mathcal{T}_{h}} \left\lbrace -a^{E}(\PiK u_{h}, v_{I} - \Pi^{E} v_{I}) - b^{E}(\PiK u_{h}, v_{I} - \Pi^{E}v_{I})\right.} 
\\ \left.+ c^{E}(\lambda_{h}\Pi^{E}u_{h}, v_{I} - \Pi^{E}v_{I})\right\rbrace + \displaystyle{\sum_{E\in\mathcal{T}_{h}} S^{E}(u_{h} - \PiK u_{h}, v_{I} - \PiK v_{I})} \\
= \displaystyle{\sum_{E\in\mathcal{T}_{h}} \int_{E} \left[-(\vartheta(\mathbf{x})\cdot\nabla\PiK u_h) + \lambda_{h}\Pi^{E}u_h\right](v_{I} - \Pi^{E}v_{I})} \\
+ \displaystyle{\sum_{E\in\mathcal{T}_{h}}\sum_{\ell \in \mathcal{E}_{E}} \int_{\ell}J_{\ell}(v_{I} - \Pi^{E}v_{I})} + \displaystyle{\sum_{E\in\mathcal{T}_{h}} S^{E}(u_{h} - \PiK u_{h}, v_{I} - \PiK v_{I})} \\
\lesssim \displaystyle{\sum_{E\in\mathcal{T}_{h}} \|\Upsilon_{E}\|_{0,E}\|v_{I}-\Pi^{E}v_{I}\|_{0,E}} + \displaystyle{\sum_{E\in\mathcal{T}_{h}} \sum_{\ell\in\mathcal{E}_{E}} \|J_{\ell}\|_{0,\ell}\|v_{I}-\Pi^{E}v_{I}\|_{0,E}} \\
 + \left(\displaystyle{\sum_{E\in\mathcal{T}_{h}} a_{h}^{E}(u_{h} - \PiK u_{h}, u_{I} - \PiK u_{h})}\right)^{1/2}\|v\|_{1,\O}, \\
\lesssim \left(\displaystyle{\sum_{E\in\mathcal{T}_{h}} h_{E}^{2}\|\Upsilon_{E}\|_{0,E}^{2}}\|\right)^{1/2}\|v\|_{1,\O} + \left(\displaystyle{\sum_{E\in\mathcal{T}_{h}} \sum_{\ell\in\mathcal{E}_{E}} h_{E}\|J_{\ell}\|_{0,\ell}^{2}}\right)^{1/2}\|v\|_{1,\O} \\ 
 + \left(\displaystyle{\sum_{E\in\mathcal{T}_{h}} a_{h}^{E}(u_{h} - \PiK u_{h}, u_{I} - \PiK u_{h})}\right)^{1/2}\|v\|_{1,\O}.
\end{multline} 
Hence, gathering \eqref{eq:T1}, \eqref{eq:T2}, \eqref{eq:T3}, \eqref{eq:T4}, \eqref{eq:T5} and \eqref{eq:T6} and invoking the inf-sup condition \eqref{eq:inf-supB}, we conclude the proof for the estimate of $\|u-u_h\|_{1,\O}$. The proof for the estimate $\|u^*-u_h^*\|_{1,\O}$ follows the same arguments.
\end{proof}

Now, we prove similar estimates for the projection errors $u - \Pi u_{h}$ and $u - \Pi^{\nabla}u_{h}$, and the same estimates for the dual projection errors. This upper bounds are necessary to estimate the eigenvalue error in terms of $\eta$ and $\eta^{*}$.
\begin{lemma}\label{lema46}
The following estimate holds
\begin{equation*}
\begin{split}
\|u - u_{h}\|_{1,\O} + \|u - \Pi u_{h}\|_{0,\O} + |u - \Pi^{\nabla}u_{h}|_{1,h} &\lesssim \eta + |\lambda - \lambda_{h}| + \|u - u_{h}\|_{0,\O}, \\
\|u^{*} - u_{h}^{*}\|_{1,\O} + \|u^{*} - \Pi u_{h}^{*}\|_{0,\O} + |u^{*} - \Pi^{\nabla}u_{h}^{*}|_{1,h} &\lesssim \eta + |\lambda - \lambda_{h}| + \|u^{*} - u_{h}^{*}\|_{0,\O},
\end{split}
\end{equation*}
where the hidden constants depends on $\kappa, \vartheta$ but not on $h$.
\end{lemma}
\begin{proof}
Let $E \in \mathcal{T}_{h}$. Then, from the triangle inequality we have
\begin{equation*}
\|u - \Pi^{E}u_{h}\|_{0,E} + |u - \PiK u_{h}|_{1,E} \leq 2\|u - u_{h}\|_{1,E} + \|u_{h} - \Pi^{E}u_{h}\|_{0,E} + \|u_{h} - \PiK u_{h}\|_{1,E}.
\end{equation*}
Then, using \eqref{eq:bestapproximation}, we obtain
\begin{equation}\label{eq:bestapprox}
\|u_{h} - \Pi^{E}u_{h}\|_{0,E} + \|u_{h} - \PiK u_{h}\|_{1,E} \lesssim \Theta_{E} \leq \eta_{E}.
\end{equation}
The proof follows from \eqref{eq:bestapprox}, summing over all polygons and Lemma \ref{lema45}. The proof is analogous for the second estimate.
\end{proof}

Now, we have the following result for the error of the eigenvalues.
\begin{lemma}\label{lemalambda}
The following estimate holds
\begin{equation*}
|\lambda - \lambda_{h}| \lesssim \eta^{2} + \eta^{*2} + |\lambda - \lambda_{h}|^{2} + \|u - u_{h}\|_{0,\O}^{2} + \|u^{*} - u_{h}^{*}\|_{0,\O}^{2},
\end{equation*}
where the hidden constant depends on $\kappa, \vartheta$, but not of $h$.
\end{lemma}

\begin{proof}
First, we have the following identity
\begin{multline*}
(\lambda_{h} - \lambda)c(u_{h},u_{h}^{*}) = \CB(u - u_{h}, u^{*} - u_{h}^{*}) - \lambda c(u - u_{h}, u^{*} - u_{h}^{*}) \\ 
+ \lambda_{h}\left[c(u_{h},u_{h}^{*}) - c_{h}(u_{h},u_{h}^{*})\right] + \left[\CB_{h}(u_{h}, u_{h}^{*}) - \CB(u_{h},u_{h}^{*})\right].
\end{multline*}
Observe that the term $c(u_{h},u_{h}^{*})$ is needed to be lower bounded. The existence of such lower bound for $c(u_{h},u_{h}^{*})$ follows from \cite[Theorem 3.2]{MR3212379}, i.e., there exists $C>0$ such that $c(u_{h},u_{h}^{*}) > C$. Then, taking modulus and applying triangle inequality, we have 
\begin{multline}\label{eq:doubleorder}
|\lambda - \lambda_{h}| \lesssim \underbrace{|\CB(u - u_{h}, u^{*} - u_{h}^{*})|}_{\widehat{T}_{1}} + \underbrace{|c(u - u_{h}, u^{*} - u_{h}^{*})|}_{\widehat{T}_{2}} + \underbrace{| \CB_{h}(u_{h}, u_{h}^{*}) - \CB(u_{h},u_{h}^{*})|}_{\widehat{T}_{3}} \\
+ \underbrace{|c(u_{h},u_{h}^{*}) - c_{h}(u_{h},u_{h}^{*})|}_{\widehat{T}_{4}},
\end{multline}
where the hidden constant depends precisely on such $C > 0$. Our task is to estimate the four terms in the right-hand side. For $\widehat{T}_{1}$, using Cauchy-Schwarz inequality we obtain
\begin{equation}\label{eq:stimaT1}
\widehat{T}_{1} \leq C_{\kappa,\vartheta} \|u - u_{h}\|_{1,\O}\|u^{*} - u_{h}^{*}\|_{1,\O} \lesssim \|u - u_{h}\|_{1,\O}^{2} + \|u^{*} - u_{h}^{*}\|_{1,\O}^{2}.
\end{equation} 
Analogously for $T_{2}$, applying the Cauchy-Schwarz inequality  we have
\begin{equation}\label{eq:stimaT2}
\widehat{T}_{2} \leq \|u - u_{h}\|_{0,\O}\|u^{*} - u_{h}^{*}\|_{0,\O} \lesssim \|u - u_{h}\|_{0,\O}^{2} + \|u^{*} - u_{h}^{*}\|_{0,\O}^{2}.
\end{equation}
For $\widehat{T}_{3}$, applying triangle inequality, Cauchy-Schwarz inequality, and the boundedness of $S^{E}(\cdot,\cdot)$, we have
\begin{multline}\label{eq:stimaT3}
\widehat{T}_{3} \leq \left\lvert \displaystyle{\sum_{E \in \mathcal{T}_{h}} a^{E}(\PiK u_{h} - u_{h}, \PiK u_{h}^{*} - u_{h}^{*})} \right\lvert \\ 
+ \left\lvert \displaystyle{\sum_{E \in \mathcal{T}_{h}} S^{E}(u_{h} - \PiK u_{h},u_{h}^{*} - \PiK u_{h}^{*})} \right\lvert \\
+ \left\lvert \displaystyle{\sum_{E \in \mathcal{T}_{h}} b^{E}(\PiK u_{h} - u_{h}, \Pi^{E}u_{h}^{*}) - b^{E}(u_{h},u_{h}^{*} - \Pi^{E}u_{h}^{*})} \right\lvert \\
\leq C_{\kappa,\vartheta}\left(|u_{h} - \Pi^{\nabla} u_{h}|_{1,h}|u_{h}^{*} - \Pi^{\nabla}u_{h}^{*}|_{1,h} + |u_{h} - \Pi^{\nabla}u_{h}|_{1,h}\|\Pi u_{h}\|_{0,\O}\right. \\ 
\left. + |u_{h}|_{1,\O}\|u_{h}^{*} - \Pi u_{h}^{*}\|_{0,\O}\right) \\ 
\lesssim |u_{h} - \Pi^{\nabla} u_{h}|_{1,h}^{2} + |u_{h}^{*} - \Pi^{\nabla}u_{h}^{*}|_{1,h}^{2} + |u_{h} - \Pi^{\nabla}u_{h}|_{1,h}^{2} + \|u_{h}^{*} - \Pi u_{h}^{*}\|_{0,\O}^{2},
\end{multline}
where in the last inequality we have used the Archimedean property on $\mathbb{R}$. Finally, for $\widehat{T}_{4}$, using Cauchy-Schwarz inequality  and the boundedness of $S_{0}^{E}(\cdot,\cdot)$, we obtain
\begin{equation}\label{eq:stimaT4}
\widehat{T}_{4} \lesssim \|u_{h} - \Pi u_{h}\|_{0,\O}\|u_{h}^{*} - \Pi u_{h}^{*}\|_{0,\O} \lesssim \|u_{h} - \Pi u_{h}\|_{0,\O}^{2} + \|u_{h}^{*} - \Pi u_{h}^{*}\|_{0,\O}^{2}.
\end{equation}
Therefore, gathering \eqref{eq:doubleorder}, \eqref{eq:stimaT1}, \eqref{eq:stimaT2}, \eqref{eq:stimaT3} and \eqref{eq:stimaT4} and using triangle inequality, we obtain 
\begin{multline*}
|\lambda - \lambda_{h}| \lesssim \|u - u_{h}\|_{1,\O}^{2} + \|u^{*} - u_{h}^{*}\|_{1,\O}^{2} + |u - \Pi^{\nabla}u_{h}|_{1,h}^{2} + |u^{*} - \Pi^{\nabla}u_{h}^{*}|_{1,h}^{2} + \|u - \Pi u_{h}\|_{0,\O}^{2} \\
+ \|u^{*} - \Pi u_{h}^{*}\|_{0,\O}^{2} + \|u - u_{h}\|_{0,\O}^{2} + \|u^{*} - u_{h}^{*}\|_{0,\O}^{2}.
\end{multline*}
Thus, invoking Lemma \ref{lema46}, we conclude the proof.
\end{proof}

\begin{remark}
It is important to mention that the additional terms accompanying the estimators on the right-hand side of Lemmas \ref{lema46} and \ref{lemalambda} are high-order terms, as we have proved  in Theorems \ref{eq:quadord} and \ref{teoremamejororden}.
\end{remark}
%
%
%

\subsection{Efficiency}
Now our aim is to prove that our proposed estimator is locally efficient. To do this task, the usual way to prove this feature is through the use of bubble functions. Hence, we present in the following results standard estimates
for bubble functions on two dimensions that will be useful in what follows (see \cite{MR1885308,MR3059294,MR3719046}).
\begin{lemma}[Interior bubble functions]
\label{burbujainterior}
For any $E\in \CT_h$, let $\psi_{E}$ be the corresponding interior bubble function.
Then, there exists a constant $C>0$ 
independent of  $h_E$ such that
\begin{align*}
C^{-1}\|q\|_{0,E}^2&\leq \int_{E}\psi_{E} q^2\leq \|q\|_{0,E}^2\qquad \forall q\in \mathbb{P}_k(E),\\
C^{-1}\| q\|_{0,E}&\leq \|\psi_{E} q\|_{0,E}+h_E\|\nabla(\psi_{E} q)\|_{0,E}\leq C\|q\|_{0,E}\qquad \forall q\in \mathbb{P}_k(E).
\end{align*}
\end{lemma}
\begin{lemma}[Edge bubble functions]
\label{burbuja}
For any $E\in \CT_h$ and $\ell\in\CE_{E}$, let $\psi_{\ell}$
be the corresponding edge bubble function. Then, there exists
a constant $C>0$ independent of $h_E$ such that
 \begin{equation*}
C^{-1}\|q\|_{0,\ell}^2\leq \int_{\ell}\psi_{\ell} q^2 \leq \|q\|_{0,\ell}^2\qquad
\forall q\in \mathbb{P}_k(\ell).
\end{equation*}
Moreover, for all $q\in\mathbb{P}_k(\ell)$, there exists an extension of  $q\in\mathbb{P}_k(E)$ (again denoted by $q$) such that
 \begin{align*}
h_E^{-1/2}\|\psi_{\ell} q\|_{0,E}+h_E^{1/2}\|\nabla(\psi_{\ell} q)\|_{0,E}&\lesssim \|q\|_{0,\ell}.
\end{align*}
\end{lemma}

We begin with by estimating  the volumetric terms $\|\Upsilon_{E}\|_{0,E}$, $\|\Upsilon_{E}^{*}\|_{0,E}$.
\begin{lemma}\label{eq:volumetrico}
The following estimates hold
\begin{multline*}
\|\Upsilon_{E}\|_{0,E} \lesssim h_{E}^{-1}\left(|u - u_{h}|_{1,E} + \Theta_{E} + h_{E}(\|u_{h} - \Pi^{E}u_{h}\|_{0,E} + \|\lambda u - \lambda_{h}u_{h}\|_{0,E})\right), \\
\|\Upsilon_{E}^{*}\|_{0,E} \lesssim h_{E}^{-1}\left(|u^{*} - u_{h}^{*}|_{1,E} + \Theta_{E}^{*} + h_{E}(\|u_{h}^{*} - \Pi^{E}u_{h}^{*}\|_{0,E} + \|\lambda u^{*} - \lambda_{h}u_{h}^{*}\|_{0,E})\right),
\end{multline*}
where the hidden constants depends on $\kappa, \vartheta$ but not on $h_{E}$.
\end{lemma}

\begin{proof}
Let $\psi_{E}$ be the interior bubble function defined  in Lemma \ref{burbujainterior}. Let us define the function $v := \Upsilon_{E}\psi_{E}$ which vanishes on the boundary of $E$ and also may be extended by zero to the whole domain $\Omega$, implying that $v \in H_{0}^{1}(\O)$. Then, according to Lemma \ref{eq:errors}, for each $E \in \mathcal{T}_{h}$ we have
\begin{multline*}
\CB^{E}(\texttt{e}_{h},\Upsilon_{E}\psi_{E}) = \lambda c^{E}(u,\Upsilon_{E}\psi_{E}) - \lambda_{h}c^{E}(u_h,\Upsilon_{E}\psi_{E}) - a^{E}(u_{h} - \PiK u_h,\Upsilon_{E}\psi_{E}) \\
- b^{E}(u_{h} - \PiK u_h,\Upsilon_{E}\psi_{E}) + \lambda_{h}c^{E}(u_{h} - \Pi^{E}u_h,\Upsilon_{E}\psi_{E}) + \displaystyle{\int_{E} \Upsilon_{E}^{2}\psi_{E}}.
\end{multline*}
Then, we have
\begin{multline*}
\|\Upsilon_{E}\|_{0,E}^{2} \lesssim \displaystyle{\int_{E} \Upsilon_{E}^{2}\psi_{E}} = \CB^{E}(\texttt{e}_{h},\Upsilon_{E}\psi_{E}) - \lambda c^{E}(u,\Upsilon_{E}\psi_{E}) + \lambda_{h}c^{E}(u_h,\Upsilon_{E}\psi_{E}) \\ 
a^{E}(u_{h} - \PiK u_h,\Upsilon_{E}\psi_{E}) + b^{E}(u_{h} - \PiK u_h,\Upsilon_{E}\psi_{E}) - \lambda_{h}c^{E}(u_{h} - \Pi^{E}u_h,\Upsilon_{E}\psi_{E}) \\
 \leq C_{\kappa}h_{E}^{-1}(|u - u_{h}|_{1,E} + |u_{h} - \PiK u_{h}|_{1,E})\|\Upsilon_{E}\|_{0,E} + |\lambda_{h}|\|u_{h} - \Pi^{E}u_{h}\|_{0,E}\|\Upsilon_{E}\|_{0,E} \\
 + C_{\vartheta}(|u - u_{h}|_{1,E} + |u_{h} - \PiK u_{h}|_{1,E})\|\Upsilon_{E}\|_{0,E} + \|\lambda u - \lambda_{h} u_{h}\|_{0,E}\|\Upsilon_{E}\|_{0,E} \\
\lesssim h_{E}^{-1}\|\Upsilon_{E}\|_{0,E}\left(|u - u_{h}|_{1,E} + 2\Theta_{E} + h_{E}(|u - u_{h}|_{1,E} + \|u_{h} - \Pi^{E} u_{h}\|_{0,E}\right. \\ 
\left. + \|\lambda u - \lambda_{h}u_{h}\|_{0,E})\right),
\end{multline*}
where in the last inequality we have used \eqref{eq:bestapprox}. This concludes the proof for the first estimate. Analogous arguments are used for the estimate involving the dual problem.
\end{proof}

The following result gives an estimate for the primal and dual consistency terms $\Theta_{E}$ and $\Theta_{E}^{*}$, respectively.
\begin{lemma}\label{eq:theta}
The following estimates holds
\begin{equation*}
\Theta_{E} \lesssim \|u - u_{h}\|_{1,E} + |u - \PiK u_{h}|_{1,E}, \,\,\text{and}\,\,\,\,
\Theta_{E}^{*} \lesssim \|u^{*} - u_{h}^{*}\|_{1,E} + |u^{*} - \PiK u_{h}^{*}|_{1,E},
\end{equation*}
where the hidden constants are independent of $h_{E}$. 
\end{lemma}

\begin{proof}
Observe that $\Theta_{E}^{2} = S^{E}(u_{h} - \PiK u_{h},u_{h} - \PiK u_{h})$. Then, using triangle inequality and the boundedness of $S^{E}(\cdot,\cdot)$, we have
\begin{equation*}
\Theta_{E}^{2} \lesssim |u_{h} - \PiK u_{h}|_{1,E}^{2} \leq (\|u - u_{h}\|_{1,E} + |u - \PiK u_{h}|_{1,E})^{2}.
\end{equation*}
The same arguments are used for the second estimate. This concludes the proof.
\end{proof}

Now we prove an estimate for the jumps terms $\|J_{\ell}\|_{0,\ell}$ and $\|J_{\ell}^{*}\|_{0,\ell}$.
\begin{lemma}\label{eq:residualterm}
The following estimates holds
\begin{align*}
h_{E'}^{1/2}\|J_{\ell}\|_{0,\ell} \lesssim \displaystyle{\sum_{E' \in \omega_{\ell}} \left(\|\texttt{e}_{h}\|_{1,E'} + |u - \Pi^{\nabla, E'}u_{h}|_{1,E'} + h_{E'}\|u_{h} - \Pi^{E'}u_{h}\|_{0,E'}\right.} \\
+ \left.h_{E'}\|\lambda u - \lambda_{h}u_{h}\|_{0,E'}\right), \\
h_{E'}^{1/2}\|J_{\ell}^{*}\|_{0,\ell} \lesssim \displaystyle{\sum_{E' \in \omega_{\ell}} \left(\|\texttt{e}_{h}^{*}\|_{1,E'} + |u^{*} - \Pi^{\nabla, E'}u_{h}^{*}|_{1,E'} + h_{E'}\|u_{h}^{*} - \Pi^{E'}u_{h}^{*}\|_{0,E'}\right.} \\
+ \left.h_{E'}\|\lambda u^{*} - \lambda_{h}u_{h}^{*}\|_{0,E'}\right),
\end{align*}
\end{lemma}
where $\omega_{\ell} := \{E' \in \mathcal{T}_{h} : \ell \in \mathcal{E}_{E'}\}$ and the hidden constants depend on $\kappa, \vartheta$ but not on $h_{E'}$.
\begin{proof}
Let $\ell \in \mathcal{E}_{E'} \cap \mathcal{E}_{\O}$. Let us define $v := J_{\ell}\psi_{\ell}$, where $\psi_{\ell}$ is the edge bubble function satisfying Lemma \ref{burbuja}. Observe that $v$ may be extended by zero to the whole domain $\O$ and, just for simplicity, we denote  this  extension by $v$. We remark that $v \in V$. Invoking  Lemma \ref{eq:errors}, we have
\begin{multline*}
\CB(\texttt{e}_{h},J_{\ell}\psi_{\ell}) = \lambda c(u,J_{\ell}\psi_{\ell}) - \lambda_{h}c(u_h,J_{\ell}\psi_{\ell}) - \displaystyle{\sum_{E' \in \omega_{\ell}} \left(a^{E'}(u_{h} - \Pi^{\nabla, E'} u_h,J_{\ell}\psi_{\ell})\right.} \\ 
\left. + b^{E'}(u_{h} - \Pi^{\nabla, E'} u_h,J_{\ell}\psi_{\ell}) - \lambda_{h}c^{E'}(u_{h} - \Pi^{E'}u_h,J_{\ell}\psi_{\ell})\right) \\ 
+ \displaystyle{\sum_{E' \in \omega_{\ell}}  \left(\int_{E'} \Upsilon_{E'}J_{\ell}\psi_{\ell} + \int_{\ell} J_{\ell}^{2}\psi_{\ell}\right)}.
\end{multline*}
Then, applying Lemma \ref{burbuja}, we have
\begin{multline}\label{eq:edgeres}
\|J_{\ell}\|_{0,\ell}^{2} \lesssim \displaystyle{\int_{\ell} J_{\ell}^{2}\psi} = \CB(\texttt{e}_{h},J_{\ell}\psi_{\ell}) - c(\lambda u - \lambda_{h}u_{h},J_{\ell}\psi_{\ell}) - \displaystyle{\sum_{E' \in \omega_{\ell}}  \int_{E'} \Upsilon_{E'}J_{\ell}\psi_{\ell}} \\
+ \displaystyle{\sum_{E' \in \omega_{\ell}} \left(a^{E'}(u_{h} - \Pi^{\nabla, E'} u_h,J_{\ell}\psi_{\ell}) + b^{E'}(u_{h} - \Pi^{\nabla, E'} u_h,J_{\ell}\psi_{\ell})\right.} \\ 
\left. - \lambda_{h}c^{E'}(u_{h} - \Pi^{E'}u_h,J_{\ell}\psi_{\ell})\right) \\
\leq C_{\kappa}\displaystyle{\sum_{E' \in \omega_{\ell}} \left(|\texttt{e}_{h}|_{1,E'} + |u_{h} - \Pi^{\nabla, E'}u_{h}|_{1,E'}\right)|J_{\ell}\psi_{\ell}|_{1,E'}} + \displaystyle{\sum_{E' \in \omega_{\ell}} \|\Upsilon_{E'}\|_{0,E'}\|J_{\ell}\psi_{\ell}\|_{0,E'}} \\
+ C_{\vartheta}\displaystyle{\sum_{E' \in \omega_{\ell}} \left(|\texttt{e}_{h}|_{1,E'} + |u_{h} - \Pi^{\nabla, E'}u_{h}|_{1,E'}\right)\|J_{\ell}\psi_{\ell}\|_{0,E'}} \\ 
+ \displaystyle{\sum_{E' \in \omega_{\ell}} \|u_{h} - \Pi^{E}u_{h}\|_{0,E'}\|J_{\ell}\psi_{\ell}\|_{0,E'}} + \displaystyle{\sum_{E' \in \omega_{\ell}} \|\lambda u - \lambda_{h}u_{h}\|_{0,E'}\|J_{\ell}\psi_{\ell}\|_{0,E'}}.
\end{multline}
Now we have the following set of inequalities:
\begin{align*}
\left(|\texttt{e}_{h}|_{1,E'} + |u_{h} - \Pi^{\nabla, E'}u_{h}|_{1,E'}\right)|J_{\ell}\psi_{\ell}|_{1,E'} \lesssim h_{E'}^{-1/2}\left(\|\texttt{e}_{h}\|_{1,E'} + \Theta_{E'}\right)\|J_{\ell}\|_{0,\ell}; \\
\left(|\texttt{e}_{h}|_{1,E'} + |u_{h} - \Pi^{\nabla, E'}u_{h}|_{1,E'}\right)\|J_{\ell}\psi_{\ell}\|_{0,E'} \lesssim h_{E'}^{1/2}\left(\|\texttt{e}_{h}\|_{1,E'} + \Theta_{E'}\right)\|J_{\ell}\|_{0,\ell}; \\
\|u_{h} - \Pi^{E}u_{h}\|_{0,E'}\|J_{\ell}\psi_{\ell}\|_{0,E'} \lesssim h_{E}^{1/2}\|u_{h} - \Pi^{E}u_{h}\|_{0,E'}\|J_{\ell}\|_{0,\ell}; \\
\|\lambda u - \lambda_{h}u_{h}\|_{0,E'}\|J_{\ell}\psi_{\ell}\|_{0,E'} \lesssim h_{E}^{1/2}\|\lambda u - \lambda_{h}u_{h}\|_{0,E'}\|J_{\ell}\|_{0,\ell}; \\
\|\Upsilon_{E'}\|_{0,E'}\|J_{\ell}\psi_{\ell}\|_{0,E'} \lesssim h_{E'}^{-1/2}\left(\|\texttt{e}_{h}\|_{1,E} + \Theta_{E'}\right) +  h_{E'}^{1/2}\left(\|\texttt{e}_{h}\|_{1,E'} + \Theta_{E'}\right. \\
\left. + \|\lambda u - \lambda_{h}u_{h}\|_{0,E'} + \|u_{h} - \Pi^{E}u_{h}\|_{0,E'}\right),
\end{align*}
where in the last inequality we have employed Lemma \ref{eq:volumetrico}. Hence, combining \eqref{eq:edgeres} with the above estimates and using Lemma \ref{eq:theta}, we conclude the proof for the first estimate. Similar arguments can be used to prove the second estimate.
\end{proof}

Now, we prove estimates for the local error indicators $\eta_{E}$ and $\eta_{E}^{*}$. To do this task is necessary to have  estimates for the local volumetric terms, consistency terms and the edge residuals. All estimates has been obtained in the previous results.

\begin{corollary}[Local efficiency]
The following estimate for the local error indicators $\eta_{E}$ and $\eta_{E}^{*}$ holds
\begin{multline*}
\eta_{E}^{2} \lesssim \displaystyle{\sum_{E' \in \omega_{E}} \left(\|\texttt{e}_{h}\|_{1,E'}^{2} + |u - \Pi^{\nabla, E'}u_{h}|_{1,E'}^{2} + \|u_{h} - \Pi^{E'}u_{h}\|_{0,E'}^{2}\right.} \\
\left.+ h_{E'}^{2}\|\lambda u - \lambda_{h}u_{h}\|_{0,E'}^{2}\right), \\
\eta_{E}^{*2} \lesssim \displaystyle{\sum_{E' \in \omega_{E}} \left(\|\texttt{e}_{h}^{*}\|_{1,E'}^{2} + |u^{*} - \Pi^{\nabla, E'}u_{h}^{*}|_{1,E'}^{2} + \|u_{h}^{*} - \Pi^{E'}u_{h}^{*}\|_{0,E'}^{2}\right.} \\
\left.+ h_{E'}^{2}\|\lambda u^{*} - \lambda_{h}u_{h}^{*}\|_{0,E'}^{2}\right),
\end{multline*}
where $\omega_{E} := \{E' \in \mathcal{T}_{h} : E' \; \text{and} \; E \; \text{share an edge}\}$.
\end{corollary}

\begin{proof}
The results follows directly from Lemmas \ref{eq:volumetrico} - \ref{eq:residualterm}.
\end{proof}

\section{Numerical experiments}
\label{sec:numerics}
In the following section we report numerical evidence to support our theoretical results. Particulalry
we present the results related to the a posteriori error estimates, in order to assess the bevahior of the a posteriori estimators defined in \eqref{eq:errorestimator}, since as we will observe on the forthcominng results, the a priori error is analyzed with the  uniform refinement that we compare with the adaptive one. With this aim, we have implemented in a MATLAB code a lowest order VEM scheme on arbitrary polygonal meshes and   the mesh refinement algorithm described in \cite{MR3342219},  which consists in splitting each element of the mesh into $n$ quadrilaterals ($n$ being the number of edges of the polygon) by connecting the barycenter of the element with the midpoint of each edge, which will be named as  \textbf{Adaptive VEM}. Notice that although this process is initiated with a mesh of triangles, the successively created meshes will contain other kind of convex polygons, as it can be seen in Figure \ref{FIG:AdaptativeT}. The schemes are based on the strategy of refining those elements $\E\in\CT_h$ that satisfy
 $$\boldsymbol{\eta}_{\E}\geq 0.5 \max_{{\E'\in\CT_{h}}}\{\boldsymbol{\eta}_{\E'}\}.$$
\begin{figure}[h!]
\begin{center}
\begin{minipage}{6.4cm}
\centering\includegraphics[height=5.7cm, width=4.5cm, angle =90]{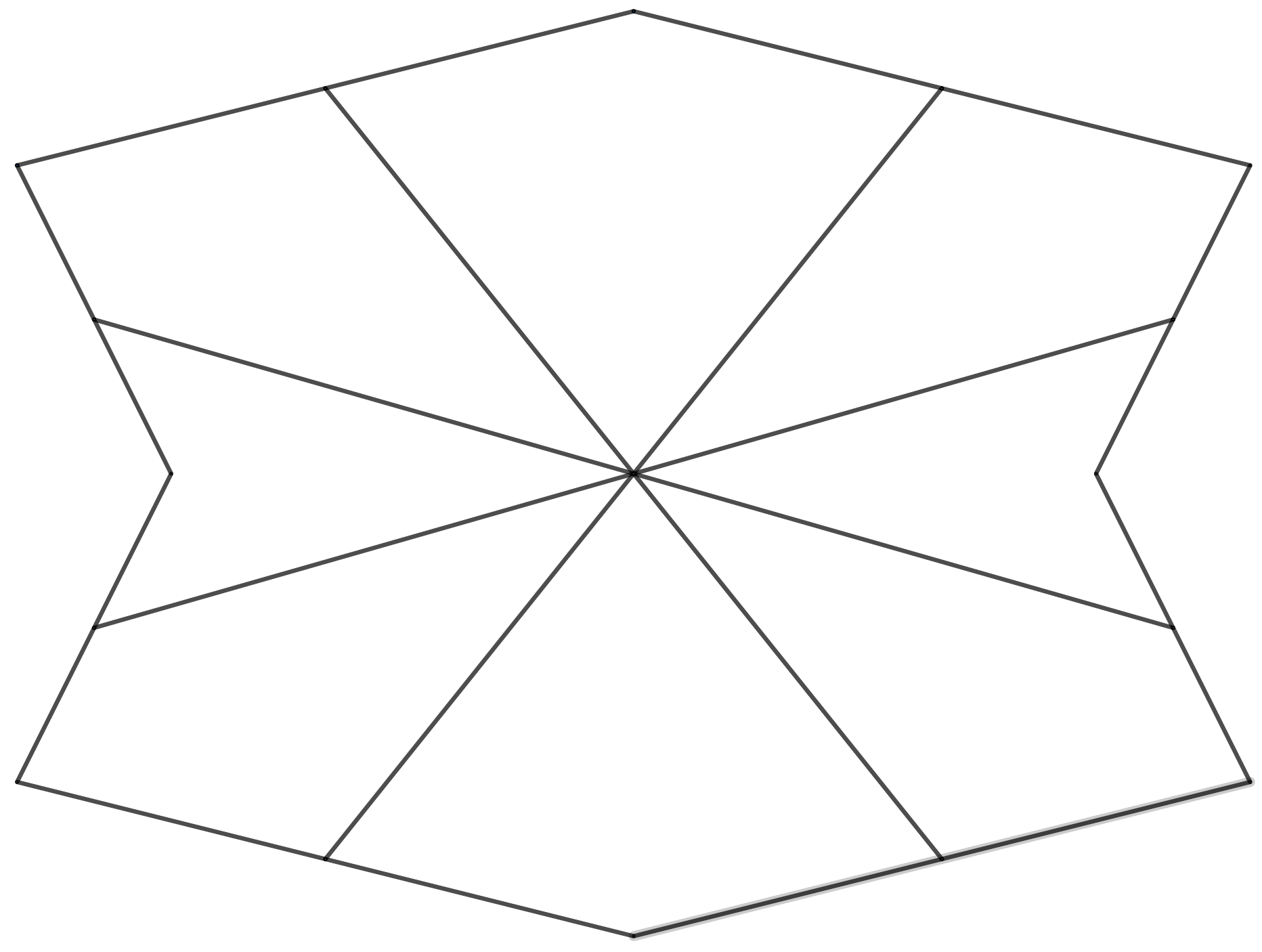}
\end{minipage}
\caption{\label{FIG:AdaptativeT} Example of refined elements for the VEM strategy.}
\end{center}
\end{figure}

We have tested the method by using different families of meshes (see Figure \ref{fig:mesh}):
\begin{itemize}
\item tria: triangular meshes;
\item quad: squares meshes;
\item hexa: structured hexagonal meshes made of convex hexagons; 
\item voro: non-structured Voronoi meshes.
\end{itemize}
\begin{figure}[h!]
\begin{center}
\begin{minipage}{5.2cm}
\centering\includegraphics[height=5.1cm, width=5.1cm]{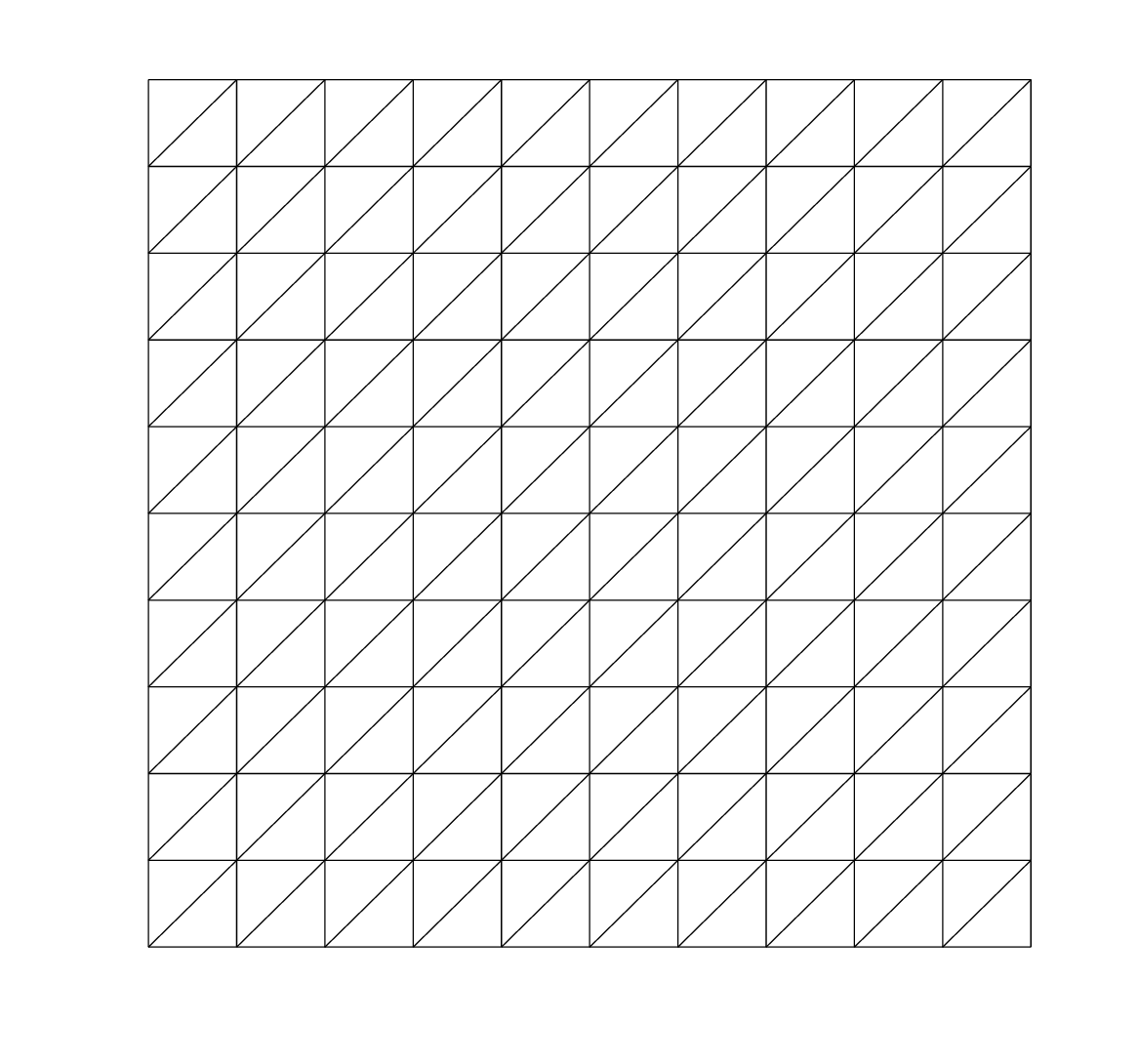}
\end{minipage}
\begin{minipage}{5.2cm}
\centering\includegraphics[height=5.1cm, width=5.1cm]{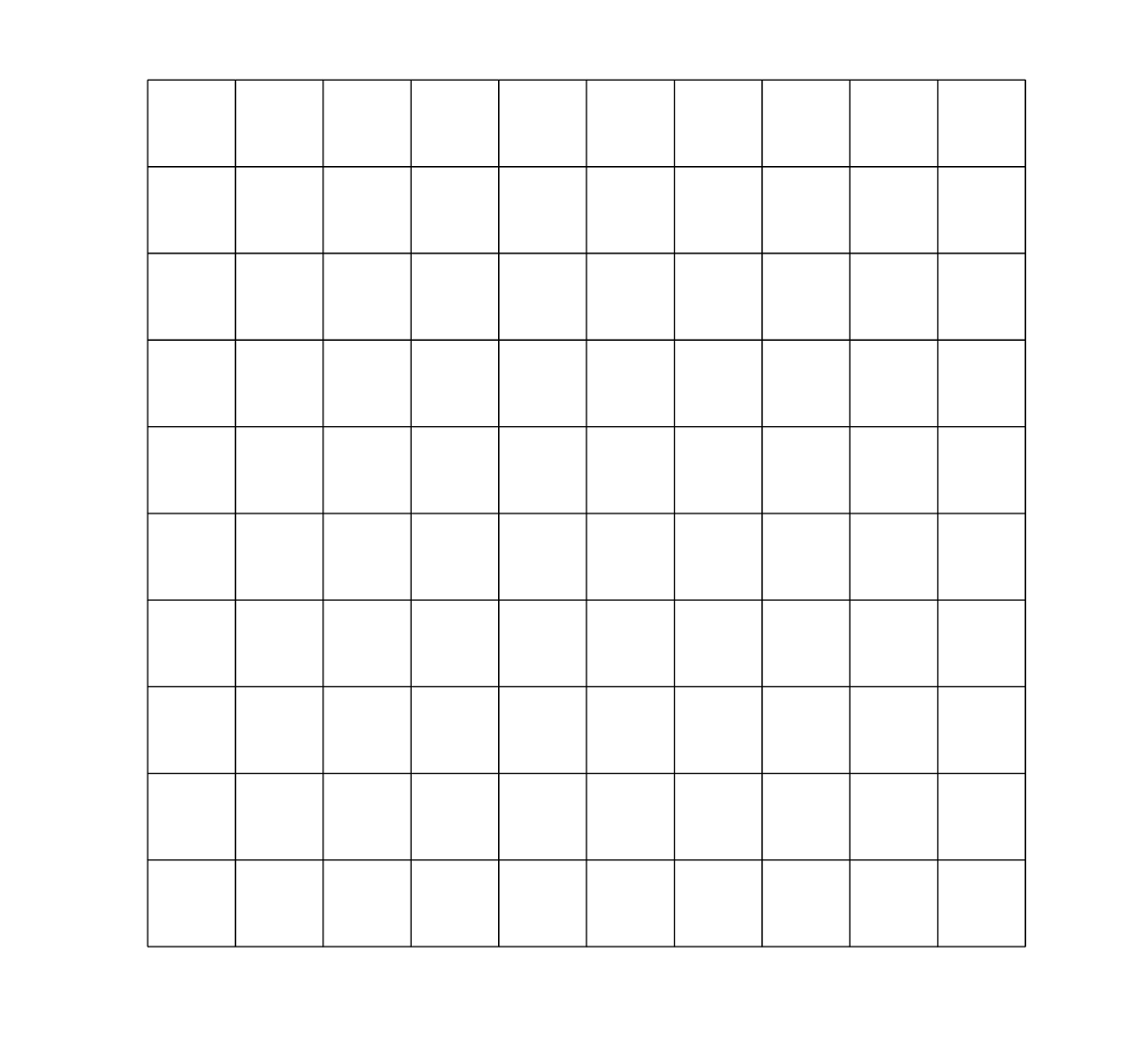}
\end{minipage}\\
\begin{minipage}{5.2cm}
\centering\includegraphics[height=5.1cm, width=5.1cm]{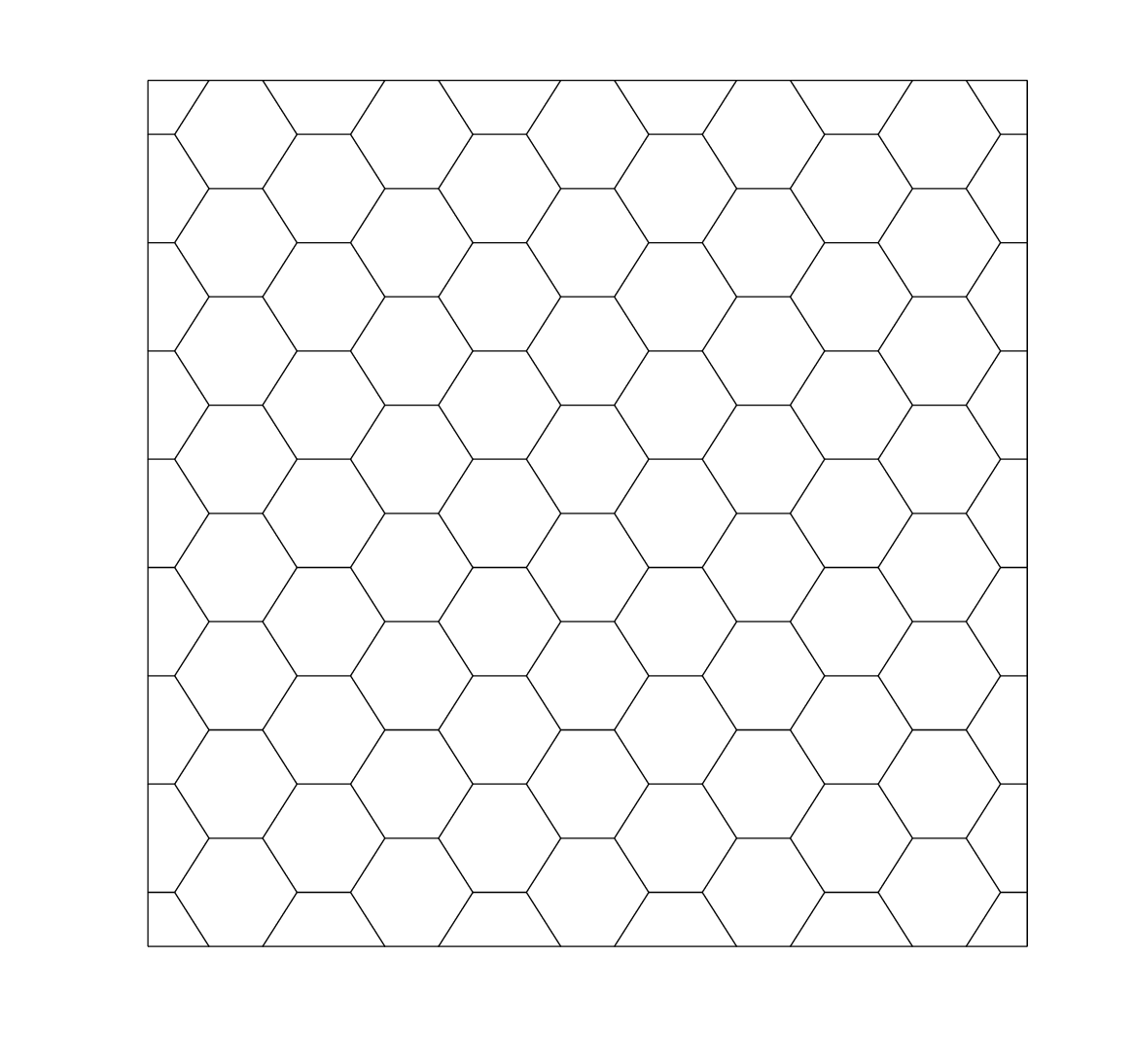}
\end{minipage}
\begin{minipage}{5.2cm}
\centering\includegraphics[height=5.1cm, width=5.1cm]{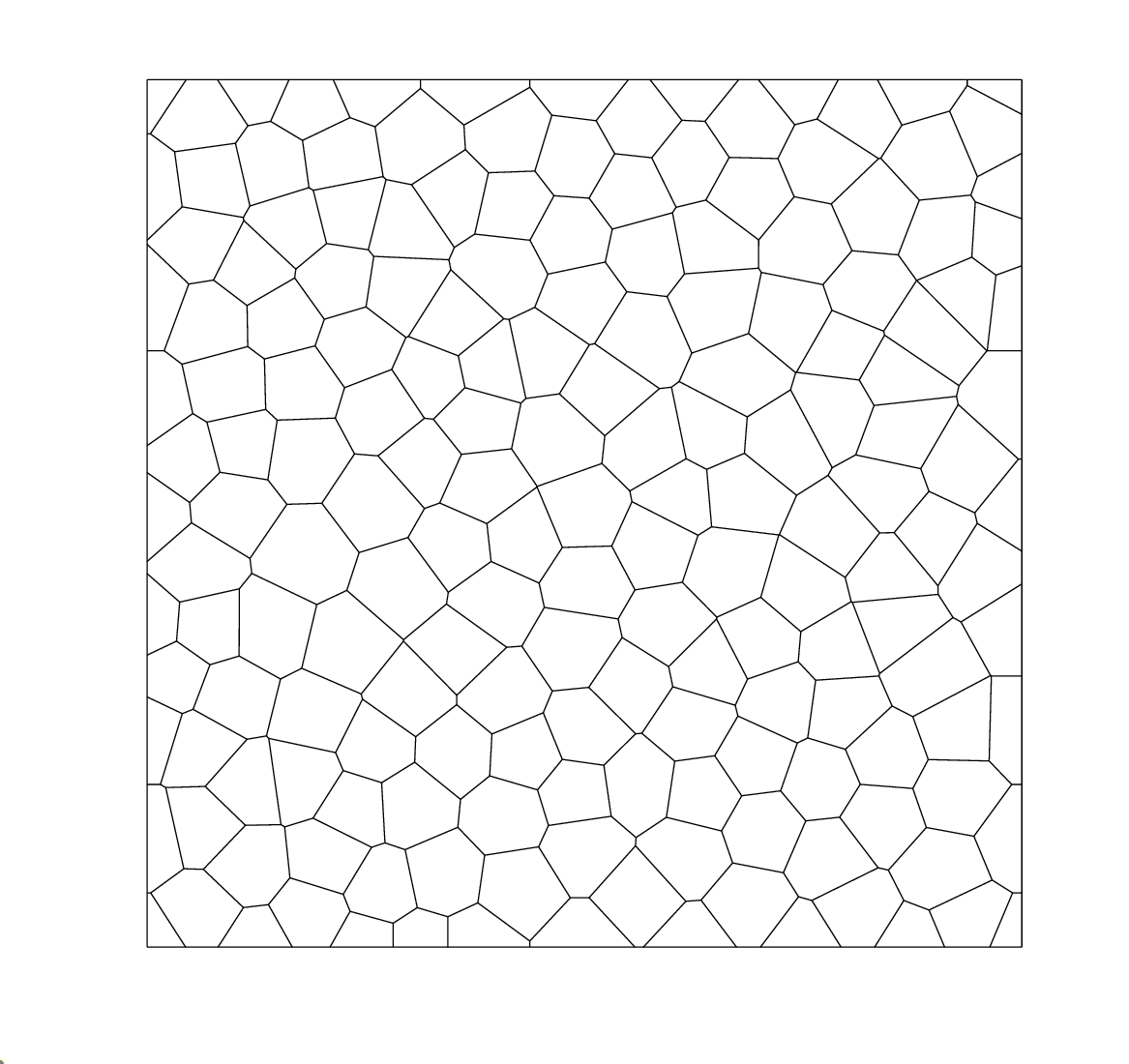}
\end{minipage}
\caption{\label{fig:mesh} Sample meshes: tria (top left), quad (top right), hexa (bottom  left), voro (bottom right).}
\end{center}
\end{figure}

\subsection{Test 1: L-shaped domain}
For the first test, we will consider the non-convex domain $\Omega:=(-1,1)\times(-1,1) \setminus [0, 1]\times [-1, 0]$, which is a L-shaped domain, with boundary condition $u = 0$. Here $\kappa(\mathbf{x})=1$ and $\vartheta(\mathbf{x})=(3,0)$. 
We will use as approximation of the first lowest eigenvalue $\lambda_1=|\vartheta(\mathbf{x})|+9.6397238$ (see \cite{MR3133493} for more details). In Figures \ref{fig:adaptiveVEMH} and \ref{fig:adaptiveVEMV} we present the adaptively refined meshes obtained with VEM procedure for different initial meshes.


\begin{figure}[h!]
\begin{center}
\caption{\label{fig:adaptiveVEMH} Test 1: Adaptively refined meshes obtained with VEM scheme at refinement steps 0, 1 and 8 initiated with an hexagonal mesh (Adaptive VEMH).}
\begin{minipage}{4.2cm}
\centering\includegraphics[height=4.1cm, width=4.1cm]{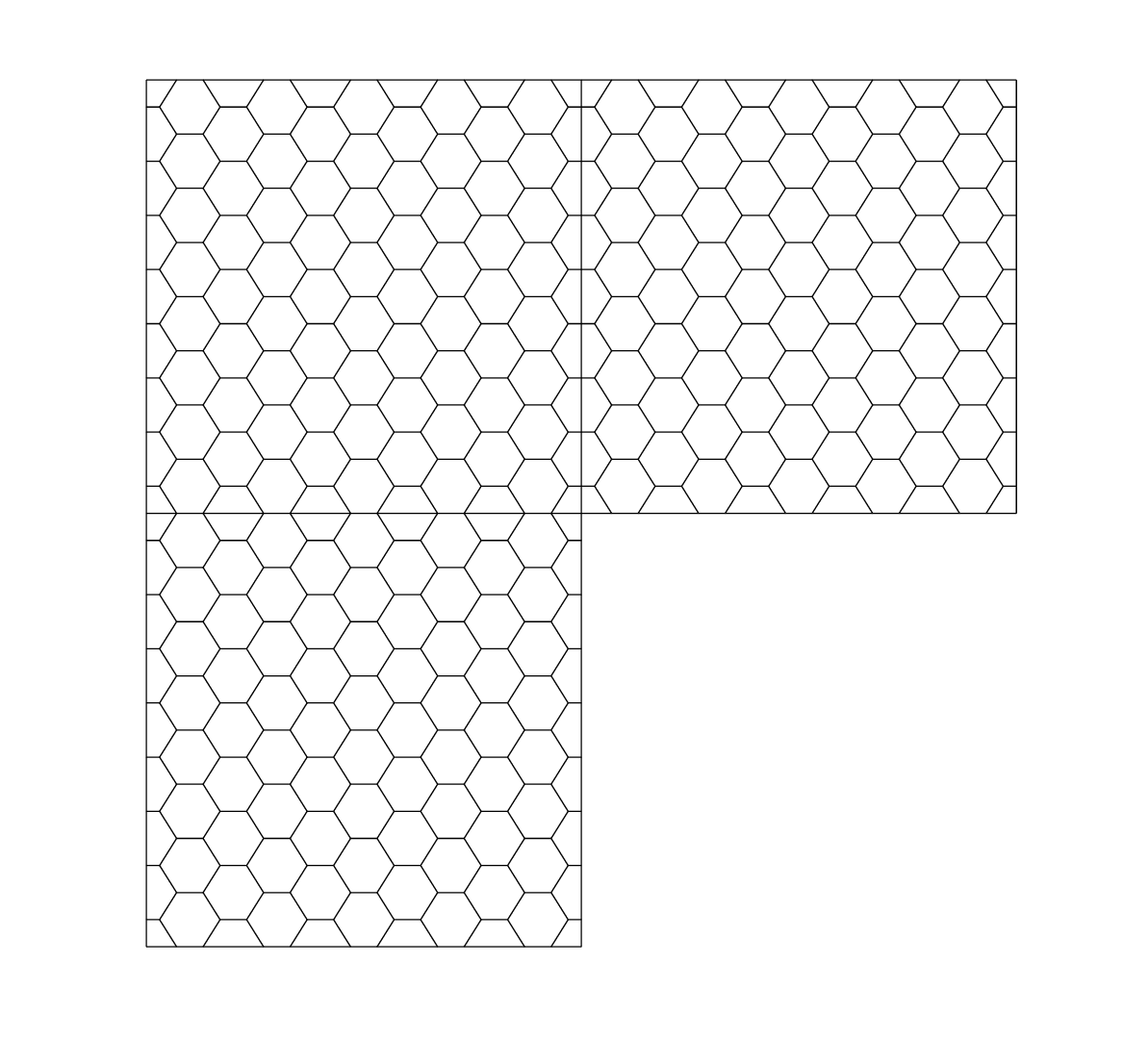}
\end{minipage}
\begin{minipage}{4.2cm}
\centering\includegraphics[height=4.1cm, width=4.1cm]{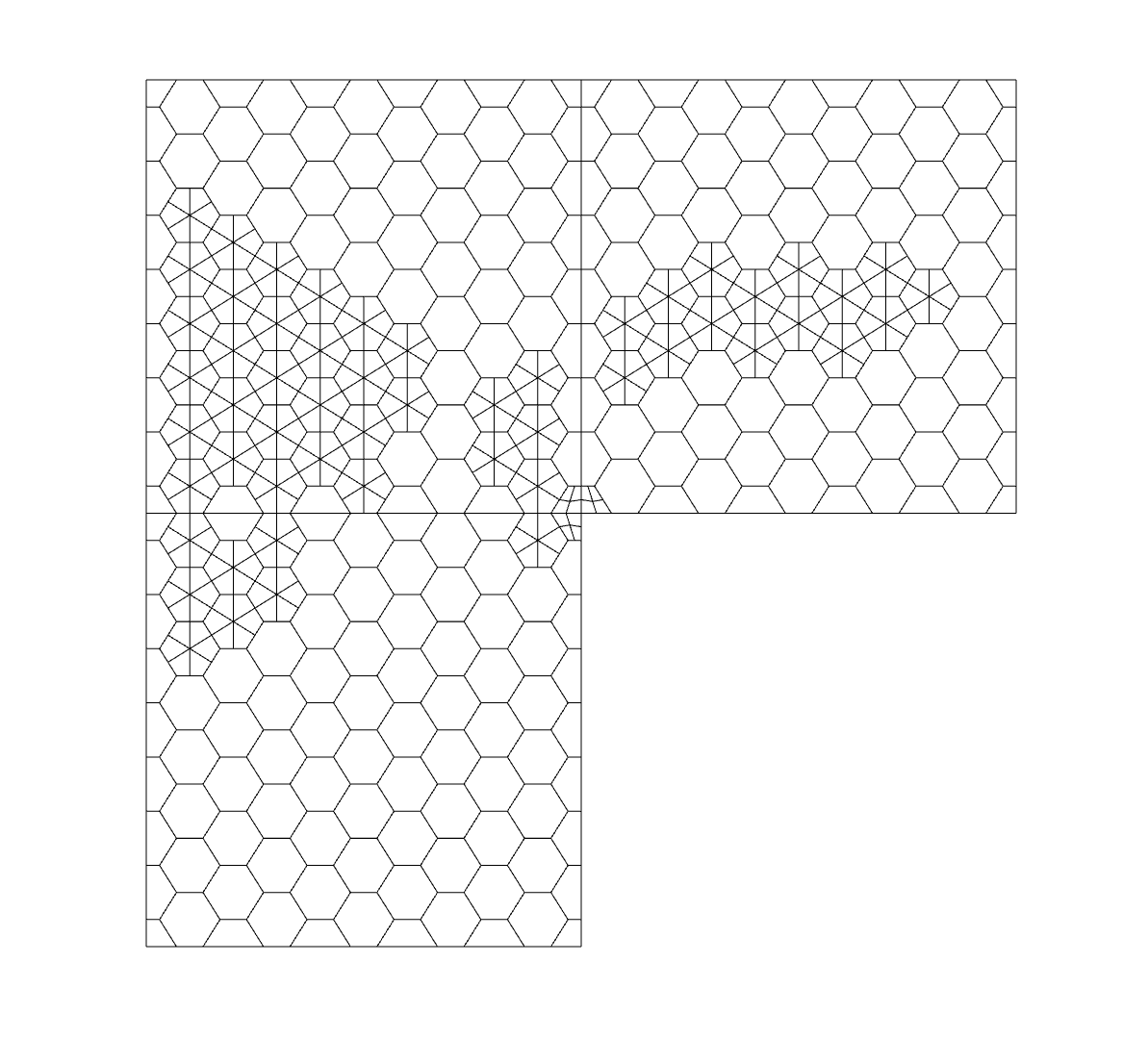}
\end{minipage}
\begin{minipage}{4.2cm}
\centering\includegraphics[height=4.1cm, width=4.1cm]{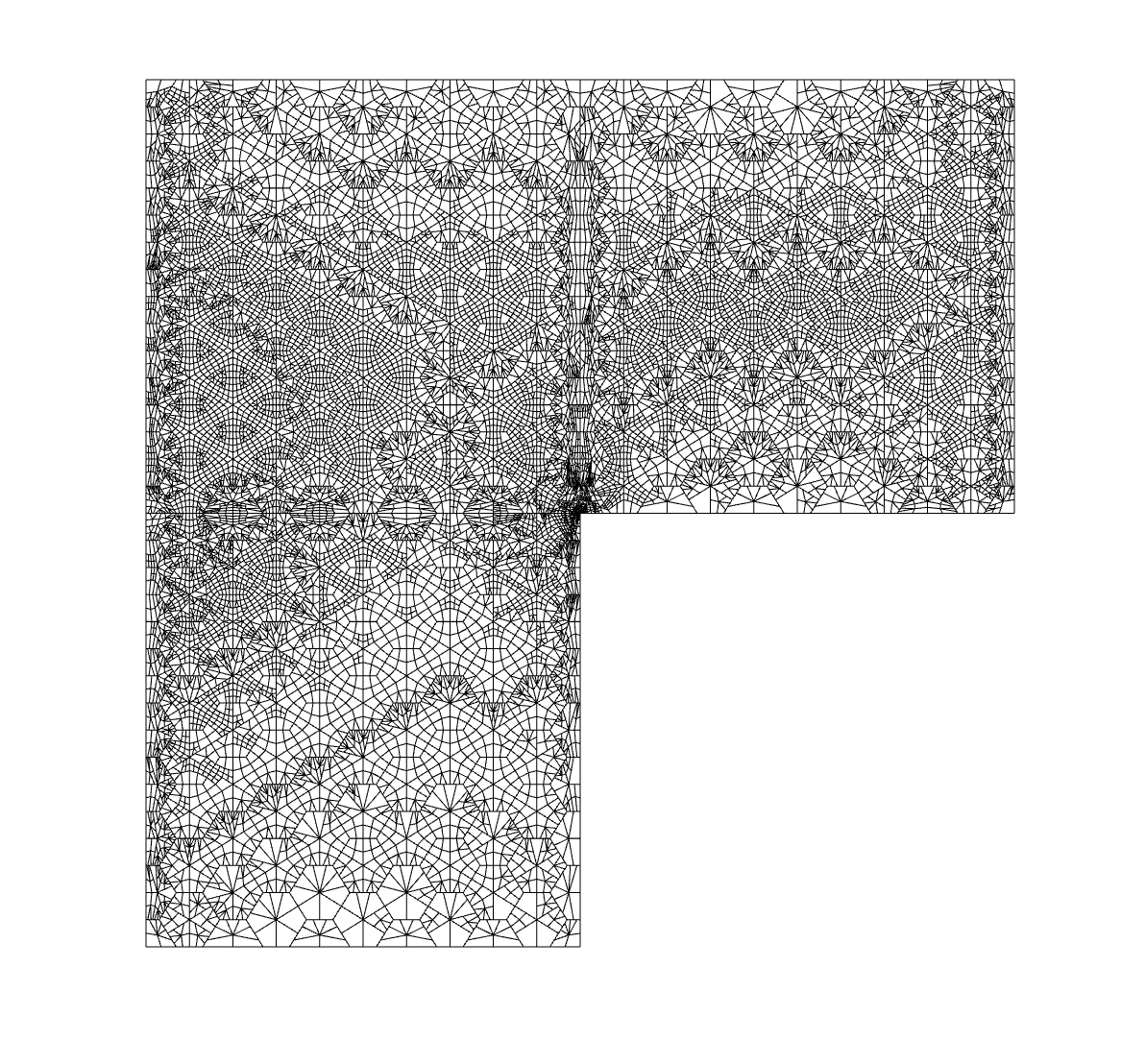}
\end{minipage}
\end{center}
\end{figure}

 \begin{figure}[h!]
\begin{center}
\caption{\label{fig:adaptiveVEMV} Test 1: Adaptively refined meshes obtained with VEM scheme at refinement steps 0, 1 and 8 initiated with an voronoi mesh (Adaptive VEMV).}
\begin{minipage}{4.2cm}
\centering\includegraphics[height=4.1cm, width=4.1cm]{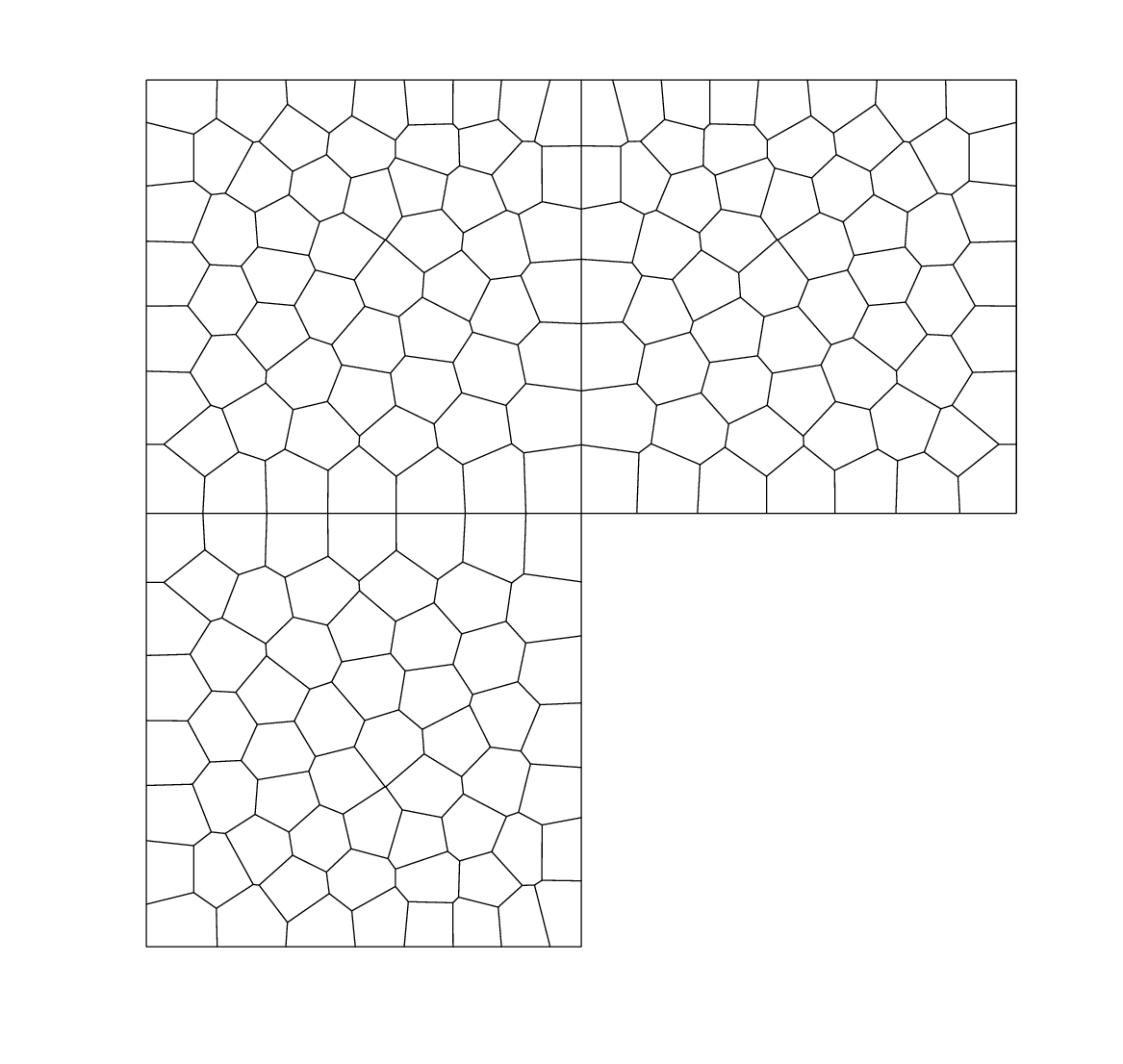}
\end{minipage}
\begin{minipage}{4.2cm}
\centering\includegraphics[height=4.1cm, width=4.1cm]{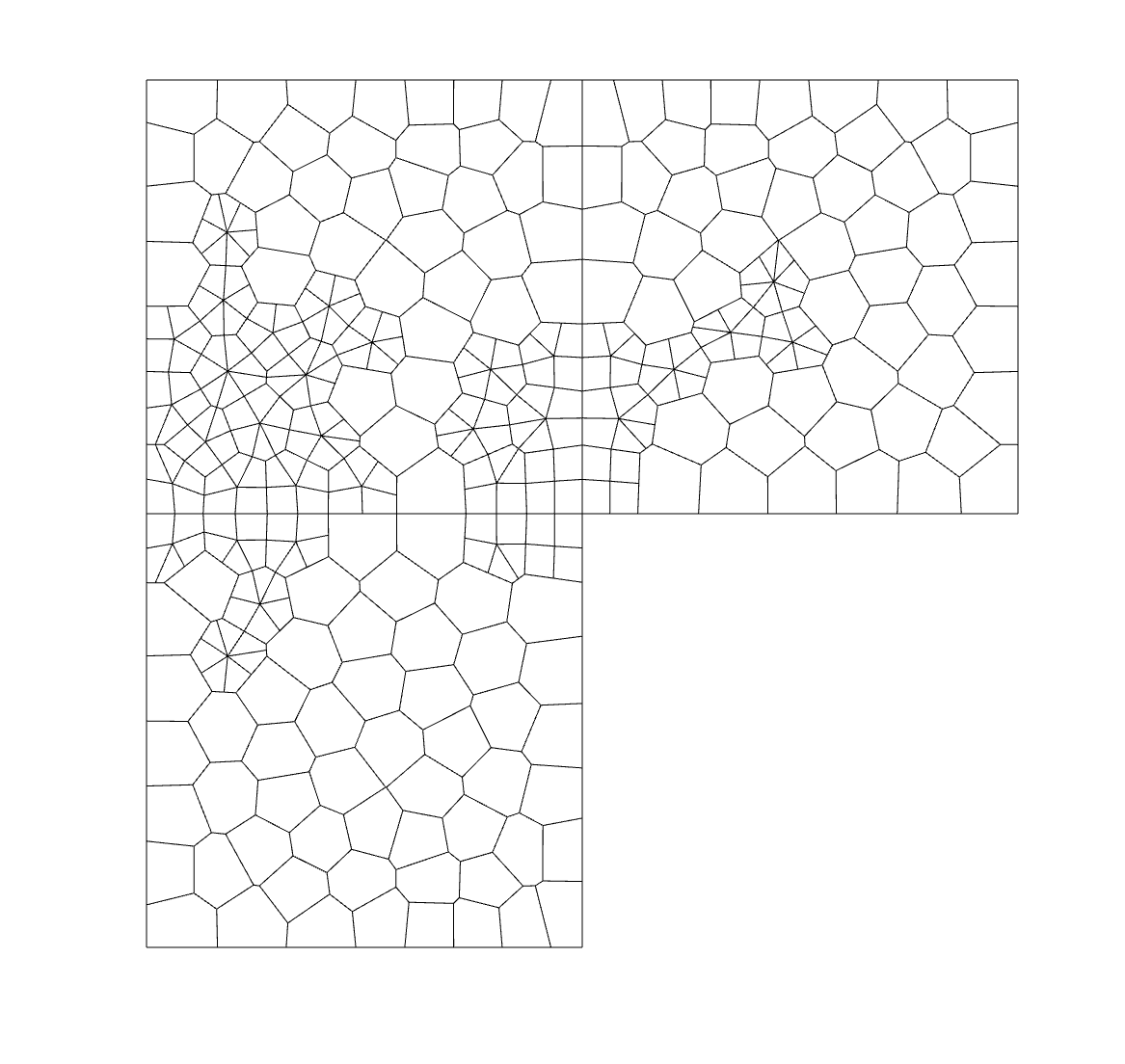}
\end{minipage}
\begin{minipage}{4.2cm}
\centering\includegraphics[height=4.1cm, width=4.1cm]{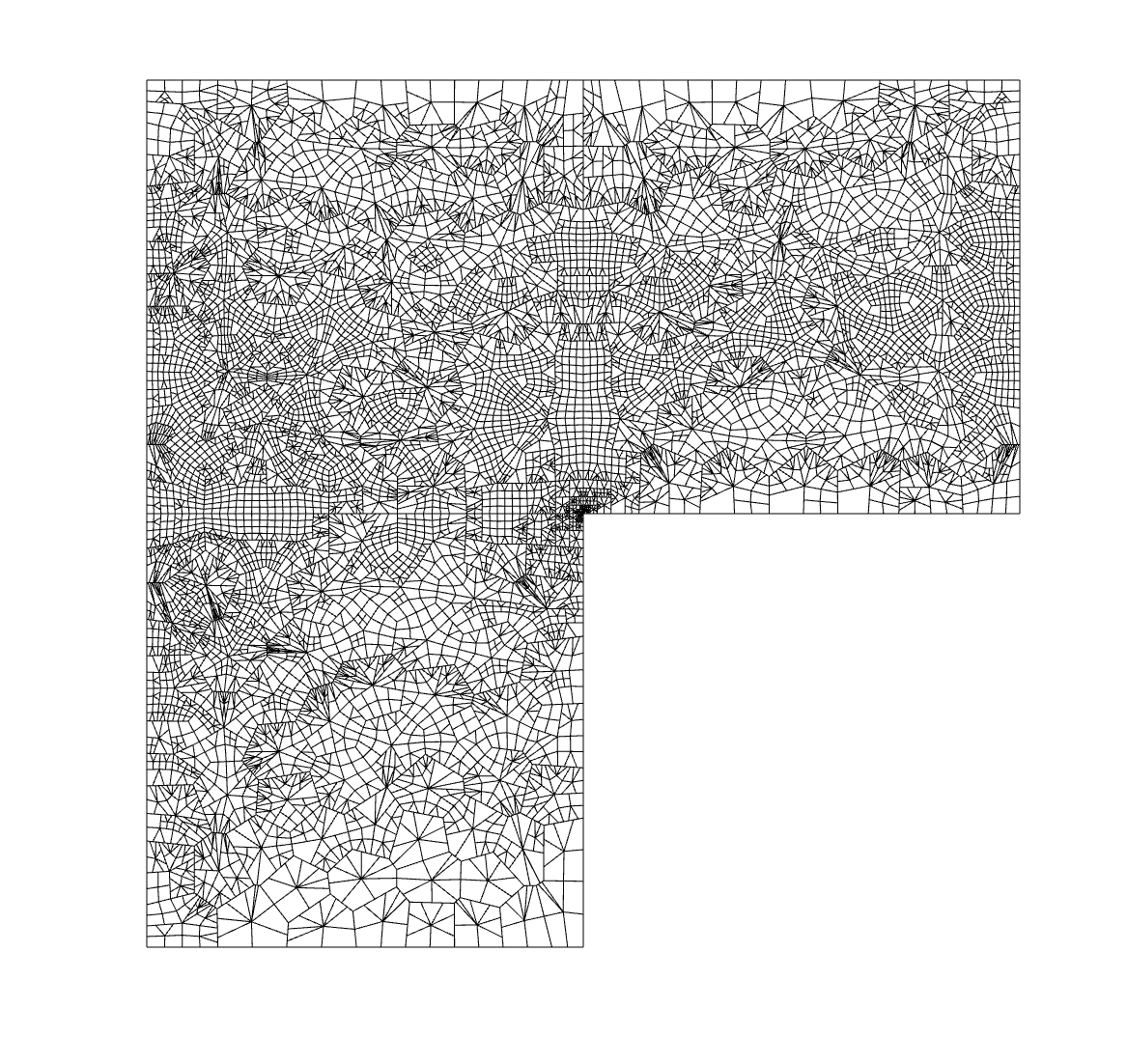}
\end{minipage}
\end{center}
\end{figure}
According to \cite{MR3133493} and as seen in Figure \ref{FIG:errorC}, using quasi-uniform meshes, the convergence rate for the eigenvalues should be $|\lambda_1-\lambda_{h,1}|=\mathcal{O}(h^{4/3})=\mathcal{O}(N^{-2/3})$ , where $N$ denotes the number of degrees of freedom. Then, the proposed a posteriori estimator  should be able to recover the optimal order 1 , when the adaptive refinement is performed near the singularity point. In Figure \ref{FIG:errorC} we present error curves where we observe that the uniform refinement leads to a convergence rate close to that predicted by the theory, while the adaptive VEM schemes allows us to recover the optimal order of convergence 1. Clearly, when we perform the refinements with triangles, the estimator works analogously as in then finite element method and it is clear that with triangles the geometrical singularities are identified in order to perform the necessary refinements.

 \begin{figure}[h!]
\begin{center}
\begin{minipage}{9.5cm}
\centering\includegraphics[height=8.9cm, width=9.0cm]{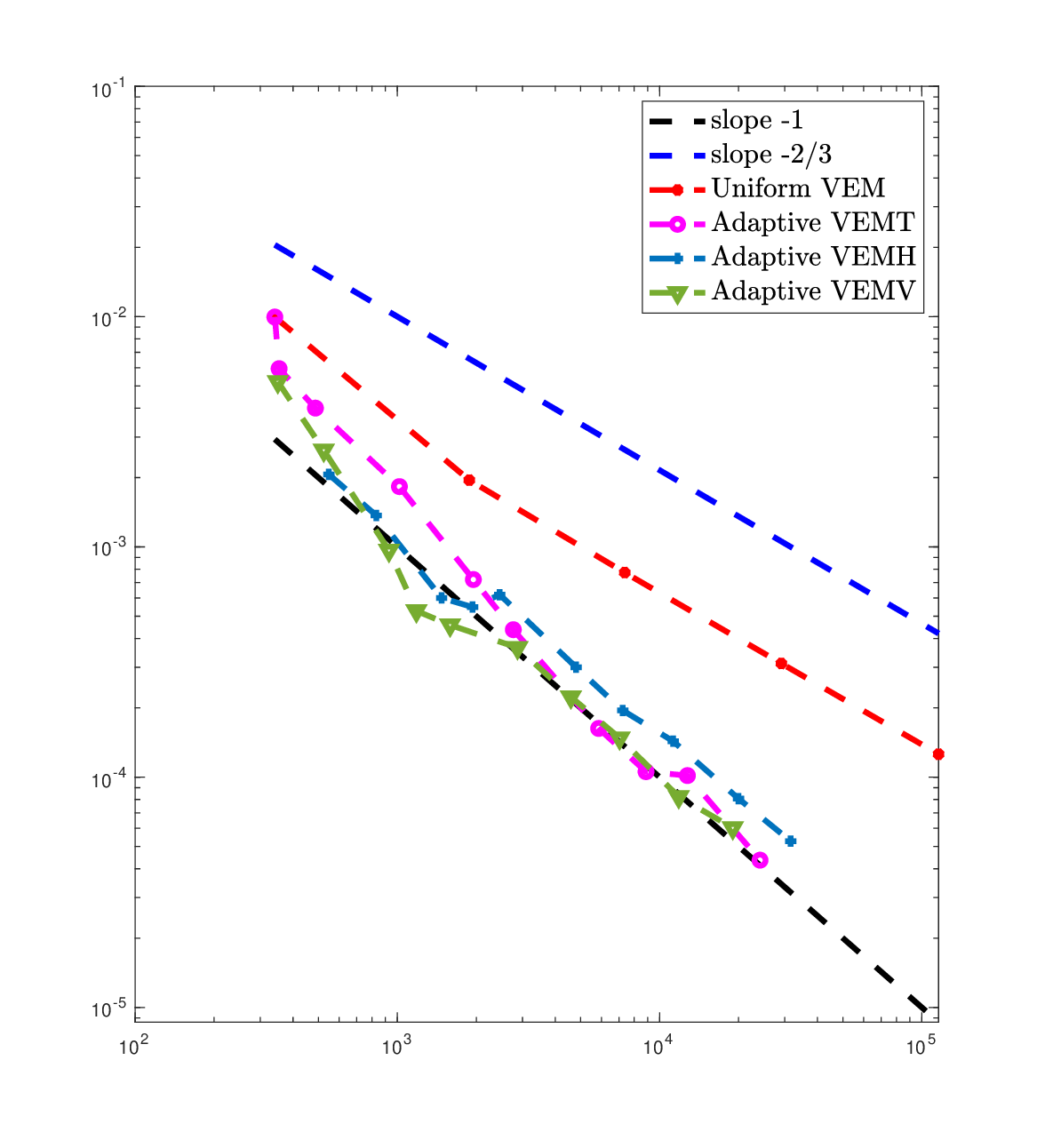}
\end{minipage}
\caption{Test 1. Error curves of $|\lambda_1-\lambda_{h,1}|$ for uniformly refined meshes (“Uniform VEM”), adaptively refined meshes for VEM with triangles (“Adaptive VEMT”), adaptively refined meshes for VEM with hexagons (“Adaptive VEMH”) and  adaptively refined meshes for VEM with voronoi (“Adaptive VEMV”).}
\label{FIG:errorC}
\end{center}
\end{figure}
We report in Table \ref{TABLA:3}, the estimators  $\boldsymbol{\eta}^2$ and  the effectivity indexes  $eff(\boldsymbol{\eta}):=\dfrac{|\l_{1}-\l_{h1}|}{\boldsymbol{\eta}^2}$ at each step of the adaptive VEM scheme. We include in the table the terms $R^2:=\sum_{E\in\CT_h}R_E^2$ which arise from the volumetric residuals, $\boldsymbol{\Theta}^{2}:=\sum_{E\in\CT_h}\Theta_{E}^{2}.$      which arise from the inconsistency of the VEM,  and $\boldsymbol{J}_h^{2}:=\sum_{E\in\CT_h}\left\{\sum_{\ell\in\CT_h}J_{\ell}^2\right\}$  which arise from the edge residuals. In the same way, Table \ref{TABLA:4} shows the analogous to Table \ref{TABLA:3}, but associated to the dual estimator.

\begin{table}[h!]
\begin{center}
\caption{Components of the error estimator and effectivity indexes on the adaptively refined meshes with VEMH.}
\resizebox{13cm}{!}{
\begin{tabular}{|c|c|c|c|c|c|c|c|c|}
\hline
$N$   & $\l_{h1}$ &  $R^2$   & $\boldsymbol{\Theta}^{2}$ & $\boldsymbol{J}_h^{2}$ &  $\boldsymbol{\eta}^2$ & $eff(\boldsymbol{\eta})$ \\
\hline
548  & 11.914   &2.0790e+00  & 9.0613e-02 &  3.9921e+00 &  6.1617e+00  & 3.3487e-04\\
   831&   11.906 &  1.4377e+00 &  7.0345e-02 &  2.5901e+00 &  4.0982e+00  & 3.3412e-04\\
   1475  & 11.897 &  4.8131e-01 &  3.4615e-02  & 9.7827e-01 &  1.4942e+00  & 4.0159e-04\\
   1935 & 11.896 &  3.4110e-01  & 2.2026e-02  & 6.7937e-01  & 1.0425e+00 &  5.2545e-04\\
   2456  & 11.897 &  3.1849e-01 &  1.7953e-02 &  5.9370e-01 &  9.3015e-01 &  6.6679e-04\\
   4802  & 11.893 &  1.5318e-01 &  1.1243e-02  & 3.0325e-01 &  4.6767e-01 &  6.4188e-04\\
   7236  & 11.892 &  9.6737e-02 &  6.9575e-03  & 1.9345e-01&   2.9714e-01&   6.5600e-04\\
   11269  & 11.891 &  7.0837e-02 &  4.9896e-03  & 1.3640e-01&   2.1223e-01&   6.7467e-04\\
   19984  & 11.891  & 3.8875e-02 &  3.4400e-03  & 7.8216e-02 &  1.2053e-01 &  6.6959e-04\\
   31615  & 11.890 &  2.3515e-02 &  2.2224e-03  & 4.9141e-02 &  7.4879e-02 &  7.0400e-04\\\hline

\end{tabular}}
\label{TABLA:3}
\end{center}
\end{table}
From Tables \ref{TABLA:3} and \ref{TABLA:4}  we observe that the effectivity indexes are bounded and far from zero. Also, the volumetric and edge residual terms are, roughly speaking, of the same order, none of them being asymptotically negliglible. 
\begin{table}[h!]
\begin{center}
\caption{Components of the error estimator and effectivity indexes on the adaptively refined meshes with VEMH.}
\resizebox{13cm}{!}{
\begin{tabular}{|c|c|c|c|c|c|c|c|c|}
\hline
$N$   & $\l_{h1}$ &  $R^{*2}$   & $\boldsymbol{\Theta^{*}}^{2}$ & $\boldsymbol{J}_h^{*2}$ &  $\boldsymbol{\eta}^{*2}$ & $eff(\boldsymbol{\eta}^*)$ \\
\hline
 548  & 11.914 &  2.5045e+00   &8.1576e-02  & 3.8235e+00&   6.4095e+00&   3.2193e-04\\
   831 &  11.906 &  1.3928e+00 &  5.2504e-02 &  1.8218e+00 &  3.2671e+00 &  4.1911e-04\\
   1475 &  11.897 &  5.8009e-01 &  2.8127e-02  & 8.9549e-01 &  1.5037e+00 &  3.9905e-04\\
   1935 &  11.896 &  4.2274e-01 &  1.8788e-02 &  6.9178e-01 &  1.1333e+00&   4.8334e-04\\
   2456 &  11.897 &  3.4216e-01 &  1.4299e-02 &  5.0855e-01 &  8.6501e-01 &  7.1700e-04\\
   4802 &  11.893 &  1.7841e-01&   8.2262e-03 &  2.7103e-01 &  4.5767e-01 &  6.5591e-04\\
   7236 &  11.892 &  1.1455e-01&   5.9439e-03 &  1.9062e-01 &  3.1112e-01 &  6.2652e-04\\
   11269 &  11.891 &  7.9144e-02 &  3.9980e-03  & 1.2025e-01 &  2.0339e-01 &  7.0399e-04\\
   19984&   11.891 &  4.4065e-02&   2.6448e-03  & 6.8238e-02 &  1.1495e-01&   7.0211e-04\\
   31615&   11.890 &  2.7508e-02&   1.7239e-03 &  4.5678e-02 &  7.4909e-02&   7.0372e-04\\\hline
\end{tabular}}
\label{TABLA:4}
\end{center}
\end{table}

Finally, Figure \ref{FIG:eigenfunction} shows the eigenfunctions  corresponding to the first lowest eigenvalue for the primal and dual problem at different levels of refinement.
 \begin{figure}[h!]
\begin{center}
\begin{minipage}{5.2cm}
\centering\includegraphics[height=5.1cm, width=5.1cm]{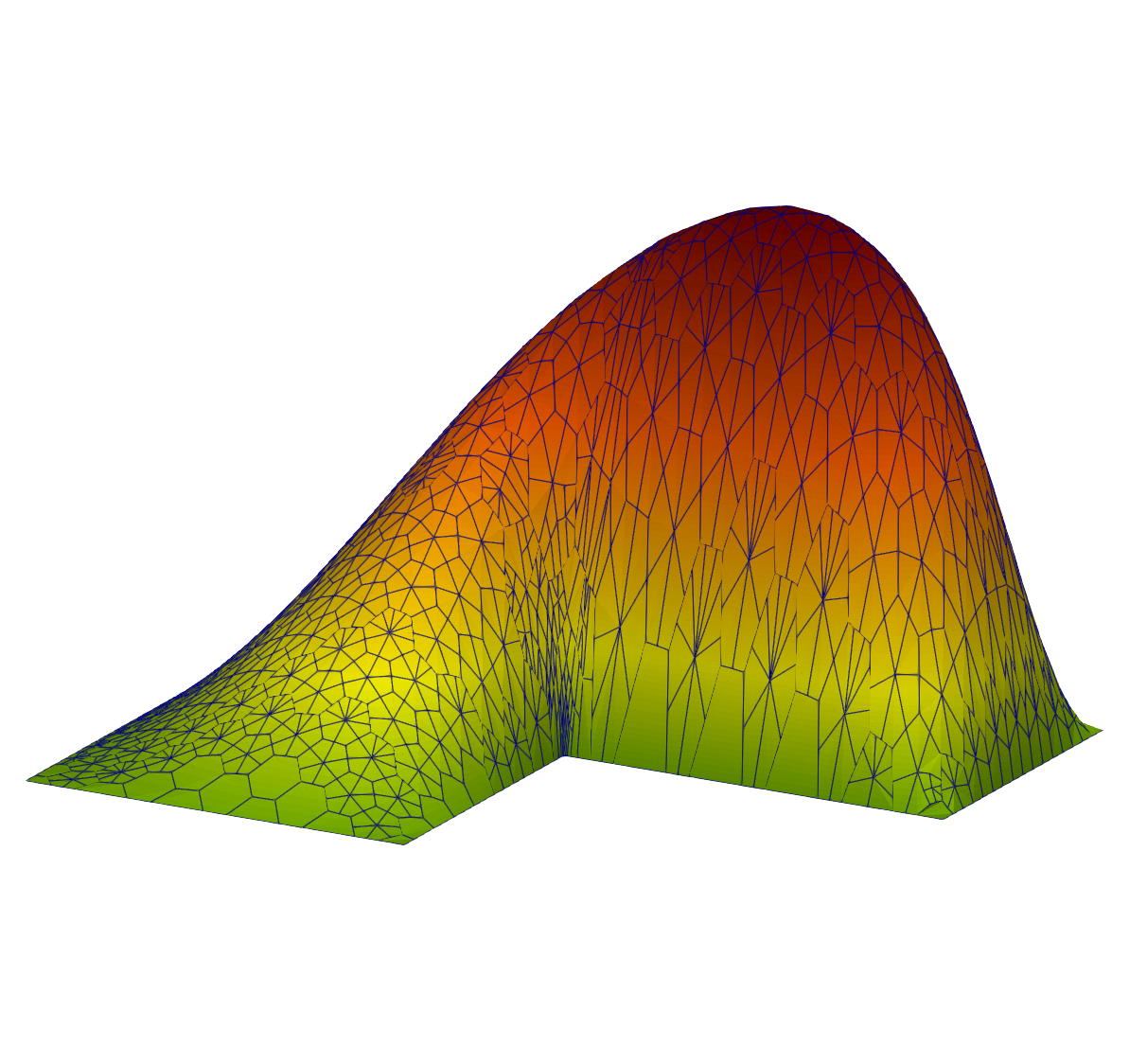}
\end{minipage}
\begin{minipage}{5.2cm}
\centering\includegraphics[height=5.1cm, width=5.1cm]{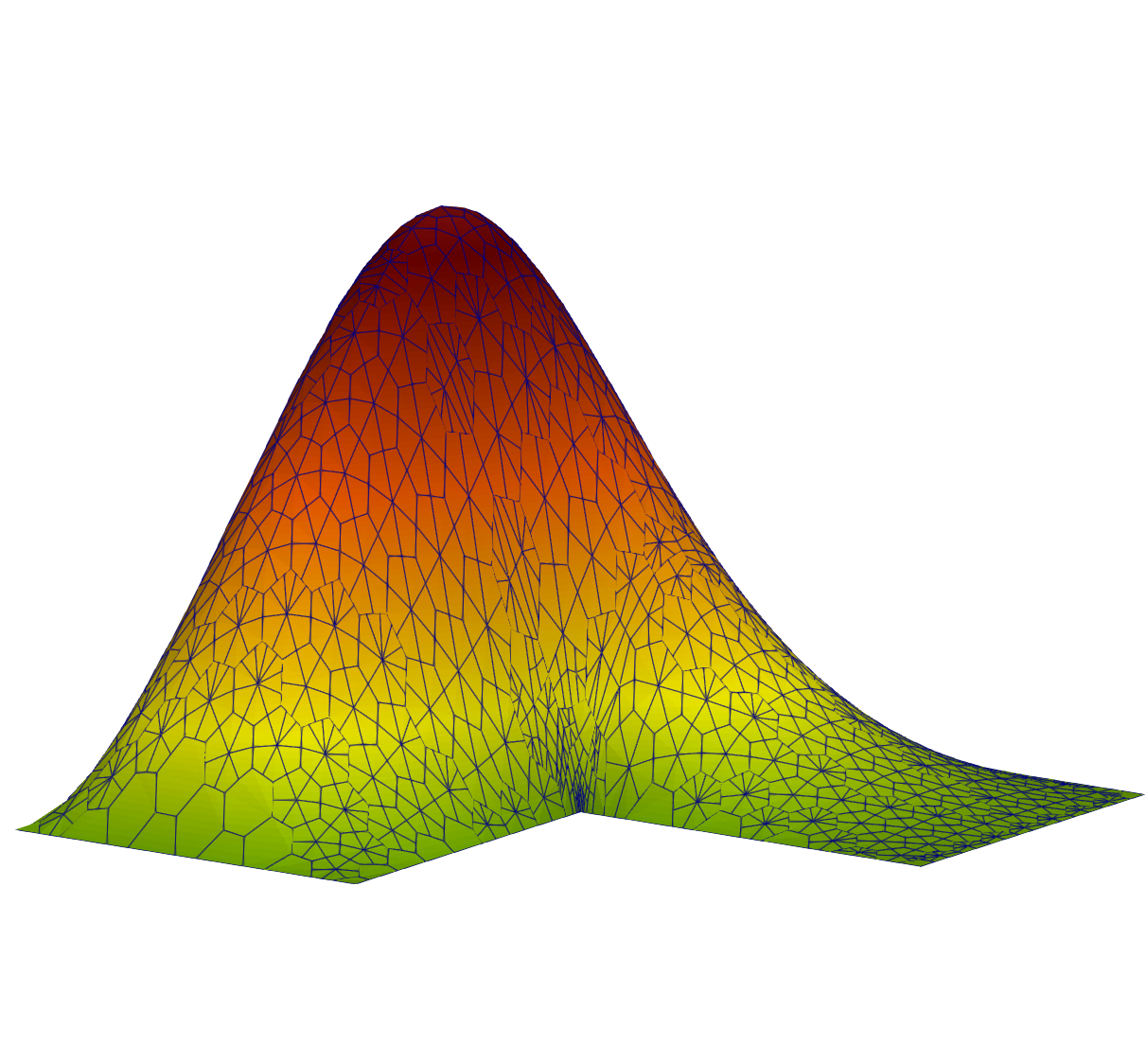}
\end{minipage}
\caption{Test 1. First eigenfunction for the primal problem  (left) and the dual problem (right).}
\label{FIG:eigenfunction}
\end{center}
\end{figure}

\subsection{Test 2: H-shaped domain}
For the following tests, we consider a two dimensional domain with for geometrical singularities,  which we call  the H-shaped domain. This domain may represent, for instance, the union of two pools containing fluids. Let us represent the H-shaped domain by $\Omega$ and its geometry is defined by 
More precisely, the geometry of this domain is given by
\begin{equation*}
\Omega:=\big\{ (0,3/2)\times(0,3) \big\} \setminus \big\{ \{ [1/2,1]\times [0,5/4]  \} \cup \{ [1/2,1]\times [15/8,3] \}  \big\}.
\end{equation*}
As in the L-shaped domain, the presence of four singularities  will lead, once again, to singular eigenfunctions with non sufficient regularity which is reflected on the convergence order on the computed a priori estimates. Hence, our proposed estimators must be capable of 
identify these singularities of the geometry and perform and adaptive refinement,
with different polygonal meshes, in order to recover optimal order of convergence. In the forthcoming snapshots we present the adaptive refinement of our estimator using different meshes.

\begin{figure}[h!]
\begin{center}
\begin{minipage}{4.2cm}
\centering\includegraphics[height=5.1cm, width=4.1cm]{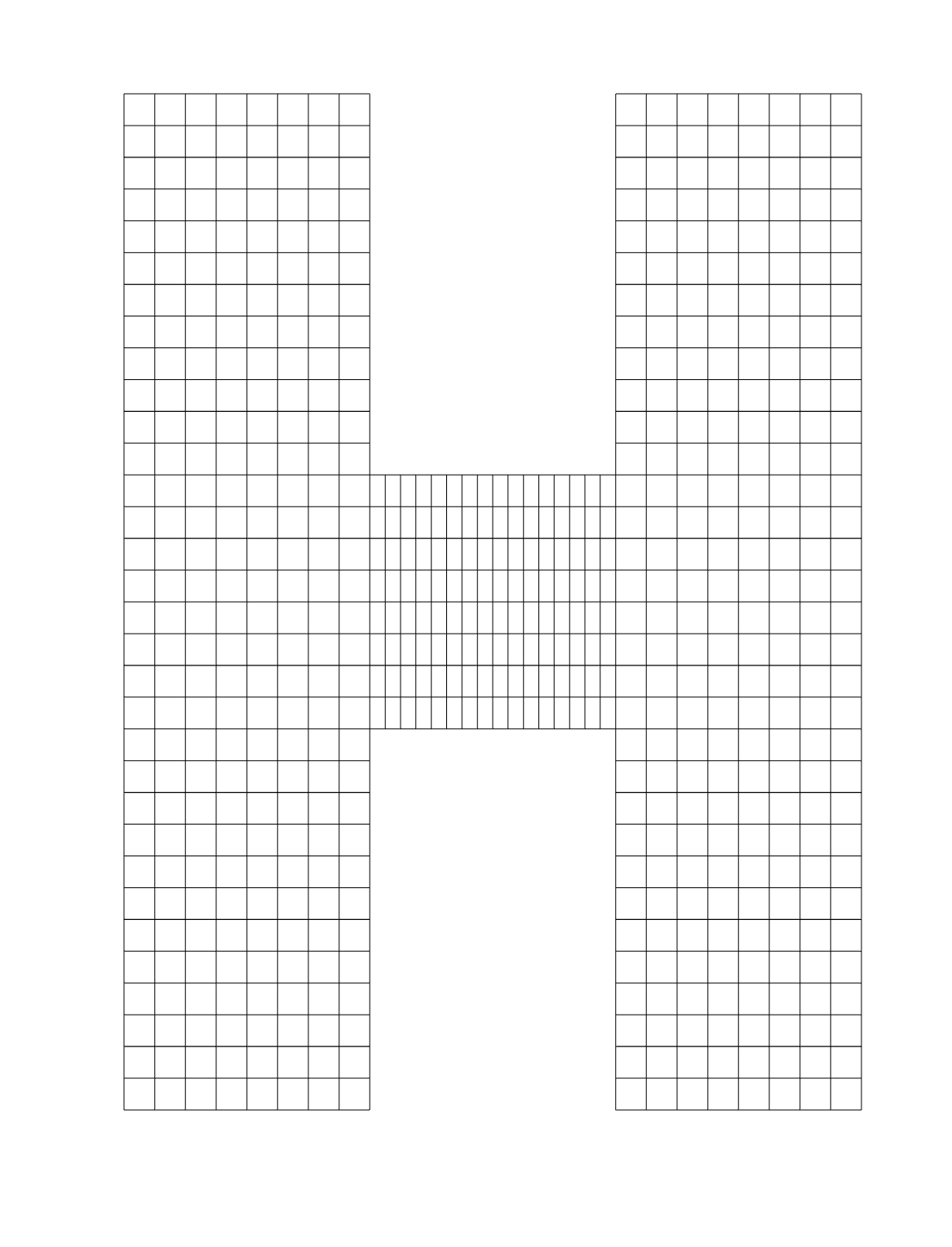}
\end{minipage}
\begin{minipage}{4.2cm}
\centering\includegraphics[height=5.1cm, width=4.1cm]{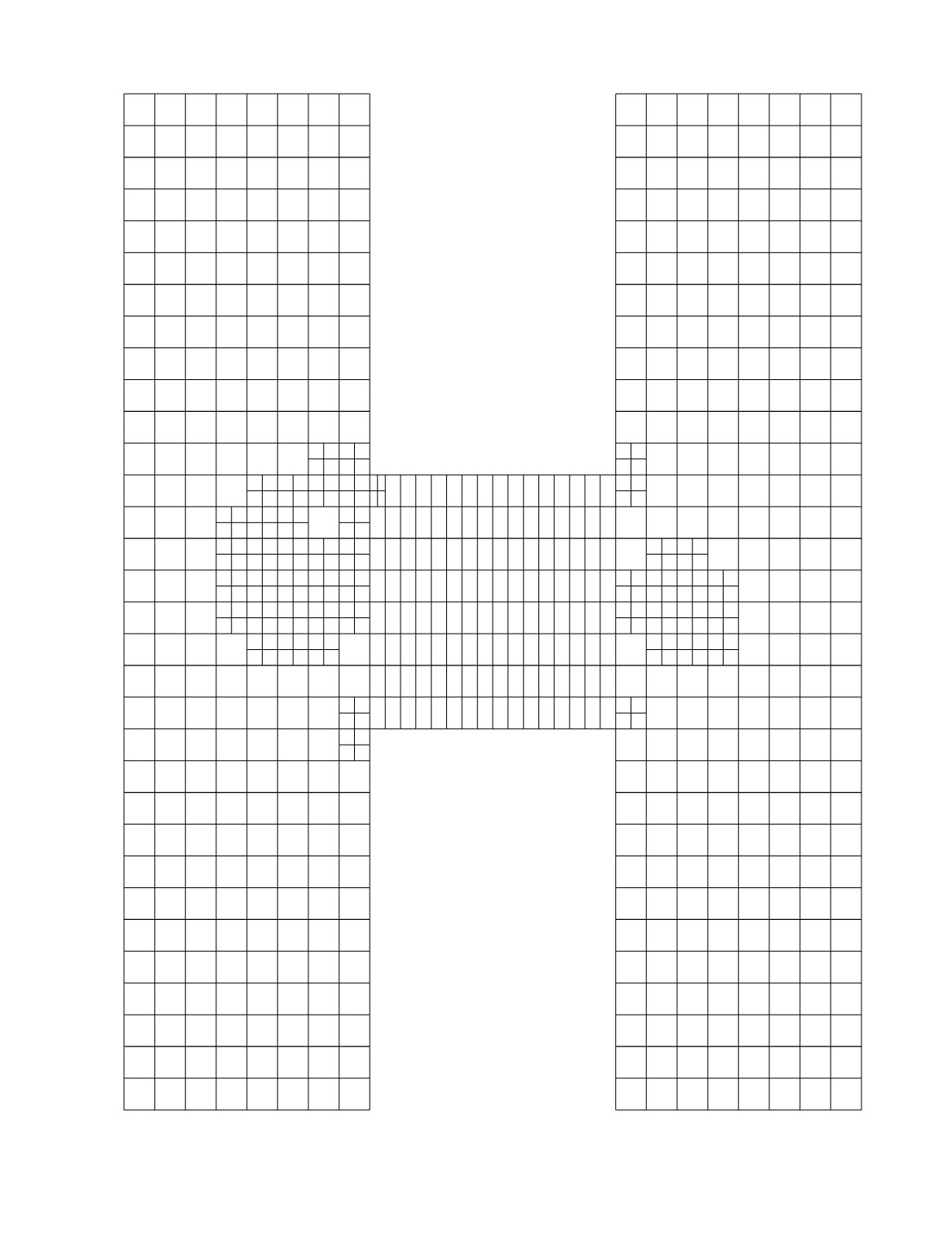}
\end{minipage}
\begin{minipage}{4.2cm}
\centering\includegraphics[height=5.1cm, width=4.1cm]{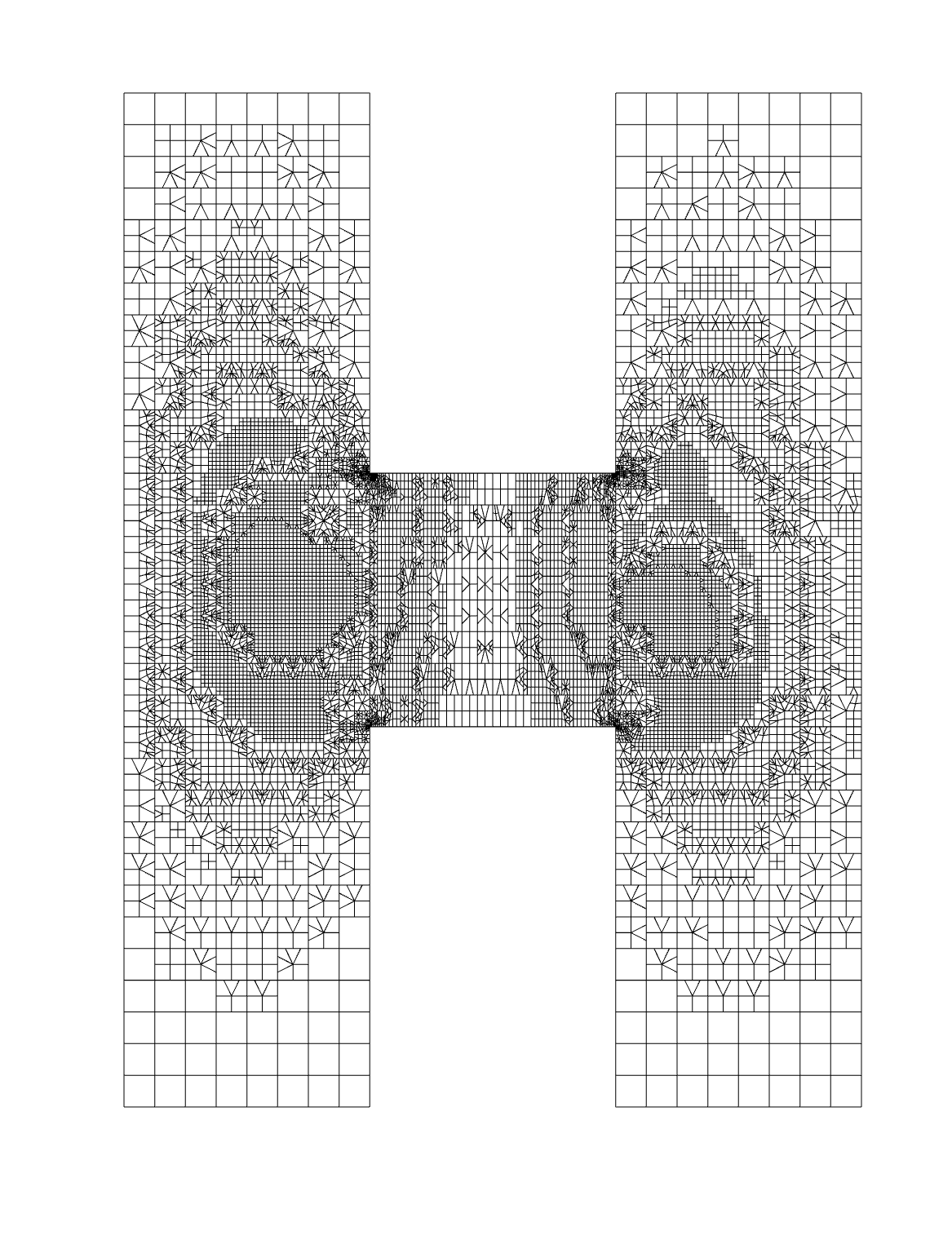}
\end{minipage}
\caption{\label{fig:adaptiveVEMHSQ} Test 2: Adaptively refined meshes obtained with VEM scheme at refinement steps 0, 1 and 8 initiated with a square mesh (Adaptive VEMSQ).}
\end{center}
\end{figure}

\begin{figure}[h!]
\begin{center}
\begin{minipage}{4.2cm}
\centering\includegraphics[height=5.1cm, width=4.1cm]{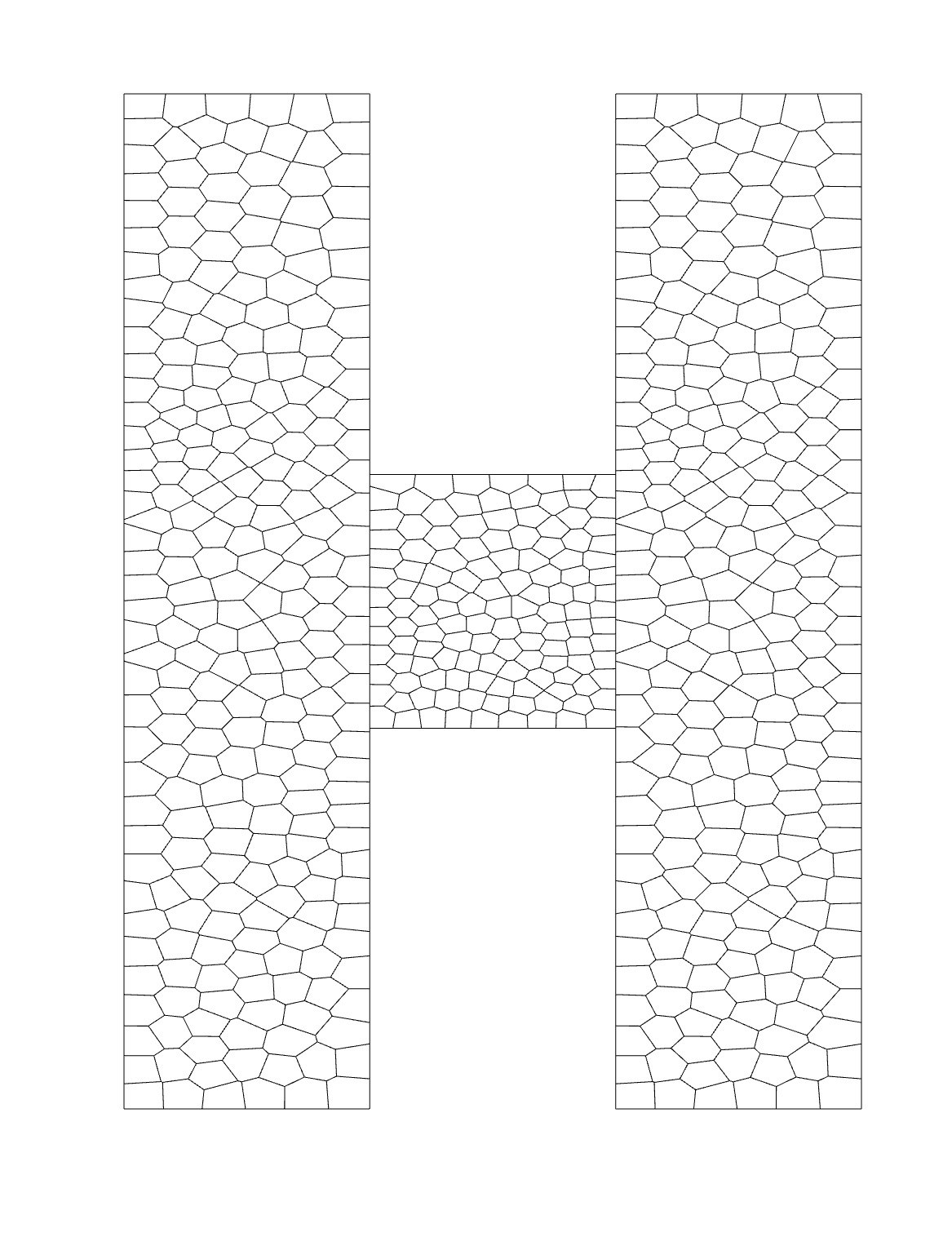}
\end{minipage}
\begin{minipage}{4.2cm}
\centering\includegraphics[height=5.1cm, width=4.1cm]{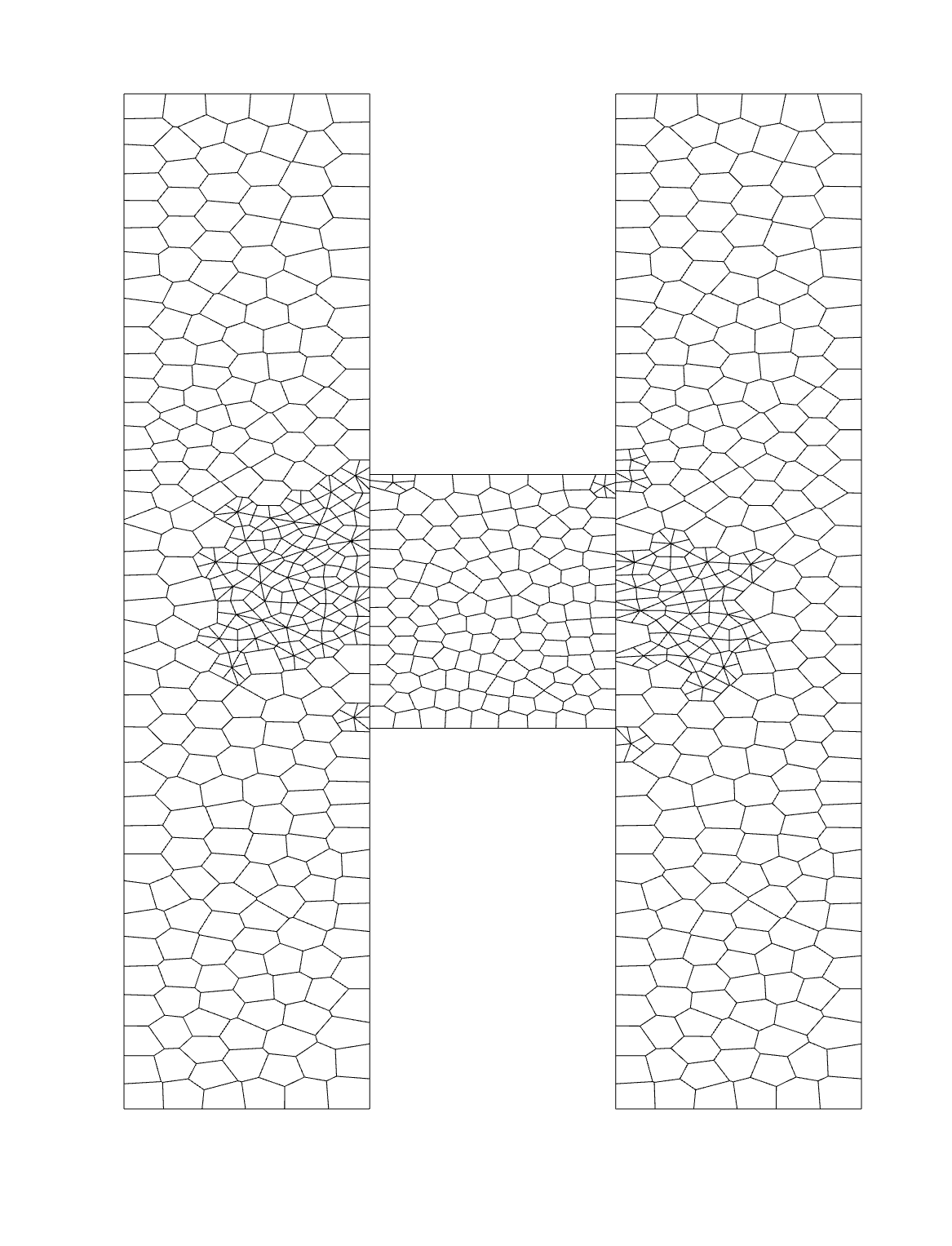}
\end{minipage}
\begin{minipage}{4.2cm}
\centering\includegraphics[height=5.1cm, width=4.1cm]{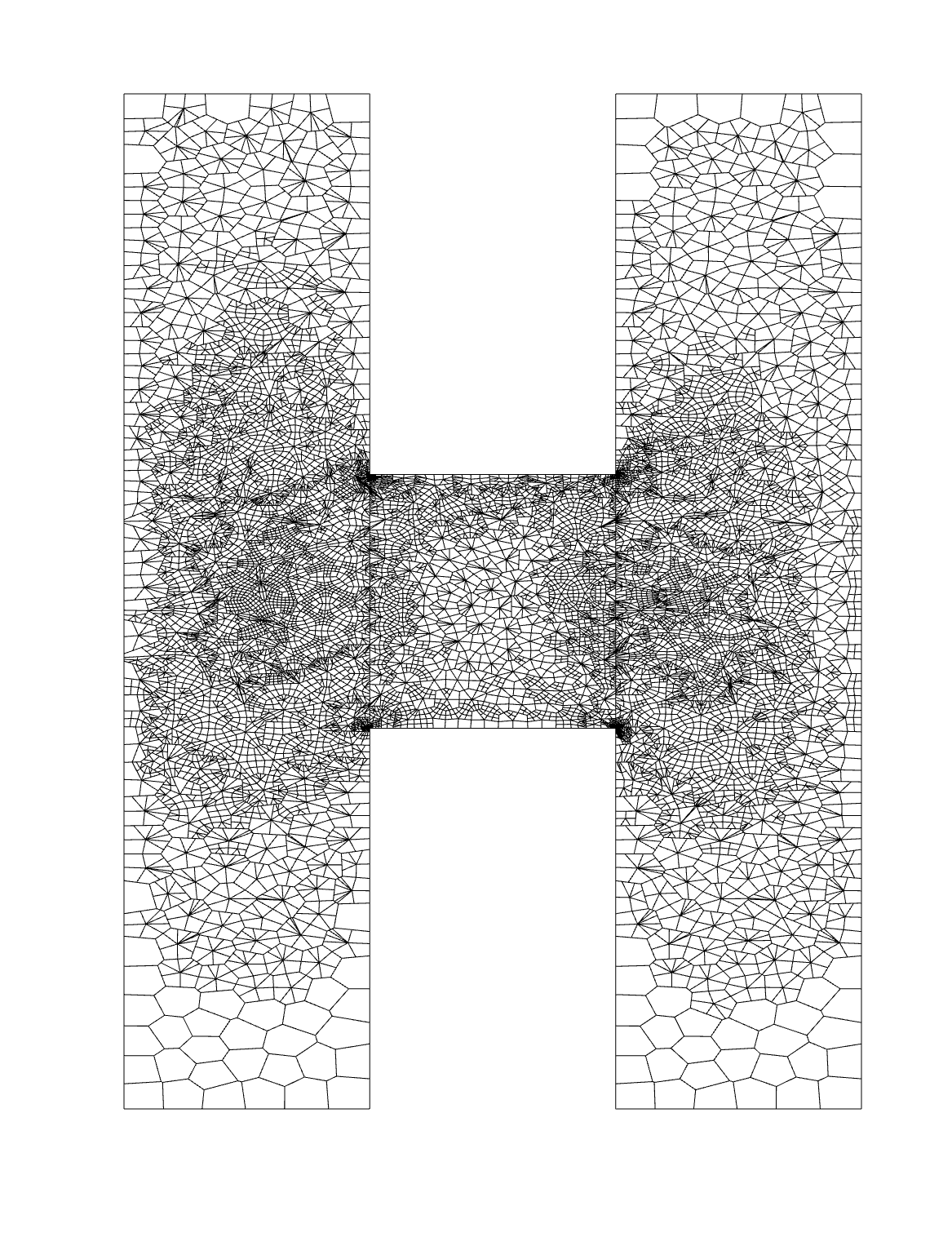}
\end{minipage}
\caption{\label{fig:adaptiveVEMHV} Test 2: Adaptively refined meshes obtained with VEM scheme at refinement steps 0, 1 and 8 initiated with a voronoi mesh (Adaptive VEMV).}
\end{center}
\end{figure}

From Figures \ref{fig:adaptiveVEMHSQ} and \ref{fig:adaptiveVEMHV}  we are able to observe that the refinements are precisely as we expect. On the other hand when polygonal meshes are considered, the estimator works perfectly when the refinements begin, in particular   when squares and voronoi meshes are considered as we observe on Figures \ref{fig:adaptiveVEMHSQ} and \ref{fig:adaptiveVEMHV} .

In Tables \ref{TABLA:11} and \ref{TABLA:12} we present the same information as in Tables \ref{TABLA:3} and \ref{TABLA:4} for this test. Similar conclusions to those of the previous test can be drawn from these tables.

\begin{table}[h!]
\begin{center}
\caption{Components of the error estimator and effectivity indexes on the adaptively refined meshes with VEMV.}
\resizebox{13cm}{!}{
\begin{tabular}{|c|c|c|c|c|c|c|c|c|}
\hline
$N$   & $\l_{h1}$ &  $R^2$   & $\boldsymbol{\Theta}^{2}$ & $\boldsymbol{J}_h^{2}$ &  $\boldsymbol{\eta}^2$ & $eff(\boldsymbol{\eta})$ \\
\hline
 1286  & 19.748e+01&   4.3111e+00&   3.7424e-01&   4.5192e-02  & 4.7305e+00&   1.6444e-03\\
   1567  & 19.717e+01&   2.6490e+00&   2.7764e-01&   3.6748e-02  & 2.9633e+00&   1.6740e-03\\
   2307  & 19.692e+01  & 1.1729e+00  & 1.7112e-01   &2.3501e-02&   1.3675e+00&   2.1471e-03\\
   3041&   19.682e+01&   8.4430e-01&   1.2692e-01&   2.1386e-02   &9.9260e-01&   2.5003e-03\\
   4108  & 19.675e+01  & 6.3121e-01&   7.6930e-02  & 7.6026e-03&   7.1575e-01  & 2.7884e-03\\
   5920 &  19.664e+01&   3.9877e-01  & 5.4853e-02&   4.7708e-03  & 4.5839e-01&   3.0739e-03\\
   8338   &19.657e+01  & 2.6452e-01&   4.1996e-02  & 3.2092e-03 &  3.0972e-01  & 3.1830e-03\\
   10493 &  19.655e+01  & 2.0562e-01&   3.5491e-02  & 2.6566e-03 &  2.4377e-01  & 3.4499e-03\\
   16770&   19.651e+01 &  1.3207e-01  & 2.4200e-02 &  2.0860e-03   &1.5835e-01 &  3.1906e-03\\
   17141  & 19.651e+01  & 1.2889e-01 &  2.3274e-02 &  2.0754e-03  & 1.5424e-01  & 3.2878e-03\\\hline
   
\end{tabular}}
\label{TABLA:11}
\end{center}
\end{table}

\begin{table}[h!]
\begin{center}
\caption{Components of the error estimator and effectivity indexes on the adaptively refined meshes with VEMV.}
\resizebox{13cm}{!}{
\begin{tabular}{|c|c|c|c|c|c|c|c|c|}
\hline
$N$   & $\l_{h1}$ &  $R*^2$   & $\boldsymbol{\Theta^{*}}^{2}$ & $\boldsymbol{J}_h^{*2}$ &  $\boldsymbol{\eta}*^2$ & $eff(\boldsymbol{\eta}^{*})$ \\
\hline
1286 &  19.748&   4.1233e+00&   4.2172e-01&   8.2429e+00 &  1.2788e+01&   6.0830e-04\\
   1567&   19.717&   2.2769e+00&   2.7066e-01&   5.4508e+00&   7.9983e+00&   6.2019e-04\\
   2307  & 19.692  & 1.0692e+00&   1.5419e-01&   2.8333e+00&   4.0566e+00&   7.2381e-04\\
   3041&   19.682 &  7.9111e-01   &1.0782e-01   &2.0435e+00   &2.9424e+00   &8.4348e-04\\
   4108  & 19.675   &5.7602e-01&   7.0627e-02 &  1.4768e+00&   2.1234e+00&   9.3988e-04\\
   5920 &  19.664&   3.4994e-01  & 5.1348e-02   &1.0383e+00  & 1.4396e+00  & 9.7878e-04\\
   8338 &  19.657  & 2.3910e-01 &  3.8194e-02&   7.4518e-01&   1.0225e+00&   9.6419e-04\\
   10493 &  19.655  & 1.9210e-01 &  3.1877e-02&   6.1777e-01&   8.4175e-01 &  9.9905e-04\\
   16770   &19.651 &  1.1519e-01 &  2.0240e-02  & 3.9690e-01  & 5.3233e-01&   9.4912e-04\\
   17141   &19.651   &1.1363e-01   &1.9662e-02  & 3.9010e-01   &5.2339e-01  & 9.6888e-04\\\hline
 \end{tabular}}
\label{TABLA:12}
\end{center}
\end{table}
Figure \ref{FIG:errorH} shows a logarithmic plot of the errors between the calculated approximations of the second smallest positive eigenvalue and the `'exact" one, versus the number of degrees of freedom $N$ of the meshes. The exact value of the second eigenvalue is obtained by using the least squares fit. The figure shows the results obtained with “uniform” meshes and with adaptively refined meshes and shows how the optimal order of convergence is recovered.
 \begin{figure}[h!]
\begin{center}
\begin{minipage}{9.5cm}
\centering\includegraphics[height=8.9cm, width=9.0cm]{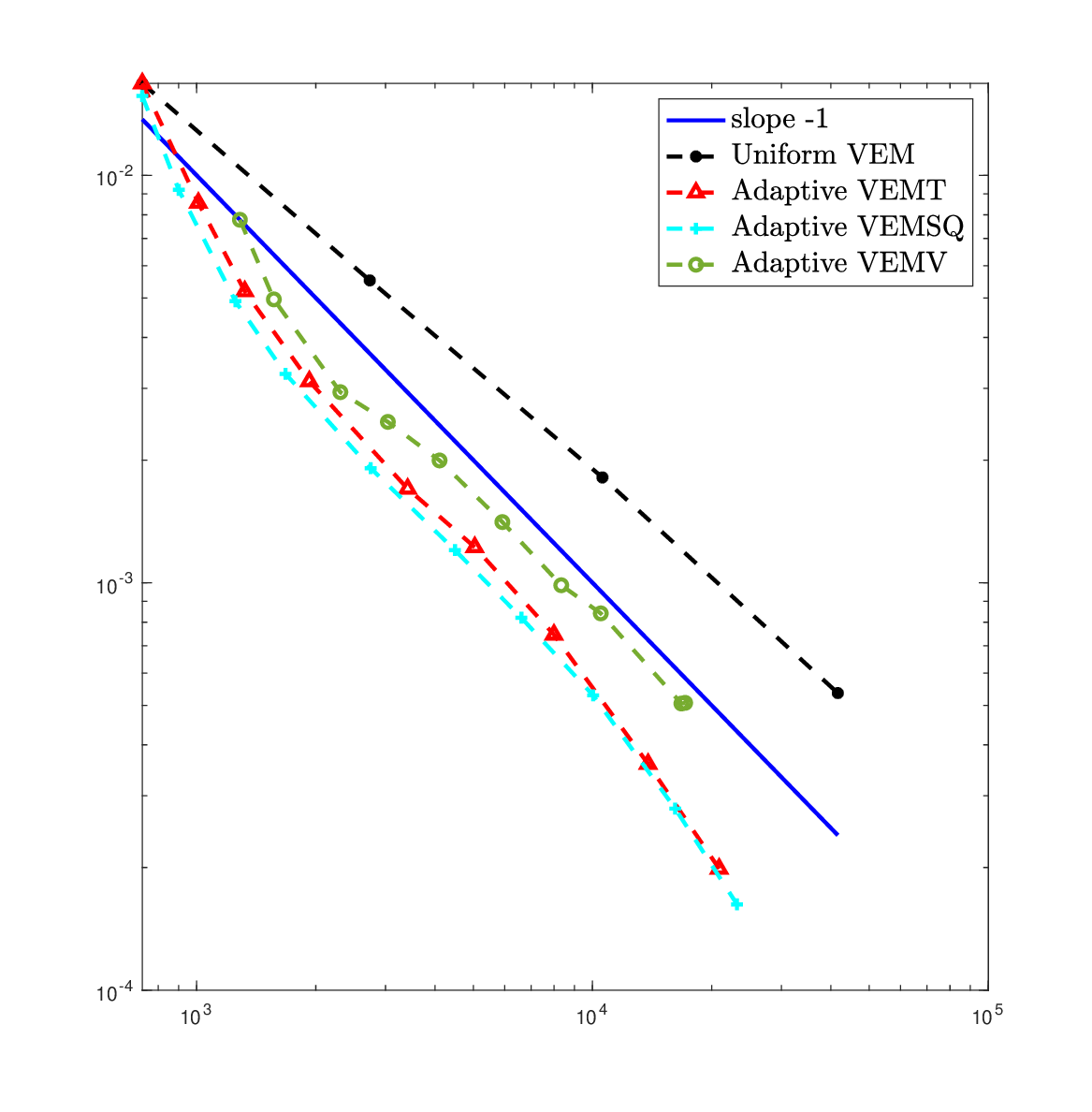}
\end{minipage}
\caption{Test 1. Error curves of $|\lambda_2-\lambda_{h,2}|$ for uniformly refined meshes (“Uniform VEM”), adaptively refined meshes for VEM  with triangles (“Adaptive VEMT”), adaptively refined meshes for VEM  with squares (“Adaptive VEMSQ”) and  adaptively refined meshes for VEM with voronoi (“Adaptive VEMV”).}
\label{FIG:errorH}
\end{center}
\end{figure}

We end the report of results related of the H-shaped domain with plots of the second computed eigenfunctions for the primal and dual eigenvalue problems. On this case, the eigenfunctions have been obtained with a voronoi mesh as we observe in Figure \ref{FIG:eigenfunctionH}. We notice the similarity between the solution of the primal and dual eigenvalue problems, where the singularities emerge due the geometrical definition of the domain.
\begin{figure}[h!]
\begin{center}
\begin{minipage}{6.2cm}
\centering\includegraphics[height=6.1cm, width=6.1cm]{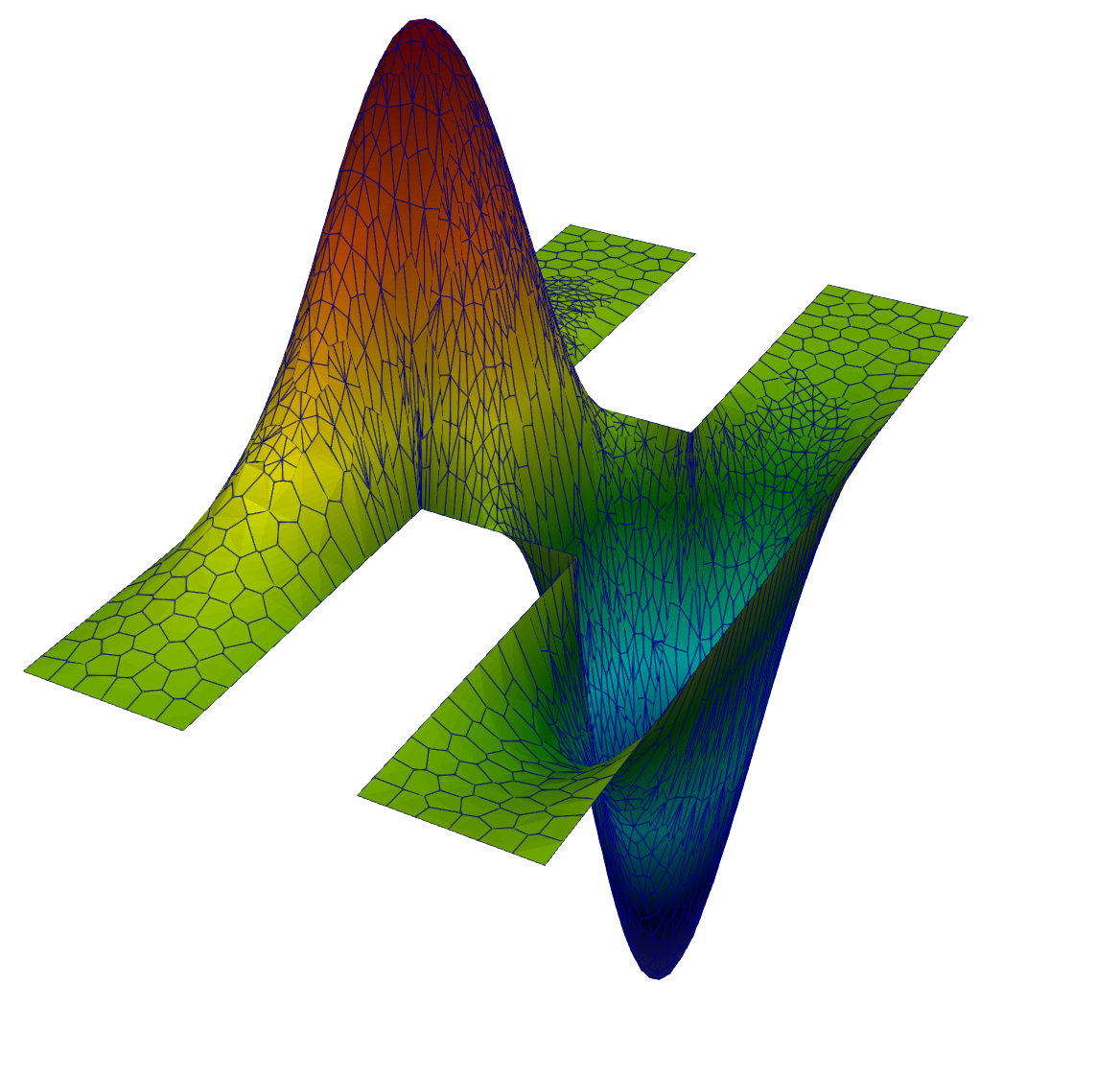}
\end{minipage}
\begin{minipage}{6.2cm}
\centering\includegraphics[height=6.1cm, width=6.1cm]{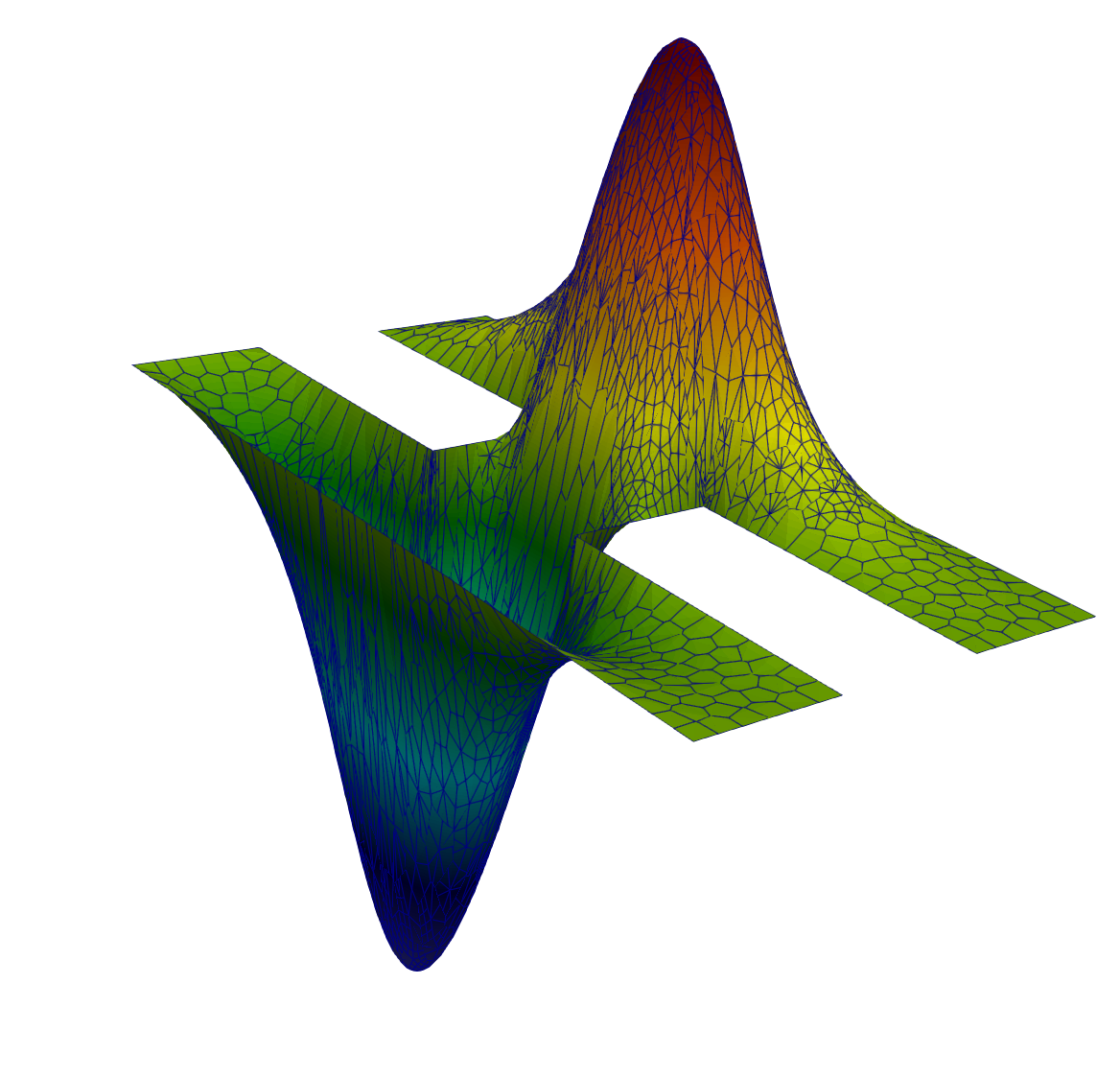}
\end{minipage}
\caption{Test 2. Second eigenfunction for the primal problem  (left) and the dual problem (right).}
\label{FIG:eigenfunctionH}
\end{center}
\end{figure}


\bibliographystyle{siamplain}
\bibliography{references}

\begin{thebibliography}{10}

\bibitem{adak2022vem}
{\sc D.~Adak, F.~Lepe, and G.~Rivera}, {\em Vem approximation for the stokes
  eigenvalue problem: a priori and a posteriori error analysis}, 2022,
  \url{https://arxiv.org/abs/2212.01961}.

\bibitem{AABMR13}
{\sc B.~Ahmad, A.~Alsaedi, F.~Brezzi, L.~D. Marini, and A.~Russo}, {\em
  Equivalent projectors for virtual element methods}, Comput. Math. Appl., 66
  (2013), pp.~376--391, \url{https://doi.org/10.1016/j.camwa.2013.05.015}.

\bibitem{MR1885308}
{\sc M.~Ainsworth and J.~T. Oden}, {\em A posteriori error estimation in finite
  element analysis}, Pure and Applied Mathematics (New York),
  Wiley-Interscience [John Wiley \& Sons], New York, 2000,
  \url{https://doi.org/10.1002/9781118032824}.

\bibitem{ABM2022}
{\sc P.~F. Antonietti, L.~Beir\~{a}o~da Veiga, and G.~Manzini}, {\em The
  Virtual Element Method and its Applications}, vol.~31, SEMA SIMAI Springer
  Series, 2022.

\bibitem{BO}
{\sc I.~Babu\v{s}ka and J.~Osborn}, {\em Eigenvalue problems}, Handb. Numer.
  Anal., II, North-Holland, Amsterdam, 1991.

\bibitem{BBCMMR2013}
{\sc L.~Beir\~{a}o~da Veiga, F.~Brezzi, A.~Cangiani, G.~Manzini, L.~D. Marini,
  and A.~Russo}, {\em Basic principles of virtual element methods}, Math.
  Models Methods Appl. Sci., 23 (2013), pp.~199--214,
  \url{https://doi.org/10.1142/S0218202512500492}.

\bibitem{beiraosec}
{\sc L.~Beir\~{a}o~da Veiga, F.~Brezzi, L.~D. Marini, and A.~Russo}, {\em
  Virtual element method for general second-order elliptic problems on
  polygonal meshes}, Mathematical Models and Methods in Applied Sciences, 26
  (2016), pp.~729--750, \url{https://doi.org/10.1142/S0218202516500160}.

\bibitem{MR3342219}
{\sc L.~Beir\~{a}o~da Veiga and G.~Manzini}, {\em Residual {\it a posteriori}
  error estimation for the virtual element method for elliptic problems}, ESAIM
  Math. Model. Numer. Anal., 49 (2015), pp.~577--599,
  \url{https://doi.org/10.1051/m2an/2014047}.

\bibitem{BS-2008}
{\sc S.~C. Brenner and L.~R. Scott}, {\em The Mathematical Theory of Finite
  Element Methods}, Springer, New York, 2008.

\bibitem{MR3719046}
{\sc A.~Cangiani, E.~H. Georgoulis, T.~Pryer, and O.~J. Sutton}, {\em A
  posteriori error estimates for the virtual element method}, Numer. Math., 137
  (2017), pp.~857--893, \url{https://doi.org/10.1007/s00211-017-0891-9}.

\bibitem{DNR1}
{\sc J.~Descloux, N.~Nassif, and J.~Rappaz}, {\em On spectral approximation.
  part 1. the problem of convergence}, ESAIM: Mathematical Modelling and
  Numerical Analysis-Mod{\'e}lisation Math{\'e}matique et Analyse
  Num{\'e}rique, 12 (1978), pp.~97--112,
  \url{https://doi.org/10.1051/m2an/1978120200971}.

\bibitem{DNR2}
{\sc J.~Descloux, N.~Nassif, and J.~Rappaz}, {\em On spectral approximation.
  part 2. error estimates for the galerkin method}, RAIRO. Analyse
  num{\'e}rique, 12 (1978), pp.~113--119,
  \url{https://doi.org/10.1051/m2an/1978120201131}.

\bibitem{MR3867390}
{\sc F.~Gardini and G.~Vacca}, {\em Virtual element method for second-order
  elliptic eigenvalue problems}, IMA J. Numer. Anal., 38 (2018),
  pp.~2026--2054, \url{https://doi.org/10.1093/imanum/drx063}.

\bibitem{MR3133493}
{\sc J.~Gedicke and C.~Carstensen}, {\em A posteriori error estimators for
  convection-diffusion eigenvalue problems}, Comput. Methods Appl. Mech.
  Engrg., 268 (2014), pp.~160--177,
  \url{https://doi.org/10.1016/j.cma.2012.09.018}.

\bibitem{MR4550402}
{\sc F.~Lepe, D.~Mora, G.~Rivera, and I.~Vel\'{a}squez}, {\em A posteriori
  virtual element method for the acoustic vibration problem}, Adv. Comput.
  Math., 49 (2023), pp.~Paper No. 10, 29,
  \url{https://doi.org/10.1007/s10444-022-10003-1}.

\bibitem{lepe2023vem}
{\sc F.~Lepe and G.~Rivera}, {\em Vem discretization allowing small edges for
  the reaction-convection-diffusion equation: source and spectral problems},
  2023, \url{https://arxiv.org/abs/2302.02240}.

\bibitem{MR4050542}
{\sc D.~Mora and G.~Rivera}, {\em {\it {A} priori} and {\it a posteriori} error
  estimates for a virtual element spectral analysis for the elasticity
  equations}, IMA J. Numer. Anal., 40 (2020), pp.~322--357,
  \url{https://doi.org/10.1093/imanum/dry063}.

\bibitem{MR3340705}
{\sc D.~Mora, G.~Rivera, and R.~Rodr\'{\i}guez}, {\em A virtual element method
  for the {S}teklov eigenvalue problem}, Math. Models Methods Appl. Sci., 25
  (2015), pp.~1421--1445, \url{https://doi.org/10.1142/S0218202515500372}.

\bibitem{MR3715326}
{\sc D.~Mora, G.~Rivera, and R.~Rodr\'{\i}guez}, {\em A posteriori error
  estimates for a virtual element method for the {S}teklov eigenvalue problem},
  Comput. Math. Appl., 74 (2017), pp.~2172--2190,
  \url{https://doi.org/10.1016/j.camwa.2017.05.016}.

\bibitem{MR3895875}
{\sc D.~Mora and I.~Vel\'{a}squez}, {\em A virtual element method for the
  transmission eigenvalue problem}, Math. Models Methods Appl. Sci., 28 (2018),
  pp.~2803--2831, \url{https://doi.org/10.1142/S0218202518500616}.

\bibitem{Mora2021}
{\sc D.~Mora and I.~Vel\'{a}squez}, {\em A {$C^1-C^0$} conforming virtual
  element discretization for the transmission eigenvalue problem}, Res. Math.
  Sci., 8 (2021), pp.~Paper No. 56, 21,
  \url{https://doi.org/10.1007/s40687-021-00291-2}.

\bibitem{Mora2021A2425}
{\sc D.~Mora and I.~Velásquez}, {\em Virtual elements for the transmission
  eigenvalue problem on polytopal meshes}, SIAM Journal on Scientific
  Computing, 43 (2021), p.~A2425 – A2447,
  \url{https://doi.org/10.1137/20M1347887}.

\bibitem{MR3059294}
{\sc R.~Verf\"{u}rth}, {\em A posteriori error estimation techniques for finite
  element methods}, Numerical Mathematics and Scientific Computation, Oxford
  University Press, Oxford, 2013,
  \url{https://doi.org/10.1093/acprof:oso/9780199679423.001.0001}.

\bibitem{MR4497827}
{\sc G.~Wang, J.~Meng, Y.~Wang, and L.~Mei}, {\em {\it {A} priori} and {\it a
  posteriori error} estimates for a virtual element method for the
  non-self-adjoint {S}teklov eigenvalue problem}, IMA J. Numer. Anal., 42
  (2022), pp.~3675--3710, \url{https://doi.org/10.1093/imanum/drab079}.

\bibitem{MR3212379}
{\sc Y.~Yang, L.~Sun, H.~Bi, and H.~Li}, {\em A note on the residual type a
  posteriori error estimates for finite element eigenpairs of nonsymmetric
  elliptic eigenvalue problems}, Appl. Numer. Math., 82 (2014), pp.~51--67,
  \url{https://doi.org/10.1016/j.apnum.2014.02.015}.

\end{thebibliography}
\end{document}